\documentclass[11pt]{article}
\usepackage[letterpaper,margin=0.9in]{geometry}
\usepackage[T1]{fontenc}
\usepackage{lmodern,amsmath,amssymb,microtype}
\usepackage{amsthm,xcolor}
\usepackage{tcolorbox}
\usepackage{enumerate}

\newtheorem{theorem}{Theorem}[section]
\newtheorem{lemma}[theorem]{Lemma}

\newtheorem{fact}[theorem]{Lemma}

\newtheorem{corollary}[theorem]{Corollary}
\newtheorem{definition}[theorem]{Definition}
\newtheorem{observation}[theorem]{Observation}
\newtheorem{remark}[theorem]{Remark}

\numberwithin{equation}{section}
\usepackage{algorithm}
\usepackage[noend]{algpseudocode}

\usepackage[colorlinks=true,linkcolor=blue,citecolor=blue,urlcolor=blue]{hyperref}

\newcommand{\ES}{\mathrm{SW}}
\DeclareMathOperator{\Var}{Var}
\DeclareMathOperator{\Ent}{Ent}
\DeclareMathOperator{\KL}{KL}
\DeclareMathOperator{\gap}{gap}
\newcommand{\E}{\mathbb E}
\newcommand{\PP}{\mathbb P}
\newcommand{\one}{\mathbf1}

\newcommand{\cE}{\mathcal E}

\newcommand{\wLong}{W_{\rm long}}

\DeclareMathOperator{\exc}{exc}

\newcommand{\cC}{\mathcal C}

\title{Rapid mixing of Potts Glauber dynamics throughout uniqueness}

\definecolor{reza}{RGB}{0,102,204}      
\definecolor{bucky}{RGB}{204,102,0}     
\definecolor{xusheng}{RGB}{0,153,0}     
\definecolor{antonio}{RGB}{153,0,153}   

 \author{Antonio Blanca \thanks{Department of CSE, Pennsylvania State University, ablanca@cse.psu.edu.
 } \and Reza Gheissari \thanks{Department of Mathematics, Northwestern University, gheissari@northwestern.edu. 
 } \and Heehyun Park \thanks{Department of CSE, Pennsylvania State University, hbp5148@psu.edu. 
 } \and Xusheng Zhang \thanks{Department of Computer Science, Durham University, xusheng.zhang@durham.ac.uk. 
 }
 }

\date{}

\begin{document}

\maketitle

\thispagestyle{empty}

\begin{abstract}
Sampling from the ferromagnetic $q$-state Potts model is a fundamental problem in statistical physics and approximate counting. Although computationally hard in the worst case, sampling is expected to be efficient throughout the high-temperature uniqueness regime of the model using the natural single-site Glauber dynamics. We establish fast mixing bounds for every $q$, in the full high-temperature regimes in the two central graph families on $n$ vertices. 
\begin{enumerate}
\item[(1)] For every $\Delta\ge 3$, the Potts Glauber dynamics mixes in $n\operatorname{polylog}(n)$ steps on most 
$\Delta$-regular graphs (i.e., with high probability for a uniformly randomly chosen $\Delta$-regular graph) for all
$\beta<\beta_u(q,\mathbb T_\Delta)$ where $\beta_u(q,\mathbb T_\Delta)$ is the uniqueness threshold on the $\Delta$-regular tree;
\item[(2)] For every $d\ge 2$, the Potts Glauber dynamics has inverse spectral gap $n^{1+o(1)}$ on cubes 
or tori in $\mathbb Z^d$ for all $\beta<\beta_c(q,\mathbb Z^d)$ where $\beta_c(q,\mathbb Z^d)$ is the critical point on $\mathbb Z^d$.
\end{enumerate}
Both temperature thresholds are
sharp as the Glauber dynamics slows down exponentially above them.
These results overcome the challenge of bounding mixing times without access to correlation decay bounds in the presence of pinnings (i.e., without spectral independence or strong spatial mixing).  
For this purpose, our proof introduces a new dichotomy principle: throughout the high-temperature regime, any initial spin distribution is either such that its pinnings ensure good correlation decay, or Glauber dynamics escapes from that distribution quickly. In particular, distributions that generate unfavorable pinnings cannot sustain an obstruction to mixing.
We expect this kind of dichotomy to have further applications. As a demonstration, we also develop its analog for random-cluster dynamics and use it to prove that Swendsen-Wang dynamics on random $\Delta$-regular graphs mixes in $O(\log n)$ time throughout uniqueness.

The results in this paper were arrived at with the assistance of ChatGPT 5.6 and 6. 
\end{abstract}

\vfill

 \pagebreak

\setcounter{page}{1} 
\section{Introduction}\label{sec:intro}

For an integer $q\ge2$, the ferromagnetic $q$-state Potts model at inverse temperature $\beta>0$ on a graph
$G=(V,E)$ is the probability distribution over spin or color assignments
$\sigma\in\{1,\ldots,q\}^V$ given by
\begin{align}\label{eq:Potts-model}
 \mu_{G,\beta,q}(\sigma)
 =\frac1Z e^{\beta\mathcal M(\sigma)},
\end{align}
where $\mathcal M(\sigma) = \sum_{uv\in E}\one_{\{\sigma_u=\sigma_v\}}$ denotes the number of monochromatic edges in the configuration $\sigma$ and $Z=\sum_{\sigma\in\{1,\ldots,q\}^V}
 e^{\beta\mathcal M(\sigma)}$ is the normalizing partition function. 

The $q=2$ case is the famous Ising model of ferromagnetism.
We consider the important computational task of producing (approximate) samples from $\mu=\mu_{G,\beta,q}$ efficiently.
For $q=2$, for any~$\beta>0$,
approximate samples can be generated efficiently in polynomial time~\cite{JSIsing,RandallWilson,GuoJer}, with recent developments
bringing the best running time down to~$\widetilde O(n^2)$ on $n$-vertex graphs~\cite{FengGuoZhang2026}.
In contrast, when $q\ge3$, approximate sampling is
$\#\mathrm{BIS}$-hard, even on graphs of bounded degree~\cite{GoldbergJerrum,GalanisStefankovicVigodaYang2016}.

Despite this hardness result, for general $q\ge3$, approximate sampling from $\mu$ is predicted to be tractable on graphs of maximum degree at most $\Delta$ whenever $\beta$ is below the so-called tree uniqueness threshold. Writing $\mathbb T_\Delta$ for the infinite $\Delta$-regular tree, this threshold is given by
\[
\beta_u(q,\mathbb T_\Delta)
=\log\Big(1+\inf_{x>1}
\frac{(x-1)(x^{\Delta-1}+q-1)}{x^{\Delta-1}-x}\Big)\,.
\]
In the regime $\beta<\beta_u(q,\mathbb T_{\Delta})$, on $\Delta$-regular trees, the influence of spins at depth $r$ on the spin at the root vanishes as $r\to\infty$. It is natural to hope, as we discuss shortly, that even the standard local Markov chain for $\mu$, the Glauber dynamics, mixes rapidly on all graphs of maximum degree at most $\Delta$ for $\beta<\beta_u(q,\mathbb T_\Delta)$. This was established for the Ising model ($q=2$) in \cite{MS13} but has remained open for $q\ge3$.

The Glauber dynamics for the Potts model is the Markov chain that at step~$t$, from configuration~$\sigma(t)$, 
generates~$\sigma(t+1)$ as follows:
\begin{enumerate}
\item A vertex~$v\in V$ is selected uniformly at random.
\item Its spin is resampled according to
$\mu(\sigma_v= \cdot\mid \sigma_{-v}(t))$, with all other spins unchanged.
\end{enumerate}
Here, $\sigma_{-v}(t)$ denotes the configuration of $\sigma(t)$ on $V\setminus\{v\}$.
This process is reversible with respect to~$\mu$ and thus converges to it.
The Glauber dynamics is both the standard Markov chain Monte Carlo approach to sampling from~$\mu$, and the go-to process used in statistical physics to study the thermodynamic evolutions of models of ferromagnetism~\cite{Glauber}.

We study the average-case version of the aforementioned problem: 
for a typical graph of maximum degree~$\Delta$, does the Potts Glauber dynamics mix
quickly for all~$\beta<\beta_u(q,\mathbb T_\Delta)$? Typical here is
understood to mean a uniformly random simple~$\Delta$-regular graph on
$n$ vertices. Note that the threshold $\beta_u(q,\mathbb T_\Delta)$ would be sharp as the Glauber dynamics is slow on typical $\Delta$-regular graphs for~$q\ge3$ and~$\beta>\beta_u(q,\mathbb T_\Delta)$, where the~$q$ ordered phases, each corresponding to one
color occupying a strictly larger fraction of vertices than the others,
become metastable, and the dynamics initialized in one of these phases remains there for exponentially many steps with high probability~\cite{COGGRSV-Metastability-Potts-RRG}.

A polynomial-time algorithm for sampling from the Potts
model throughout the tree uniqueness regime on random~$\Delta$-regular
graphs was first given in~\cite{BGGSVY20}. Sampling was subsequently achieved in~$\widetilde O(n)$ time using a different Markov chain,
the random-cluster dynamics~\cite{BG21}. That chain updates an edge representation of the model rather than individual spins;
its steps depend on global connectivity properties of the configuration, and implementing it efficiently requires sophisticated dynamic graph connectivity data structures. 
It remained notably open 
whether the simplest and most natural sampling algorithm, the Potts Glauber dynamics, mixes in polynomial time for all $q\ge 3$ and~$\beta<\beta_u(q,\mathbb T_\Delta)$,
a regime with no obvious source of metastability or other slowdowns. Our Theorem~\ref{mainthm:Glauber} resolves this question, establishing a mixing time of $\tilde O(n)$ steps.

Let us discuss the challenges posed by this problem. On general graphs of maximum degree $\Delta$, the best known bounds on the mixing time  of Potts Glauber dynamics go by contractive couplings at sufficiently high temperatures, under Dobrushin-type conditions, or optimizations thereof for large $q$ and $\Delta$ as in~\cite{BGP14}, but these do not get up to the conjectured $\beta_u(q,\mathbb T_\Delta)$. In the $q=2$ case of the Ising model, rapid mixing on all graphs of maximum degree $\Delta$ throughout tree uniqueness was established in~\cite{MS13}. Their argument fundamentally relies on monotonicity to localize the dynamics and Weitz's self-avoiding walk tree reduction~\cite{Weitz} to transfer correlation decay estimates from trees to general graphs. 

More modern approaches like spectral independence~\cite{AnariLiuOveisGharan2020} and localization schemes~\cite{ChenEldan}, have led to sharp mixing results for more general two-spin systems where monotonicity is unavailable (e.g., the hardcore model and spin glasses)~\cite{AnariLiuOveisGharan2020,chenchenchen,BaBoDa-Kawasaki,SKWang,SK,OpenAISK}. Notably, these methods rely instead on correlation bounds that hold uniformly over pinnings, that is, after fixing a subset of vertices to specified spins. For Potts, however, this uniform control will fail even within the uniqueness regime. 
Indeed, there are parameters of $q,\Delta$ and $\beta<\beta_u(q,\mathbb T_\Delta)$ at which 
the spectral independence constant on most graphs is polynomially large in $n$, ruling out the route from $O(1)$-spectral independence to rapid mixing; see Remark~\ref{rem:si}. 

The issues are compounded by the fact that the standard tools available for establishing correlation decay bounds under pinnings do not extend from two-spin systems to multi-spin systems.
For instance, there are fundamental obstacles to extending Weitz’s reduction to the ferromagnetic Potts model~\cite{NoWeitzForPotts}, and, likewise, the Ding-Song-Sun inequality relating correlation decay with and without pinnings in the Ising model~\cite{DingSongSun} does not apply in the Potts case.

Our task thus belongs to the broader program of obtaining efficient Markov chain mixing rates without uniform correlation control under conditioning. Recently, correlation decay estimates under typical (rather than worst-case) pinnings have been leveraged to give guarantees on the mixing time of simulated annealing in random-field Ising models and spin glasses~\cite{RandomFieldIsing,WeakPoincareInequalities,SKsimulatedAnnealing,du2026samplingsherringtonkirkpatrickmodelcriticality}. These results illustrate the possibilities offered by weaker correlation hypotheses, but do not easily extend to the Potts model nor give worst-case mixing time bounds.

Our first main result is fast mixing of Potts Glauber dynamics throughout
the tree uniqueness regime on almost all~$\Delta$-regular graphs,
in the sense of holding with high probability for a random
$\Delta$-regular graph.
Recall that the mixing time of a discrete-time Markov chain with
transition matrix~$P$ on a finite state space~$\Omega$, reversible with
respect to a stationary distribution~$\mu$, is defined as
\[
 t_{\rm mix}
 =\min\big\{t\ge0:
 \max_{\sigma\in\Omega}
 \|P^t(\sigma,\cdot)-\mu\|_{\rm TV}\le\tfrac14\big\},
\]
where~$t$ is an integer and~$\|\cdot\|_{\rm TV}$ denotes total variation
distance.

\begin{theorem}\label{mainthm:Glauber}
Fix integers~$\Delta\ge3$ and $q\ge2$, and let $0<\beta<\beta_u(q,\mathbb T_\Delta)$.
With probability~$1-o(1)$, the $q$-state Potts Glauber dynamics at inverse temperature~$\beta$ on a uniformly random $\Delta$-regular graph~$G$ on~$n$ vertices has mixing time $O(n(\log n)^\kappa)$, for some constant~$\kappa=\kappa(\Delta,q,\beta)>0$.
\end{theorem}

Considering the challenges described above Theorem~\ref{mainthm:Glauber}, the key mathematical contribution of our work is a new method to prove
fast mixing without control of correlation decay under pinnings,
i.e., using only weak-spatial-mixing-type control. We establish a
dichotomy between good pinnings and fast escape of the Markov chain,
which we call the
\emph{shattering--escape dichotomy} under uniqueness. 

\smallskip
\begin{tcolorbox}[
  colback=white,
  colframe=black,
  boxrule=0.5pt,
  arc=0pt
]
\noindent \textbf{The shattering-escape dichotomy}.
For every distribution $\nu$ over spin configurations, 
\begin{itemize}
\item (Shattering) 
Either pinnings drawn from $\nu$ induce sufficiently nice pinnings to have correlation decay in their presence;
\item (Escape) Or $\nu$ is easy
to escape from; that is, it cannot be sufficiently stable under Glauber updates to create a bottleneck for fast mixing. 
\end{itemize}
\end{tcolorbox}

We describe this dichotomy in more detail in our proof overview in Section~\ref{subsec:proof:ideas}, but demonstrate 
the strength of this new framework by giving another result that we deduce from the same dichotomy.

Like the Ising model, the Potts model originates in the statistical
physics literature, where it is a prototypical spin system. In that
context, it is of most significance on integer lattice graphs
$(\mathbb Z^d)_{d\ge2}$.
The natural target from this perspective is then rapid mixing throughout the high-temperature regime $\beta<\beta_c(q,\mathbb Z^d)$, a regime that extends beyond tree uniqueness; i.e., $\beta_c(q,\mathbb Z^d) > \beta_u(q,\mathbb T_{2d})$ where the comparison is to $\mathbb T_{2d}$ since~$2d$ is the vertex degree on~$\mathbb Z^d$.
Martinelli and Olivieri~\cite{MaOl2} showed that 
strong spatial mixing (i.e., exponential decay of correlations even in the presence of arbitrary nearby pinnings) implies rapid mixing of Glauber dynamics on subsets of $\mathbb Z^d$. Using~\cite{DingSongSun} to get strong spatial mixing for all $\beta<\beta_c(2,\mathbb Z^d)$, the $q=2$ Ising dynamics is thus fast mixing at high temperatures in all dimensions. Additionally, on~$\mathbb Z^2$, for general $q\ge 2$, 
the requisite strong spatial mixing on boxes holds for all
$\beta<\beta_c(q,\mathbb Z^2)$~\cite{WeakSpatialMixingImpliesStrong}. 
However, for~$q\ge3$ in dimensions~$d\ge3$, there is no corresponding deduction of strong spatial mixing throughout~$\beta<\beta_c(q,\mathbb Z^d)$. 

This left Potts Glauber mixing throughout the high-temperature regime
on~$\mathbb Z^d$, for~$q\ge3$ and~$d\ge3$, as an open challenge 
of a similar spirit to the random regular graph case.
Using a shattering--escape
dichotomy on the lattice, 
exploiting the subexponential volume growth of the lattice in place of the local tree structure of random graphs, 
we prove for all~$d\ge3$ that the inverse gap of high-temperature Potts Glauber dynamics on cubes and tori on $n$ vertices in $\mathbb Z^d$ is~$n^{1+o(1)}$.
\begin{theorem}\label{mainthm:Glauber-lattice}
Fix integers~$d\ge 2$, $q\ge 2$, and
$0<\beta<\beta_c(q,\mathbb Z^d)$.
The Potts Glauber dynamics on a cube in $\mathbb Z^d$ on $n$ vertices has inverse spectral gap $n^{1+o(1)}$ and therefore mixing
time at most~$n^{2+o(1)}$. The same holds for the $d$-dimensional torus graph on $n$ vertices. 
\end{theorem}

We expect the optimal mixing time to be~$O(n\log n)$; the extra factor of~$n$ in our bound comes from using a spectral gap estimate. A tight bound on the log-Sobolev constant instead would yield this improvement, but requires an additional ingredient not yet available. Specifically, the results for the related random-cluster dynamics in~\cite{gheissari2024spatial} provide the spectral gap input we need, but not the corresponding log-Sobolev bound. We do not pursue this improvement, as it is not the main point of the paper.
As in Theorem~\ref{mainthm:Glauber},
the threshold $\beta_c(q,\mathbb Z^d)$ in Theorem~\ref{mainthm:Glauber-lattice} is sharp; the mixing time slows down exponentially in $n^{(d-1)/d}$ either at or immediately above $\beta_c(q,\mathbb Z^d)$ depending on the value of $q$ and dimension~\cite{BorgsChayesTetali,GheissariLubetzkyCritical}.

Finally, to demonstrate the potential of shattering-escape dichotomies for other Markov chains, we use an edge version
of the dichotomy to get the optimal $O(\log n)$ mixing bound for Swendsen--Wang dynamics
on typical~$\Delta$-regular graphs throughout the tree uniqueness
regime.
Swendsen--Wang~\cite{SW} dynamics is the Markov chain that from a configuration~$\sigma(t)$, generates~$\sigma(t+1)$ as follows:
\begin{enumerate}
\item Percolation: independently for each edge $e= xy$, include $e$ in $\omega_t$ with probability $1-e^{-\beta}$ if $\sigma_x(t) = \sigma_y(t)$ (do not include it if $\sigma_x(t)\ne \sigma_y(t)$); 
\item Recoloring: independently color each connected component
$\mathcal C$ of~$(V,\omega_t)$ uniformly from $\{1,...,q\}$, i.e., draw
$a_{\mathcal C}\sim\operatorname{Unif}\{1,\ldots,q\}$ and set
$\sigma_v(t+1)=a_{\mathcal C}$ for every~$v\in\mathcal C$.
\end{enumerate}
This Markov chain leverages the Edwards--Sokal coupling of the Potts and
random-cluster models~\cite{ES}. Namely, if $\sigma\sim \mu$, then after the percolation step, the resulting edge configuration $\omega$ will be distributed as the random-cluster model with parameters $p=1-e^{-\beta}$ and $q$; equivalently, the random-cluster model is the probability distribution $\phi$ over $\omega \subseteq E$ with $\phi_{G, p, q}(\omega) \propto (\frac{p}{1-p})^{|\omega|} q^{c(V,\omega)}$ where $c(V,\omega)$ is the number of connected components of the graph $(V,\omega)$. 

The Swendsen--Wang dynamics is widely used for sampling from the Ising and Potts
models, as its global recolorings allow it to circumvent low-temperature
bottlenecks that slow down Glauber dynamics. Still, for the~$q\ge3$
Potts model, tree uniqueness remains a sharp boundary of the
high-temperature fast-mixing regime: on random~$\Delta$-regular graphs,
the Swendsen--Wang dynamics is exponentially slow in a metastability window $\beta \in (\beta_u(q,\mathbb T_\Delta),
                 \beta_s(q,\mathbb T_\Delta))$ where $\beta_s(q,\mathbb T_\Delta)
 =\log(1+\frac q{\Delta-2})$~\cite{COGGRSV-Metastability-Potts-RRG}. 
 Our next result establishes $O(\log n)$ mixing for the Swendsen--Wang dynamics throughout the tree uniqueness regime on random $\Delta$-regular graphs.

\begin{theorem}\label{mainthm:Swendsen-Wang}
Fix integers~$\Delta\ge3$, $q\ge2$, and
$0<\beta<\beta_u(q,\mathbb T_\Delta)$.
With probability~$1-o(1)$, a uniformly random~$\Delta$-regular
graph~$G$ on~$n$ vertices is such that the mixing time of the
Swendsen--Wang dynamics at parameters $q$ and $\beta$ on~$G$ is~$O(\log n)$. 
\end{theorem}

The best previously known mixing time bound was~$O(n^2\log n)$~\cite{BG21} and was obtained by combining rapid mixing of the random-cluster  dynamics with the comparison results of~\cite{Ullrich-SW-singleBond}.
It has generally been a hard task to
obtain the sharp mixing rate of Swendsen--Wang dynamics. See~\cite{galanis2019swendsen} for results of
this flavor for the Potts model on
the complete graph,
~\cite{BlancaEtAl2022} for results on
~$\mathbb Z^d$ in the regime where strong spatial mixing holds, and~\cite{SwendsenWangTrees} for results
on trees.  

Along the way to proving Theorem~\ref{mainthm:Swendsen-Wang},
we establish a sharp general comparison theorem between the Swendsen--Wang and Potts Glauber dynamics. We show that 
on every bounded-degree graph on $n$ vertices, the inverse spectral gap of the Swendsen--Wang dynamics is at most that of Glauber dynamics divided by an $O(n)$ factor; this resolves a conjecture of Ullrich~\cite[Remark~3.2]{Ullrich2}, improving the previous general comparison by the optimal $O(1/n)$ factor. 
We further establish the analogous comparison for the inverse modified log-Sobolev constants, which is in fact a stronger result; see Corollary~\ref{cor:rr-sw-relaxation}. Combined with previously known results for the Glauber dynamics, these new comparison inequalities recover several 
previous relaxation and mixing bounds for Swendsen--Wang dynamics on trees, lattices, and general graphs of bounded degree~\cite{BCSV,BlancaEtAl2022,SwendsenWangTrees,BCCPSV22}.

We expect that shattering--escape dichotomies may have further implications for proving fast mixing of spin system Markov chains in settings where sharp thresholds for weak spatial mixing are known, but strong spatial mixing or spectral independence do not extend up to those thresholds.

\begin{remark}
\label{rem:si}
We provide an explicit counterexample to $O(1)$-spectral independence holding throughout uniqueness in either trees or typical $\Delta$-regular random graphs. 
In the counterexample, set $\Delta=4$, $q=7735$, and
$\beta=\log36<\beta_u(q,\mathbb T_4)$.
Start with a complete binary tree \emph{core} of height $h$, attach one leaf child 
pinned to color $1$ to every core vertex, and two free
children to every depth-$h$ vertex of the core; finally, give each of those newly added free children two children pinned to color $1$ and one child that is an unpinned leaf.
The pinning structure and parameters are chosen carefully so that 
upward messages in the tree recursion are homogeneous throughout
the core, which allows for a direct computation that shows the largest eigenvalue of the influence matrix is $\Omega(\rho^h)$ for a fixed $\rho>1$.

We embed this example in a $4$-regular random graph by taking $h= c \log n$ for a small constant $c$, and note that a random regular graph has many balls of radius $2h$ that are complete $4$-regular trees of depth $2h$. Let $B_{2h}(o)$ be one such ball centered at $o$, and let $T_o$ be the above tree construction with root $o$. By naturally identifying $T_o$ with a subgraph of $B_{h+2}(o)$, it remains to describe any additional pinnings on the subtrees of $B_{2h}(o) \setminus B_{h+2}(o)$ that emanate from vertices of $T_o$ with degree less than $4$ (i.e., its root and the unpinned leaves). By uniqueness, even if we pin all vertices of $\partial B_{2h}(o)$ to $1$, the messages incoming to those exposed vertices of $T_o$ will be independent and nearly neutral (indeed after a calculation at the parameter values chosen above, sufficiently neutral) that one preserves a $(1+\epsilon)^h$ eigenvalue lower bound for the influence matrix within $T_o$. This therefore provides the desired $n^{\epsilon}$ lower bound on the largest eigenvalue of the full influence matrix. 
\end{remark}

\subsection{Proof ideas}
\label{subsec:proof:ideas}

We next outline the main ideas behind our approach and highlight its key innovations.
 We begin with describing how the proof of Theorem~\ref{mainthm:Glauber} would go for just getting $t_{\rm mix} = \text{poly}(n)$ (note that no subexponential bound was previously known). This is simpler to describe than that for the near-optimal $\tilde O(n)$ of Theorem~\ref{mainthm:Glauber} but already contains the essential novelty of our arguments.
 
\subsubsection{Warm-up: a polynomial bound using the shattering-escape dichotomy}\label{sec:poly-warmup}

Let $P_{\rm Gl}$ and $P_{\rm SW}$ be the transition matrices of the Potts Glauber and Swendsen--Wang dynamics, respectively.
When $P$ has nonnegative spectrum, we write $\operatorname{gap}(P)=1-\lambda_2(P)$ for the spectral gap of $P$, where $\lambda_2(P)$ is its second-largest eigenvalue.
The standard mixing-time bound in terms of $\text{gap}$ (e.g.,~\cite[Theorem 12.4]{LP17}) gives
\begin{align}\label{eq:tmix-spectral-gap}
t_{\rm mix}
\le
\frac{1}{\operatorname{gap}(P)}
\log\frac4{\mu_{\min}},
\end{align}
where $\mu_{\min}=\min_\sigma\mu(\sigma)$.
Since $\log\frac4{\mu_{\min}}=O(n)$ in our setting, and the polynomial lower bound $\operatorname{gap}(P_{\rm SW})=\Omega(n^{-3})$ is known from~\cite{BG21}, it suffices to compare the two spectral gaps of 
$P_{\rm Gl}$ and $P_{\rm SW}$ with at most a polynomial loss.

For the cleanest exposition, we first compare the conductances of these two chains. Define
\[
\Phi_{P}=\min_{A:0<\mu(A)\le1/2}\Phi_{P}(A), \qquad
\Phi_P(A)=\frac{Q_P(A,A^c)}{\mu(A)}
\qquad \text{where} \qquad
Q_P(A,A^c)=\sum_{\sigma\in A,\sigma'\notin A}\mu(\sigma)P(\sigma,\sigma').
\]
Cheeger's inequality, i.e., $\frac{\Phi_P^2}{2}\le\operatorname{gap}(P)\le2\Phi_P$,
then allows us to transfer the conductance comparison back to the spectral gaps~\cite[Theorem 13.10]{LP17}.

For any set $A$ with $\mu(A) > 0$, its conductance $\Phi_{P}(A)$ is the probability of escaping the set $A$ in one step of $P$ when the initial configuration is drawn from $\mu(\cdot\mid A)$. 
We compare these escape probabilities by decomposing Swendsen--Wang updates into sequences of single-site Glauber updates. Such a comparison can incur an exponential in $n$ loss in general, but by recoloring the random-cluster components in parallel, 
we observe that the loss can be at most exponential in the size of the largest component (up to polynomial factors in $n$). Thus, for Swendsen--Wang steps whose intermediate edge configurations have all components of size $O(\log n)$, the comparison incurs only a polynomial loss. 

Let $\mathrm{Good}_L$ be the set of random-cluster configurations $\omega \in \{0,1\}^E$
where every connected component of $(V,\omega)$ has at most $L$ vertices and write $\mathrm{Bad}_L=\mathrm{Good}_L^c$.
To restrict the comparison to these configurations, 
we consider an auxiliary Markov chain $\widehat P_{\rm SW}$ that performs the proposed Swendsen--Wang {move} only
when the intermediate edge configuration belongs to $\mathrm{Good}_L$, 
and otherwise stays at the current configuration.
This chain is reversible with respect to $\mu$. 
By comparing each component recoloring in a Swendsen--Wang step with a prescribed sequence of single-site Glauber updates that realizes the same recoloring, we obtain
\begin{equation*}
Q_{\widehat P_{\rm SW}}(A,A^c)
\le n e^{O(L)} Q_{P_{\rm Gl}}(A,A^c),
\end{equation*}
and for any set $A$ such that $
0<\mu(A)\le \tfrac12
$, we therefore have
\begin{equation}
    \label{eq:SW-flow-comparison}
\Phi_{P_{\rm Gl}}(A)\ge \frac{1}{n} e^{ - O(L)} \Phi_{\widehat P_{\rm SW}}(A).
\end{equation}
To move back to $P_{\rm SW}$, 
for a spin distribution $\nu$ on $[q]^V$,  let $\mathbb P_{\nu,{\rm SW}}(\omega \in \cdot )$
denote the distribution on random-cluster configurations obtained by first drawing $\sigma$ from $\nu$ 
and then applying the percolation step of the Swendsen--Wang dynamics.
Rejecting updates in $\widehat P_{\rm SW}$ whose edge configurations belong to $\mathrm{Bad}_L$ reduces the escape probability by at most $\mathbb P_{\nu,{\rm SW}}(\mathrm{Bad}_L)$. Consequently,
\begin{align}\label{eq:bottleneck-PhatSW-to-P}
 \Phi_{\widehat P_{\rm SW}}(A)
 \ge\Phi_{P_{\rm SW}}(A) - \mathbb P_{\mu(\cdot \mid A), \rm SW} (\text{Bad}_L)\,.
\end{align}

The challenge is that for $L=O(\log n)$, $\mathbb P_{\mu(\cdot \mid A),\rm SW}(\text{Bad}_L)$ may be large for some sets~$A$ even though $\beta<\beta_u(q,\mathbb T_\Delta)$.
For instance, from spin configurations with large monochromatic regions, even in uniqueness, the percolation step of the Swendsen--Wang dynamics may not be subcritical. 
Our dichotomy shows that sets of spin configurations producing such unfavorable edge configurations cannot also have small Glauber escape probabilities. To make this precise, we introduce a quantity measuring how far the single-site conditional distributions of $\nu$ are from those of $\mu$.
\begin{definition}\label{def:local-relent}
For any spin distribution $\nu$, define its \emph{aggregate local entropy} to stationarity as
$$
 I_\mu(\nu)=\sum_v\E_{\nu_{-v}}
 [\KL\!\left(\nu_v(\cdot\mid \sigma_{-v})\,\Vert\,
                    \mu_v(\cdot\mid \sigma_{-v})\right)]\,.
$$
where ${\rm KL}(\cdot \,\|\,\cdot)$ is the Kullback--Leibler divergence, or relative entropy, $\mu_v,\nu_v$ are 
the $v$ marginals, and $\nu_{-v}$ denotes the marginal of configurations on $V\setminus \{v\}$. 
\end{definition}
The aggregate local entropy estimates how far $\nu$ is from $\mu$ as a sum of single-site relative entropies, 
rather than a more global notion like $\rm{KL}(\nu\, \|\, \mu)$. 
The shattering-escape dichotomy will say that for any spin distribution $\nu = \mu(\cdot \mid A)$, large $I_\mu(\nu)$ forces correspondingly large conductance for $A$, and therefore rapid escape from $A$, while small $I_\mu(\nu)$ makes $\text{Bad}_L$ unlikely under $\mathbb P_{\nu,\rm SW}$, rendering~\eqref{eq:bottleneck-PhatSW-to-P} useful.
 
The first implication is formalized in the following lemma, which has a short proof because $I_\mu(\nu)$ is a sum of single-site resampling distances---essentially the Dirichlet form of Glauber dynamics. For more grounded intuition for $I_\mu(\nu)$ implying rapid escape from $A$, observe that by convexity of $\KL(\cdot \,\|\,\mu)$, 
\begin{align}
    \KL(\nu\,\|\,\mu) - \KL (P_{\rm Gl}\nu \,\|\,\mu) \ge \frac{1}{n} I_\mu(\nu)
\end{align}
so $I_\mu(\nu)/n$ bounds from below the relative entropy decay in one Glauber step started from $\nu$.

\begin{lemma}\label{lem:cut-information} There exists $C>0$ such that for all $A \subseteq  [q]^V$, we have 
$
 I_\mu(\mu(\cdot \mid A))\le Cn\Phi_{P_{\rm Gl}}(A).
$
\end{lemma}

The next lemma gives the complementary implication:
any distribution with sufficiently small aggregate local entropy relative to $\mu$ induces a shattered random-cluster configuration with high probability after the percolation step of the Swendsen--Wang dynamics.

\begin{lemma}
\label{lem:info:bound}
With probability $1-o(1)$, a random $\Delta$-regular graph $G$ is such that the following holds. Suppose $\beta<\beta_u(q,\mathbb T_\Delta)$. There exists $C, C_0 > 0$ such that, for
$L=C_0\log n$, for every spin distribution~$\nu$,
$$
 I_\mu(\nu)\le n^{-18}
 \quad\Longrightarrow\quad
 \PP_{\nu,{\rm SW}}(\mathrm{Bad}_L)\le Cn^{-5}.
$$
\end{lemma}

Lemma~\ref{lem:info:bound}, together with several  generalizations we prove later, formalizes the shattering-escape dichotomy and lies at the core of our paper. Its proof goes by iteratively exploring the random-cluster components under $\mathbb P_{\nu, \rm SW}$ using a coarsened exploration process where the next edge is revealed conditional on $\sigma(B^c_R(v))$---that is, the spins outside $B_R(v)$, the radius-$R$ ball centered at vertex $v$. Up to a cost that is $e^{O((\Delta-1)^R)} \cdot I_\mu(\nu)$, when doing this exploration, 
we can replace sampling under $\nu(\cdot\mid\sigma(B^c_R(v))$ with sampling under $\mu(\cdot\mid\sigma(B^c_R(v))$, while keeping the exterior spins drawn from $\nu$. By the locally treelike structure of the random graph and tree uniqueness, the influence of these exterior spins on the root vanishes as $R$ grows. Consequently, for sufficiently large $R$, the conditional edge-opening probability is close to its value on an infinite $\Delta$-regular tree, which is strictly below $1/(\Delta-1)$ in the uniqueness regime, so the exploration process is strictly subcritical.
We mention that a reminiscent comparison of conditional laws appears in the two-spin setting in~\cite[Lemma~3.5]{LocallyStationaryDistributions}: locally stationary distributions there have conditional laws that are close, on average, to their equilibrium counterparts.
Lemma~\ref{lem:info:bound} and its strengthenings used for the $\tilde O(n)$ mixing time are proved in Section~\ref{sec:key-stationary-estimate}.

\begin{proof}[\textbf{\emph{Proof of polynomial mixing of Potts Glauber for $\beta<\beta_u$}}]
Fix a graph~$G$ satisfying the conclusion of
Lemma~\ref{lem:info:bound}. 
Let~$A \subseteq [q]^V$ be any subset of configuration with
$0<\mu(A)\le \frac12$, and let~$\nu_A=\mu(\cdot\mid A)$. Let $L = C_0 \log n$ for $C_0$ from Lemma~\ref{lem:info:bound}. 
By~\eqref{eq:SW-flow-comparison}--\eqref{eq:bottleneck-PhatSW-to-P}, and the 
fact from~\cite{BG21} that
$\operatorname{gap}(P_{\rm SW})=\Omega(n^{-3})$ which implies that $\Phi_{\rm SW}(A) =\Omega(n^{-3})$,  we obtain
\begin{align}\label{eq:PhiGA-lower-bound}
    \Phi_{P_{\rm Gl}}(A) \ge \frac{1}{n} e^{ - O(L)}\Big( \frac{1}{n^3} - \mathbb P_{\nu_A, \rm SW}(\text{Bad}_L)\Big)
\end{align}
By Lemma~\ref{lem:cut-information}, if $I_\mu(\nu_A) \ge n^{-18}$, then $\Phi_{P_{\rm Gl}}(A) \ge \frac{1}{C} n^{-19}$. On the other hand, if $I_\mu(\nu_A)\le n^{-18}$, by Lemma~\ref{lem:info:bound}, $\mathbb P_{\nu_A,\rm SW}(\text{Bad}_L) = O(n^{-5})$, and therefore~\eqref{eq:PhiGA-lower-bound} gives for large $n$ that 
\begin{align*}
    \Phi_{{P_{\rm Gl}}}(A) \ge \min\Big\{ \frac{1}{C} n^{-19}, \frac{1}{2n^4} e^{ - O(L)}\Big\}\,.
\end{align*}
Plugging in $L = C_0 \log n$, this gives a polynomial lower bound on $\Phi_{P_{\rm Gl}}(A)$ and by Cheeger's inequality and~\eqref{eq:tmix-spectral-gap} the mixing time of the Glauber dynamics is polynomial in $n$. 
\end{proof}

\subsubsection{Improvements to get $\tilde O(n)$ mixing}

There are a few obstructions to obtaining $\tilde O(n)$ bounds on the inverse spectral gap and mixing time above. The first is that the comparison factor for recoloring a component of size $L$ monochromatically using single-site updates is exponential in $L$, so even with $L=O(\log n)$, we incur a polynomial loss. Further obstructions are the suboptimal Swendsen--Wang mixing estimate available and the use of conductance and Cheeger's inequality, rather than directly establishing Poincar\'e or entropy inequalities.

Perhaps the most significant innovation to deal with the above obstructions is working not with comparison to Swendsen--Wang dynamics, but with a new global chain $M$.  
The idea behind $M$ is to soften the hard constraints imposed by the Swendsen--Wang percolation step while preserving the stationary measure $\mu$. In Swendsen--Wang, observing an open edge tells us that its endpoints must have the same spin. For $M$, we add independent edges to the configuration $\omega$ that results from the percolation step to get a subgraph $H$. An open edge in $H$ now provides evidence that its endpoints agree, but does not force agreement. This turns the hard constraints into finite ferromagnetic interactions. In fact, conditional on the edges added in both steps, the spins follow a Potts distribution at a finite inverse temperature $\tilde\beta$, which tensorizes across the connected components.

\begin{definition}\label{def:M-SW-chain}
    For a parameter $r\in [0,1]$, the Markov chain $M^{(r)}$ on Potts configurations goes from a spin configuration $\sigma(t)$ to $\sigma(t+1)$ as follows: 
    \begin{enumerate}
        \item Percolation: independently for each edge $e= xy$, include $e$ in $\omega_t$ with probability $1-e^{-\beta}$ if $\sigma_x(t) = \sigma_y(t)$ (do not include it if $\sigma_x(t) \ne \sigma_y(t)$). 
        \item Additional edge sprinkling: independently for each edge $e$, include $e$ in $\eta_t$ with probability $r$, and let $H_t = \omega_t \cup \eta_t$. 
        \item Recoloring: Independently, for each connected component $\mathcal C= (V(\mathcal C),E(\mathcal C))$ of $(V,H_t)$, draw $\sigma_{V(\mathcal C)}(t+1)$ as a sample from the Potts distribution  $ \mu_{(V(\mathcal C), E(\mathcal C)), \tilde\beta,q}$ for $\tilde\beta(r)= \log (1+\frac{e^\beta -1}{r})$.  
    \end{enumerate}
\end{definition}

The chain $M^{(r)}$ interpolates between Swendsen--Wang dynamics and direct sampling from $\mu_{G,\beta,q}$. At $r=0$, we have $\tilde\beta(r) = \infty$ and the recoloring step becomes the zero-temperature Potts limit where
each connected component receives independently a uniformly random color; that is, $M^{(0)}$ is exactly the
Swendsen--Wang chain. At $r=1$, on the other hand, the edge configuration is deterministically $E$ and $\tilde\beta(1)=\beta$, so each step draws an independent sample from $\mu_{G,\beta,q}$.
More generally, $M^{(r)}$ admits a noising--denoising interpretation in the framework of localization schemes~\cite{ChenEldan,chen2025rapid,chen2026edge}. This representation implies
reversibility with respect to $\mu_{G,\beta,q}$ for every
$r\in[0,1]$. 

The benefit of working with the chain $M^{(r)}$ for $r>0$ is the following. Since
$\beta<\beta_u(q,\mathbb T_\Delta)$, we can choose $r>0$
sufficiently small such that even with the sprinkling step, starting from a stationary $\sigma$,
the resulting edge set still has all components of size at most $L=O(\log n)$ with high probability. At the same time, $r>0$
makes the Potts recoloring step on each connected component
positive-temperature rather than zero-temperature, and we can therefore use the
polynomial bounds on the inverse spectral gap of the Glauber
dynamics on trees and treelike graphs, obtained from cutwidth
bounds, to simulate the component recolorings using single-site Glauber updates at a cost polynomial in $L$.
Finally, because of the extra sprinkling of individual edges in step (2), for any $r>0$, the inverse gap of $M^{(r)}$ is bounded by a constant times the inverse gap of continuous-time random-cluster Glauber dynamics without the loss incurred by the comparison of~\cite{Ullrich-SW-singleBond}. 

The preceding improvements yield an $n\cdot\operatorname{polylog}(n)$ bound on the inverse spectral gap of the Potts Glauber dynamics, but the usual conversion to a mixing time bound (see~\eqref{eq:tmix-spectral-gap}) still incurs an additional factor of $n$. To avoid this loss, we replace variance estimates by entropy estimates and establish a $\tilde\Omega(1/n)$ bound on the log-Sobolev constant of the Glauber dynamics, so the conversion to mixing time costs only a factor of order $\log n$. The entropy argument follows essentially the same strategy, but requires one additional ingredient. Namely, an asymptotically optimal log-Sobolev constant for the random-cluster Glauber dynamics, which does not follow from the mixing time bound in~\cite{BG21}. In Section~\ref{sec:fk-entropy}, we use a random-cluster version of our dichotomy to establish such a bound on random $\Delta$-regular graphs throughout the tree uniqueness regime.

Finally, the polylogarithmic loss in
Theorem~\ref{mainthm:Glauber} relative to the optimal
$O(n\log n)$ mixing time reflects a limitation of our
method. In our comparison over components, some components have logarithmic size and, at inverse temperature $\tilde\beta$, the Glauber dynamics on them can have inverse gaps growing polynomially in their size. 
One might hope to remove this loss by choosing $r$ so that
$\tilde\beta(r)$ lies in the optimal tree-mixing regime
of~\cite{MSW}. However, there is a fundamental tension here:
lowering $\tilde\beta(r)$ requires increasing $r$, which can destroy shattering near the uniqueness threshold.

\subsubsection{Potts mixing on $\mathbb Z^d$ up to criticality}
The argument on the lattice essentially follows the structure outlined for the random regular graph. We highlight the main changes. To obtain the required spectral-gap bound for $M$, we use~\cite{gheissari2024spatial} (see also~\cite{BS,HarelSpinka}) which gives fast mixing of random-cluster dynamics in $\mathbb Z^d$ whenever $\beta<\beta_c(q,\mathbb Z^d)$. 

All other steps in Section~\ref{sec:poly-warmup} carry over, except the essential shattering-escape dichotomy Lemma~\ref{lem:info:bound}. Its lattice replacement  takes the following form. Consider the chain $M^{(r)}$ from Definition~\ref{def:M-SW-chain} with $r=r(q,\beta,d) >0 $ sufficiently small and let $\mathbb P_{\nu,M}$ be the law of the subgraph $H$ after steps (1)--(2) of the chain $M^{(r)}$ with initial spins drawn from $\nu$. As before, let $\mathrm{Bad}_L$ denote the event that $H$
contains a connected component with more than $L$ vertices.
We show that on cubes on $n$ vertices in $\mathbb Z^d$ at $\beta<\beta_c(q,\mathbb Z^d)$, there exist constants $C_0$ and $C$ such that, for
$L=C_0\log n$, for every spin distribution $\nu$,
\begin{align}\label{eq:lattice:shattering-escape-dichotomy}
 I_\mu(\nu)\le n^{-18}
 \quad\Longrightarrow\quad
 \PP_{\nu,\rm{M}}(\mathrm{Bad}_L)\le C n^{-5}\,.
\end{align}

The revealing procedure used to prove~\eqref{eq:lattice:shattering-escape-dichotomy} uses cubes of a sufficiently large but $O(1)$ side length $R$ as coarse-graining blocks. Unlike the random regular graph case, we do not condition on one edge at a time, but on the bulk of a block at a time. Due to the sub-exponential growth of the underlying lattice graph, large $R$ boosts the subcriticality of the underlying random-cluster percolation, so that edge connectivity events under $\mathbb P_{\nu,\rm M}$ are subcritical whenever the aggregate local entropy $I_\mu(\nu)$ is small.

\subsection{AI Disclosure}
The mixing of Potts on random $\Delta$-regular graphs up to $\beta_u(q,\mathbb T_\Delta)$ was a question the authors had been working towards without the use of LLMs for the past three years. With assistance from ChatGPT 5.5--5.6 in Summer 2026, the authors arrived at counterexamples to spectral independence and coupling independence for the Potts model on trees and random graphs in uniqueness. In Fall 2026, we told ChatGPT~6 Astra about the failure of spectral independence, and suggested a block version of spectral independence on trees, and then on random graphs. In response, it discovered the novel dichotomy of Lemma~\ref{lem:info:bound} using a block revealing procedure that broke open the proof of Theorem~\ref{mainthm:Glauber}. The current form of the proofs was arrived at with significant back and forth with Astra. In particular, comparison through the Markov chain $M$ of Definition~\ref{def:M-SW-chain} instead of Swendsen--Wang came from human interactions and simplified arguments that previously were completely functional analytic. 
After seeing Lemma~\ref{lem:info:bound}, the authors also realized that the same dichotomy can be used to get fast mixing on the lattice up to criticality, and better bounds on the Swendsen--Wang dynamics' mixing time (Theorems~\ref{mainthm:Glauber-lattice}--\ref{mainthm:Swendsen-Wang}). 
 No AI was used in preparing Section~\ref{sec:intro}, and all proofs in Sections~\ref{sec:skeleton-of-proof}--\ref{var:sec:torus-extension} underwent rewriting and significant simplification and clarification by the authors, before being compiled into this final version. Both Claude Opus 5.5 and ChatGPT 6 Astra were used in the editing process. The authors take full responsibility for the presentation, clarity, and correctness of the results in this paper. 

\subsection*{Acknowledgments}
The research of A.B.\ and H.P.\ was supported in part by NSF CAREER grant CCF-2143762. The research of R.G.\ is supported in part by NSF CAREER grant DMS-2440509 and NSF DMS grant 2246780. The research of X.Z.\ is supported in part by EPSRC New Investigator Award UKRI155. This work was supported in part by grants of access to OpenAI models through the ChatGPT for Academic Researchers program.

\section{Proof of main theorem: Theorem~\ref{mainthm:Glauber}}\label{sec:skeleton-of-proof}
Here, we give a detailed skeleton of the Potts Glauber mixing proofs:  Theorems~\ref{mainthm:Glauber}--\ref{mainthm:Glauber-lattice}. The proofs mostly mirror each other, except for the method by which the shattering-escape dichotomy is proved. For the sake of presentation, our proof skeleton will have the setting of the random $\Delta$-regular graphs in mind. Given the content of this section, the proof of Theorem~\ref{mainthm:Glauber-lattice} will be rather short; it is given in Section~\ref{var:sec:torus-extension} and can be read without the intervening Sections~\ref{sec:key-stationary-estimate}--\ref{sec:rr-sw}. 
Hence, until Section~\ref{var:sec:torus-extension} we focus on a graph $G$ drawn as a random $\Delta$-regular graph, consider $q$ and $\Delta$ fixed, and let $\beta_u = \beta_u(q,\mathbb T_\Delta)$.  

For the proofs, we will use $d=\Delta-1$, and $B= e^{\beta}$, so that $B_u =  e^{\beta_u}$. Then $$p = \frac{B-1}{B} \qquad \text{and}\qquad \hat p = \frac{B-1}{B+q-1}$$
are the respective probabilities that an edge is open given it is not a cut-edge, or given it is a cut-edge, under the random-cluster distribution $\phi= \mathbb P_{\mu,\rm SW}$. 
We write $\pi_{\ES}(\sigma,\omega)=\mu(\sigma)\,\mathbb P_{\sigma,\rm SW}(\omega)$ for the joint Edwards--Sokal law of spins and edges; its marginals are $\mu$ and $\phi$.
To make  a concrete choice, take $\tilde\theta = \frac{1+d\hat p}{2}$ (though any choice in $(d\hat p,1)$ could work), and let $M = M^{(r)}$ from Definition~\ref{def:M-SW-chain} for 
\begin{align}\label{eq:parameter-defs}
 r = \frac{B-1}{\tilde B-1}\,,\qquad\text{where} \qquad 
 \tilde B=1+\frac{q\tilde\theta}{1-\tilde\theta}\,.
\end{align}
Importantly, if $\beta < \beta_u$, then 
$d\hat p<1$; so as long as $\beta<\beta_u(q, \mathbb T_\Delta)$ these satisfy $B<\tilde B<\infty$ and $r>0$.

\subsection{Near-optimal spectral gap}\label{sec:spectral-gap-rrg}

We begin by giving the proof for an $\tilde O(n)$ inverse spectral gap for the Potts Glauber dynamics up to $B_u$. 
This serves two purposes: it carries out the proof ideas in the simpler setting of variance, and it provides an ingredient for the mixing-time proof, which adapts the argument to the entropy functional.

\begin{theorem}\label{mainthm:Glauber:gap}
Fix integers~$\Delta\ge3$, $q\ge2$, and
$0<\beta<\beta_u(q,\mathbb T_\Delta)$.
With probability~$1-o(1)$, a uniformly random $\Delta$-regular
graph~$G$ on~$n$ vertices is such that the inverse spectral gap of Potts Glauber
dynamics at parameters~$q$ and $\beta$ on~$G$ is at most
$n\operatorname{polylog}(n)$.
\end{theorem}

The below could be carried out with the conductance $\Phi$ instead of the spectral gap, but that would cost a quadratic factor due to the Cheeger's inequality, so we work with the functional form instead. (One recovers the conductance by taking $f=\mathbf{1}_A$ in the below.)  
Recall that for a Markov chain $P$ with stationary distribution $\mu$, for a function $f:[q]^V\to\mathbb R$, 
its Dirichlet form is $\mathcal E_P(f,g) = \frac{1}{2}\sum_{\sigma,\sigma'} \mu(\sigma) P(\sigma,\sigma') (f(\sigma) - f(\sigma'))(g(\sigma) - g(\sigma'))$ and one has the variational form of spectral gap:
$$\gap(P)=\inf_{f:\Var_\mu(f)\neq 0} \frac{\cE_P(f,f)}{\Var_\mu(f)},$$
where $\Var_\mu(f) = \E_\mu[f^2] - \E_\mu[f]^2$ and where $\E_\mu[f]$ denotes the expectation with respect to $\mu$. In what follows, we abbreviate $\mathcal E_P(f) := \mathcal E_P(f,f)$.

Given a current spin distribution  $\nu$, let $\mathbb P_{\nu, \rm M}$ be the distribution over the edge-subset after the two edge steps of the chain $M$.
In particular, for an edge set $H$ under $\PP_{\mu, \rm M}$,
$H\overset{d}{=} \omega \cup \eta$ with $\omega \sim \phi$ and $\eta\sim \text{Ber}(r)$.
We use $\mathbb P_{\sigma,\rm M}$ for $\mathbb P_{\delta_\sigma,\rm M}$. By definition of $M$, we have
\begin{equation}
    \label{eq: reversibility}
     \mu(\sigma) \mathbb P_{\sigma,\rm M}(H) =\mathbb P_{\mu,\rm M}(H)\mu_H(\sigma)\,,
\end{equation}
where we used $\mu_H(\sigma)$ as the law on spins given $H$, which is the $\tilde\beta$-Potts distribution on $(V,H)$.

Our first improvement over the polynomial bound of Section~\ref{sec:poly-warmup} is that $M$ can be shown to have a spectral gap which is comparable up to an $O(|E|)$ factor to that of the random-cluster model's Glauber dynamics, which we henceforth refer to as \emph{FK dynamics}. We recall that FK dynamics is the Markov chain on $\{0,1\}^E$ which from configuration $\omega(t)$, generates $\omega(t+1)$ as follows: 
\begin{enumerate}
    \item Select an edge $e\in E$ uniformly at random; 
    \item If $e$ is a cut-edge in $\omega(t)\cup \{e\}$, draw $\omega_e(t+1) \sim \text{Ber}(\hat p)$; otherwise, draw $\omega_e(t+1)\sim \text{Ber}(p)$. 
\end{enumerate}
Let $P_{\rm FK}$ be its transition matrix, and note that it is reversible with respect to $\phi$. 

\begin{lemma}\label{lem: gap of M}
Let $M = M^{(r)}$ for $r = (B-1)/(\tilde B-1)$. Then 
$
 \gap(M)\ge
 \frac{r}{2\Big(\frac{1}{\gap(P_{\rm FK})\cdot |E|}+1\Big)}.
$
\end{lemma}
Lemma~\ref{lem: gap of M} is proved in Section~\ref{sec:M-FK-comparison}. 
For random-regular graphs, the FK dynamics gap lower bound from
\cite{BG21} then implies that for all $B<B_u$, we have  $\gap(M)\ge a$ for some constant $a(q,\Delta,B)>0$.

\begin{definition}\label{def:M-tilde-chain}
Define $\widetilde M$ as the following modification of $M$: go from $\sigma(t)$ to $\sigma(t+1)$ as follows. 
\begin{enumerate}
    \item Draw $H$ according to $\mathbb P_{\sigma(t),\rm M}$;
\item For every component $\mathcal C=(V(\mathcal C),E(\mathcal C))$ of $H$ with $|V(\mathcal C)|>L$, let $\sigma_{V(\mathcal C)}(t+1)=\sigma_{V(\mathcal C)}(t)$;
\item For every component $\mathcal C=(V(\mathcal C),E(\mathcal C))$ of $H$ with $|V(\mathcal C)|\le L$, draw independently $\sigma_{V(\mathcal C)}(t+1)$ according to $\mu_{_{(V(\mathcal C), E(\mathcal C)), \tilde\beta,q}}$ (the $\tilde B$-Potts model on $V(\mathcal C)$ with edge set $E(\mathcal C)$).
\end{enumerate}    
\end{definition}

Since the recoloring step for the chain $\widetilde M$ is exactly sampling from a positive temperature Potts distribution on its individual components of size at most $L = O(\log n)$, and these components are treelike, the recoloring step of $M$ can be simulated at a $\text{poly}(L)  = \text{polylog}(n)$ factor.

\begin{lemma}\label{lemma:dirichlet:stronger}
Fix $\Delta\ge3$, $q\ge2$, $\tilde B>1$, and $K>0$. There exist
constants
$\kappa,C>0$ such that, with probability $1-o(1)$ the random $\Delta$-regular graph $G$ is such that 
for every $1\le L\le K\log n$ and every $f$,
\begin{equation*}
 \cE_{\rm Gl}(f)\ge\frac{1}{CnL^\kappa}\cE_{\widetilde M}(f).
\end{equation*}
\end{lemma}

\begin{proof}
It is a standard random graph fact (see, e.g.,~Lemma~\ref{lem:poly-geometry}) that, with probability $1-o(1)$, every connected subgraph $F$ of $G$ with $s\le K\log n$ vertices has cycle rank at most a fixed constant $t_K$. Choosing a spanning tree $T$ and using the bound on the cutwidth $\text{cw}$ of a tree gives 
\begin{equation*}
 \operatorname{cw}(F)
 \le \operatorname{cw}(T)+t_K
 \le \Delta\lceil\log_2 s\rceil+t_K\,.
\end{equation*}
The bound of~\cite[Section~2.1]{BKMP05} (which though written for $q=2$ applies without change when $q\ge 3$) on spectral gap in terms of cutwidth of a graph, applied to the $\tilde B$-Potts weights, therefore gives
\begin{equation}\label{eq:tilde-M-component-poincare}
 \begin{aligned}
 \operatorname{Var}_{\mu_F}(g)
 &\le Cs\tilde B^{2\operatorname{cw}(F)}
 \sum_{x\in F}\mu_F\!\left[
   \operatorname{Var}_{\mu_F}
     (g\mid\sigma_{F\setminus\{x\}})
 \right]\le Cs^\kappa
 \sum_{x\in F}\mu_F\!\left[
   \operatorname{Var}_{\mu_F}
     (g\mid\sigma_{F\setminus\{x\}})
 \right].
 \end{aligned}
\end{equation}

Let $\mathcal C(H)$ be the set of all connected components of $H$, and let $U^+(H)$ be the union of the components of $H$ of size at least $L$. Since $\widetilde M$ resamples the spins on only on components of $H\setminus U^+(H)$, independently, tensorization of variance and~\eqref{eq:tilde-M-component-poincare} yield
\begin{align*}
 \cE_{\widetilde M}(f)
 =\E_{\mu,\rm{M}}\big[\mu_H\!\big[
   \operatorname{Var}_{\mu_H}
     (f\mid\sigma_{U^+(H)})
 \big]\big] &\le\E_{\mu,\rm{M}} \Big[
 \sum_{\substack{F\in\mathcal C(H)\\|F|\le L}}
 \mu_H\!\left[
   \operatorname{Var}_{\mu_H}
     (f\mid\sigma_{V\setminus F})
 \right]\Big]\\
 &\le CL^\kappa\sum_{x\in V}
 \E_{\mu,\rm M}\Big[\mu_H\!\left[
   \operatorname{Var}_{\mu_H}
     (f\mid\sigma_{V\setminus\{x\}})
 \right]\Big]
\end{align*}
Using the law of total variance under the joint law $\mathbb P_{\mu,\rm M}(H)\mu_H(\sigma)=\mu(\sigma)\mathbb P_{\sigma,\rm M}(H)$, we see that conditioning on $H$ only decreases the expected conditional variance. Thus, the right-hand side is at most 
$CL^\kappa\sum_{x\in V}
 \E_\mu[
   \operatorname{Var}_{\mu}
     (f\mid\sigma_{V\setminus\{x\}})
 ]
 =CnL^\kappa\cE_{\rm Gl}(f)$ as desired. 
\end{proof}

The following functional analytic bound gives the analogue of Lemma~\ref{lem:cut-information} showing that the aggregate local entropy $I_\mu(\nu)$ lower bounds the Dirichlet form. 

\begin{lemma}\label{lem:function-information}
There is a constant $C>0$, depending only on $q$, $B$ and $\Delta$ such that the following holds.
For every function $f:[q]^V\to\mathbb R$ with $\E_\mu[f^2]>0$, the
probability measure $\nu_f:=f^2\mu/\E_\mu[f^2]$ satisfies
\[
 \E_\mu[f^2]\,I_\mu(\nu_f)\le C\,n\,\cE_{\rm Gl}(f).
\]
\end{lemma}

\begin{proof}
Fix $x\in V$. For every $i\in[q]$,
the conditional distribution $\nu_f(\cdot\mid\sigma_{-x})$ also has density
$f^2/\E_\mu[f^2\mid\sigma_{-x}]$ with respect to $\mu(\cdot\mid\sigma_{-x})$.
Using that and
$\E_{\nu_f}[F]=\E_\mu[f^2F]/\E_\mu[f^2]$, we get
\begin{equation}\label{eq:function-information-calculation}
 \begin{aligned}
 \E_\mu[f^2]\,I_\mu(\nu_f)
 &=\E_\mu[f^2]\sum_{x\in V}\E_{\nu_f}\Bigl[
   \operatorname{KL}\bigl(\nu_f(\cdot\mid\sigma_{-x})\,\big\|\,
   \mu(\cdot\mid\sigma_{-x})\bigr)\Bigr]\\
 &=\E_\mu[f^2]\sum_{x\in V}\E_{\nu_f}\Bigl[\E_{\nu_f}\Bigl[
   \log\frac{f^2}{\E_\mu[f^2\mid\sigma_{-x}]}\Bigm|\sigma_{-x}\Bigr]\Bigr]
  =\E_\mu[f^2]\sum_{x\in V}\E_{\nu_f}\Bigl[
   \log\frac{f^2}{\E_\mu[f^2\mid\sigma_{-x}]}\Bigr]\\
 &=\sum_{x\in V}\E_\mu\Bigl[f^2\log\frac{f^2}{\E_\mu[f^2\mid\sigma_{-x}]}\Bigr]
  =\sum_{x\in V}\E_\mu\bigl[\Ent_\mu(f^2\mid\sigma_{-x})\bigr]
 \end{aligned}
\end{equation}
By the standard entropy-variance comparison, applied to $\mu(\cdot\mid\sigma_{-x})$, which
satisfies $\mu(\sigma_x=i\mid\sigma_{-x})\ge (qB^\Delta)^{-1}$ for all
$i\in[q]$, the right-hand side is at most $C\sum_{x\in V}\E_\mu\bigl[\Var_\mu(f\mid\sigma_{-x})\bigr]
  =C\,n\,\cE_{\rm Gl}(f).$\end{proof}

With the above ingredients in hand, the following analogue of Lemma~\ref{lem:info:bound} (one bounds the probability of $\text{Bad}_L$ by Markov's inequality) is the  essence of the shattering-escape dichotomy. 
Let 
\begin{align}\label{eq:W-L}
 W_L(H)=|U^+(H)| \qquad \text{where} \qquad U^+(H)=\bigcup_{F\in\mathcal C(H):\,|F|>L} F\,.
\end{align}

\begin{lemma}\label{lem:entropy-mass}
Suppose $\beta<\beta_u$. There exist large constants $C_0$, $C'_0 > 0$,  such that for $L= C_0\log n$, with probability $1-o(1)$ for a random $\Delta$-regular graph $G$, for every spin distribution $\nu$,
$$
 \E_{\nu,M}[W_L(H)]\le C'_0 I_\mu(\nu)+n^{-5}.
$$
\end{lemma}
In words, either $I_\mu(\nu)$ is large and by Lemma~\ref{lem:function-information} escape from $\nu$ is fast, or $I_\mu(\nu)$ is small and the lemma is showing a strong form of shattering of the edge-configuration drawn from spin distribution~$\nu$. Lemma~\ref{lem:entropy-mass} is proved via a coarse-grained revealing procedure in Section~\ref{sec:key-stationary-estimate}. 

We are now ready to give the proof of Theorem~\ref{mainthm:Glauber:gap}.

\begin{proof}[Proof of Theorem~\ref{mainthm:Glauber:gap}]
Fix $G$ in the intersection of the events of probability $1-o(1)$ in
Lemmas~\ref{lem: gap of M}, \ref{lemma:dirichlet:stronger} and~\ref{lem:entropy-mass},
and in the FK dynamics gap bound of~\cite{BG21}. Let $L= C_0\log n$, with $C_0$ from
Lemma~\ref{lem:entropy-mass}.
 
Take any real $f$ with $\E_\mu[f^2]>0$. Given $H$, by~\eqref{eq: reversibility}, the chain $M$ redraws all spins from $\mu_H$, while
$\widetilde M$ keeps $\sigma_{U^+(H)}$ and, since $\mu_H$ is a product over the components of $H$,
redraws the other spins from $\mu_H(\,\cdot\mid \sigma_{U^+(H)})$. Hence
 \begin{align}
 \cE_M(f)-\cE_{\widetilde M}(f)
 &=\sum_{H\subseteq E}\mathbb P_{\mu,\rm M}(H)\one_{\{W_L(H)\ge1\}}\Bigl[\Var_{\mu_H}(f)
   -\E_{\mu_H}\bigl[\Var_{\mu_H}(f\mid X_{U^+(H)})\bigr]\Bigr] \label{eq: var indicator}\\
 &\le\sum_{H\subseteq E}\mathbb P_{\mu,\rm M}(H)\one_{\{W_L(H)\ge1\}}\, \E_{\mu_H}[f^2] \label{eq: apply variance crude bound}\\
 &=\sum_\sigma\mu(\sigma)f(\sigma)^2
                         \mathbb P_{\sigma,\rm M}\bigl(W_L(H)\ge1\bigr). \label{eq: apply reversibility var}
 \end{align}
The inequality~\eqref{eq: apply variance crude bound} holds since
$\E_{\mu_H}[\Var_{\mu_H}(f\mid \sigma_{U^+(H)})]\ge0$ and $\Var_{\mu_H}(f)\le\E_{\mu_H}[f^2]$, and the
last equality~\eqref{eq: apply reversibility var} follows from \eqref{eq: reversibility}.
If $\nu_f(\sigma)=\frac{f(\sigma)^2\mu(\sigma)}{\E_\mu[f^2]}$,
and we rewrite \eqref{eq: apply reversibility var}, we get 
\begin{equation}
\label{eq: var M vs tilda M}
 \cE_M(f)-\cE_{\widetilde M}(f)
 \le \E_\mu[f^2]\,\PP_{\nu_f,M}\bigl(W_L(H)\ge1\bigr).
\end{equation}
By Markov's inequality, and Lemma~\ref{lem:entropy-mass} (with $C_1:=C_0'$), 
\begin{align*}
    \cE_M(f)-\cE_{\widetilde M}(f)
  \le  \E_\mu[f^2]( C_1 I_\mu(\nu_f)+n^{-5} )\,.
\end{align*}
By Lemma~\ref{lem:function-information} (with $C_2 = C$), this yields 
\begin{equation}
\label{eq:for:sw}
 \cE_M(f)-\cE_{\widetilde M}(f)
 \le C_1 C_2n\cE_{\rm Gl}(f)+n^{-5}\E_\mu[f^2]\,.
\end{equation}
Combining this with the bounds from Lemma~\ref{lemma:dirichlet:stronger} and the line following Lemma~\ref{lem: gap of M}, we obtain
\begin{equation}
     a\Var_\mu(f) \le \mathcal E_{M}(f) 
 \le (CL^\kappa + C_2C_3) n\cE_{\rm Gl}(f)+n^{-5}\E_\mu[f^2]\,. \label{eq: var bound almost done}
\end{equation}
To remove the last term, apply \eqref{eq: var bound almost done} with $g=f-\E_\mu[f]$ and subtract the last term from both sides: 
\[
 \bigl(a-n^{-5}\bigr)\Var_\mu(f)\le (CL^\kappa + C_2C_3) n\cE_{\rm Gl}(f).
\]
For $n$ large enough that $n^{-5}\le a/2$, and with $L= C_0\log n$,
\begin{equation}
    \label{eq: polylog var bound}
     \Var_\mu(f)\le \frac{2(CL^\kappa + C_2C_3)}{a}\, n\cE_{\rm Gl}(f)= O(n(\log n)^\kappa)\cE_{\rm Gl}(f)\,.
\end{equation}
This is the claimed Poincar\'e inequality, i.e.\ $\gap(P_{\rm Gl})^{-1}=O(n(\log n)^\kappa)$.
\end{proof}

\subsection{The same argument with entropy for nearly optimal mixing}

In order to get an $\tilde O(n)$ mixing time, rather than just inverse spectral gap, we wish to carry out the argument of the previous section with entropy instead of variance and get a $\tilde O(n)$ inverse log-Sobolev constant. In the proof above, the first step was Lemma~\ref{lem: gap of M} which together with the inverse spectral gap bound of~\cite{BG21} for the FK dynamics gave an $O(1)$ inverse spectral gap for the chain $M$. 

We replace those steps with the following analogue of Lemma~\ref{lem: gap of M} and the line following it. Recall that $\mu_H = \mu_{(V,H),\tilde \beta,q}$. 
The entropy counterpart of the Dirichlet form $\cE_M(f)=\E_H[\Var_{\mu_H}(f)]$ is
$$
 \mathcal H_M(f)=\E_{\mu, \rm{M}}[\Ent_{\mu_H}(f)]=\sum_{H\subseteq E}\PP_{\mu, \rm{M}}(H)\operatorname{Ent}_{\mu_H}(f)\,,
$$
where $\Ent_{\mu}(f) = \E_\mu [f \log f] - \E_{\mu}[f]\log \E_\mu[f]$ with the convention $0\log 0 = 0$ 
for every $f:[q]^V\rightarrow \mathbb R_{\ge0}$.

\begin{lemma}\label{lem:entropy-M}
Fix $q\ge 2, \Delta \ge 3$ and $B<B_u$. With probability $1-o(1)$ over $G$, there is a fixed $a(q,\Delta,B)>0$ such that 
for every $f:[q]^V\rightarrow \mathbb R_{\ge0}$, one has $
 a\Ent_\mu(f)\le\mathcal H_M(f).
$
\end{lemma}

Observe that $\mathcal H_M(f)\le \mathcal E_M(f,\log f)$ and therefore this implies a modified log-Sobolev inequality (MLSI) for $M$: i.e., an inequality of the form $\rho(P)  = \inf_{f: f\ge 0} \frac{\mathcal E_P(f,\log f)}{\Ent_\mu(f)}>0$.

Note that Lemma~\ref{lem:entropy-M} is combining two steps of our proof; the first is an MLSI comparison between $M$ and the FK dynamics, replacing Lemma~\ref{lem: gap of M}, and the second is a proof of an MLSI for the FK dynamics on random graphs for all $p<p_u$. This latter step is itself fairly significant, yielding an improvement of the FK mixing result of~\cite{BG21}: see Theorem~\ref{thm:fk-at}. Our proof of it actually uses a shattering-escape dichotomy this time in the world of random-cluster distributions rather than spin distributions. Namely, for a parameter $\ell \le L$, define for edge configurations $A \subset E$, 
\begin{equation}\label{eq:fk-long-mass}
 W_{\rm long}(A)=\sum_{U\in\cC(A)}|U|\,
 \one\{G[U]\text{ contains a cycle longer than }\ell\},
\end{equation}
where the sum runs over $\cC(A)$ denoting connected components in $(V,A)$.

\begin{lemma}\label{lem:fk-source}
Fix $q\ge 1, \Delta \ge 3$ and $p<p_u$. 
There exist constants $\alpha,C>0$ such that with probability $1-o(1)$ a random $\Delta$-regular graph on $n$ vertices is such that all its cycles of length at most $\ell = \alpha \log n$ are pairwise vertex disjoint, while at the same time the following holds. For every distribution $\nu$ on edge configurations $A\in \{0,1\}^E$, 
\begin{equation}\label{eq:fk-source}
 \E_\nu[W_{\rm long}]\le C\,I_\phi(\nu)+o(1)\,.
\end{equation}
Here, $I_\phi(\nu)$ is defined exactly as in Definition~\ref{def:local-relent} but with vertices replaced by edges in the obvious way. 
\end{lemma}

With Lemma~\ref{lem:entropy-M} in hand, recall $\widetilde M$ from Definition~\ref{def:M-tilde-chain}. Since the spins on $U^+(H)$ do not change in the recoloring step of $\widetilde M$, we can express $ \mathcal H_{\widetilde M}(f)$ in terms of its entropy conditional on $\sigma_{U^+(H)}$: 
$$ 
 \mathcal H_{\widetilde M}(f)
 =\sum_{H\subseteq E}\PP_{\mu, \rm{M}}(H)\E_{\mu_H}
       [\operatorname{Ent}_{\mu_H}(f\mid \sigma_{U^+(H)})]\,.
$$
In place of Lemma~\ref{lemma:dirichlet:stronger}, we will prove the following comparison of Dirichlet forms.
\begin{lemma}\label{lem: tilda M vs G}
Fix $\Delta\ge3$, $q\ge2$, $\tilde B>1$, and $K>0$. There exists
$C>0$ such that, with probability $1-o(1)$ the random $\Delta$-regular graph $G$ is such that 
for every $1\le L\le K\log n$ and every $f$, $$\cE_{\rm Gl}(f) \ge \frac{1}{C nL^\kappa}\mathcal H_{\widetilde M}(f^2)
 \,, \qquad \kappa=2+2\Delta\log_2 \tilde B\,.$$
\end{lemma}

\begin{proof}
Work on the graph event of Lemma~\ref{lemma:dirichlet:stronger}. 
For a component $F$ with $s\le L$ vertices, the Potts distribution on 
$F$ satisfies
$\mu_{F,\min}\ge q^{-s}(\tilde B)^{-\Delta s/2}$.
The entropy--variance comparison of~\cite[Theorem~A.1]{DSC96}
and~\eqref{eq:tilde-M-component-poincare} therefore give
\[
\Ent_{\mu_F}(g^2)
\le C\log(\tfrac{1}{\mu_{F,\min}})\Var_{\mu_F}(g)
\le Cs^\kappa\sum_{x\in F}\mu_F\!\left[
  \Var_{\mu_F}(g\mid\sigma_{F\setminus\{x\}})
\right].
\]

Tensorizing entropy over the components resampled by
$\widetilde M$, and applying the preceding bound with the exterior
spins fixed, we obtain
\[
\begin{aligned}
\mathcal H_{\widetilde M}(f^2)
&\le \E_{\mu,M}\Big[
  \sum_{\substack{F\in\mathcal C(H): |F|\le L}}
  \mu_H\!\left[
    \Ent_{\mu_H}(f^2\mid\sigma_{V\setminus F})
  \right]\Big]
\le CL^\kappa\sum_{x\in V}
  \E_{\mu,M}\Big[\mu_H\!\left[
    \Var_{\mu_H}(f\mid\sigma_{-x})
  \right]\Big]\,.
\end{aligned}
\]
But this last quantity is at most $C L^\kappa \sum_{x\in V} \E_\mu [\text{Var}_\mu(f \mid \sigma_{-x})]$ (as the above variance is conditioning on more information), and in turn that is exactly $ C nL^\kappa\cE_{\rm Gl}(f)$. 
\end{proof}

\begin{proof}[\textbf{\emph{Proof of Theorem~\ref{mainthm:Glauber}}}]
Fix $G$ in the intersection of the events of probability $1-o(1)$ in Lemmas~\ref{lem:entropy-M}, \ref{lem: tilda M vs G} and~\ref{lem:entropy-mass}, and on which~\eqref{eq: polylog var bound} holds. Let $L= C_0\log n$, with $C_0$ from Lemma~\ref{lem:entropy-mass}.

Take any $f$ not identically $0$. Given $H$, by~\eqref{eq: reversibility}, the chain $M$ redraws all spins from $\mu_H$, while $\widetilde M$ keeps $\sigma_{U^+(H)}$ fixed. Since $\mu_H$ is a product over the components of $H$, $\widetilde M$ redraws the other spins from $\mu_H(\,\cdot\mid\sigma_{U^+(H)})$. It will follow from a short calculation that 
\begin{align}
\mathcal H_M(f^2)-\mathcal H_{\widetilde M}(f^2)
&=\sum_{H\subseteq E}\mathbb P_{\mu,\rm M}(H)
  \Ent_{\mu_H}\!\left(
    \E_{\mu_H}[f^2\mid\sigma_{U^+(H)}]
  \right)\nonumber\\
&\le C_1\sum_{H\subseteq E}\mathbb P_{\mu,\rm M}(H)
  W_L(H)\E_{\mu_H}[f^2]
  \label{eq: apply entropy crude bound}\\
&=C_1\sum_\sigma\mu(\sigma)f(\sigma)^2
  \E_{\sigma,\rm M}[W_L(H)]\,.
  \label{eq: apply reversibility}
\end{align}
Indeed, the inequality~\eqref{eq: apply entropy crude bound} holds because $U^+(H)$ is a union of components of $H$ with $W_L(H)$ many vertices total and at most $\Delta W_L(H)/2$ edges. Thus, for every assignment $\tau$ on $U^+(H)$,
\begin{equation*}
\mu_H(\sigma_{U^+(H)}=\tau)
\ge e^{-C_1W_L(H)}\,,
\qquad
C_1=\log q+\frac{\Delta}{2}\log \tilde B\,.
\end{equation*}
Consequently, for every $\tau$, we have $\E_{\mu_H}[f^2\mid\sigma_{U^+(H)}=\tau] \le e^{ C_1 W_L(H)} \E_{\mu_H}[f^2]$,
which gives the claimed entropy bound. The last equality~\eqref{eq: apply reversibility} follows from~\eqref{eq: reversibility}.

If $\nu_f(\sigma)=f(\sigma)^2\mu(\sigma)/\E_\mu[f^2]$, rewriting~\eqref{eq: apply reversibility} gives
\begin{equation}
\label{eq: entropy M vs tilda M}
\mathcal H_M(f^2)-\mathcal H_{\widetilde M}(f^2)
\le C_1\E_\mu[f^2]\,\E_{\nu_f,M}[W_L(H)]\,.
\end{equation}
By Lemma~\ref{lem:entropy-mass}
and Lemma~\ref{lem:function-information}, this yields
\begin{equation*}
\mathcal H_M(f^2)-\mathcal H_{\widetilde M}(f^2)
\le C_1C_2C_3n\cE_{\rm Gl}(f)+C_1n^{-5}\E_\mu[f^2]\,.
\end{equation*}
Combining this with the bounds from Lemmas~\ref{lem: tilda M vs G}  (with $C_4 =C$) and Lemma~\ref{lem:entropy-M}, we obtain
\begin{equation}
a\Ent_\mu(f^2)
\le\mathcal H_M(f^2)
\le (C_4 L^\kappa+C_1C_2C_3)n\cE_{\rm Gl}(f)
   +C_1n^{-5}\E_\mu[f^2]\,.
\label{eq: ent bound almost done}
\end{equation}

To remove the last term, put $g=f-\mu[f]$. By Rothaus's lemma (see e.g.,~\cite[Lemma~5.1.4 and (5.1.8)]{BGL14} for a short proof using concavity and Jensen's inequality) and~\eqref{eq: ent bound almost done}, 
\begin{align*}
\Ent_\mu(f^2)
&\le\Ent_\mu(g^2)+2\Var_\mu(f)\\
&\le\frac{C_4 L^\kappa+C_1C_2C_3}{a}\,
       n\cE_{\rm Gl}(f)
  +\left(2+\frac{C_1}{an^5}\right)\Var_\mu(f)\,.
\end{align*}
By~\eqref{eq: polylog var bound}, and with $L= O(\log n)$, we get the log-Sobolev inequality, 
\begin{equation*}
\Ent_\mu(f^2)
\le O\bigl(n(\log n)^\kappa\bigr)\cE_{\rm Gl}(f)\,.
\end{equation*}
Since $\log \log(1/\mu_{\min})=O(\log n)$, we get $t_{\rm mix} = O(n (\log n)^{\kappa +1})$; see, e.g.,~\cite{MontenegroTetaliBook}. 
\end{proof}

\section{The shattering-escape dichotomy}\label{sec:key-stationary-estimate}

Our main aim in this section is to establish the shattering-escape dichotomies we use in the uniqueness regime on the random regular graph. Namely, we prove Lemma~\ref{lem:entropy-mass} and its random-cluster analogue  Lemma~\ref{lem:fk-source}. (Lemma~\ref{lem:info:bound}
is strictly weaker and therefore also proved; see Remark~\ref{rmk: weak prob bound}).

\subsection{The component exploration process}\label{sec:local}

For the proof of Lemma~\ref{lem:entropy-mass}, we use a coarse-grained exploration process of the sprinkled edge configuration under the distribution $\mathbb P_{\nu,M}$ where one first draws the spin configuration $\sigma \sim \nu$ and then the edge configuration under $\mathbb P_{\sigma,M}$. There will also be a similar exploration used for the purely random-cluster dichotomy of Lemma~\ref{lem:fk-source}. 

Let us first describe a general framework for an exploration process iteratively revealing connected components of an edge-subset $A\subset E$ under some distribution.  This will be applied both to the edge set~$H$ sampled from a spin law $\nu$ through~$\mathbb P_{\nu,M}$, and one sampled directly from an edge distribution.

\begin{definition}\label{def:edge-exploration-process}
    For a graph $G = (V,E)$, the following is an exploration procedure for the connected components of an edge subset $A\subset E$. As the process exposes edges of $A$, it will leave a remainder subgraph to be explored which is denoted $G^{\rm rem}$. We use $\partial^{\rm rem} U$ for a vertex subset $U$ to mean the edge boundary of $U$ in $G^{\rm rem}$. 

    Fix an arbitrary ordering of the vertices and edges of $G$ given by $\ell_V: V \to [n]$ and $\ell_E:E\to [n]$, and let $A \subset E$. The revealing will iteratively reveal connected components of $A$, with each connected component being iteratively explored by \emph{tests} of one edge $e$ at a time, and will retain a \emph{record} $\mathcal R$ of edges it has explored, and their indicators $\mathbf 1\{e\in A\}$. The  exploration is written in pseudocode in Algorithm~\ref{alg:exploration-reza}. 
    \begin{algorithm}[h]
\caption{Exploration of components of $A\subseteq E$}\label{alg:exploration-reza}
\begin{algorithmic}[1]
\State Initialize $G^{\rm rem}\gets G$;\ $j\gets 0$;\ and $\mathcal R \gets \emptyset$
\While{$G^{\rm rem} \neq\emptyset$}  \State  $u\gets\operatorname*{arg\,min}_{w\in V(G^{\rm rem})}\ell_V(w)$
 \State Initialize $U\gets\{u\}$ 
 \While{$\partial^{\rm rem}U \neq\emptyset$} \Comment{Explore component $\mathcal C_u$}
  \State $e\gets\operatorname*{arg\,min}_{f\in\partial^{\rm rem} U}\ell_E(f)$ and write $e=yv$ with $y\in U$, $v\notin U$
  \State append $(e,\mathbf 1\{e\in A\})$ to $\mathcal R$ \Comment{test $e$}
  \If{$e\in A$}
   \State $U\gets U\cup\{v\}$
  \Else
   \State $E(G^{\rm rem})\gets E(G^{\rm rem})\setminus \{e\}$
   \EndIf
 \State $j\gets j+1$;\ 
 \EndWhile
 \State $\mathcal C_u \gets U$ and $V(G^{\rm rem}) \gets V(G^{\rm rem})\setminus \mathcal C_u$
 \EndWhile
\end{algorithmic}
\end{algorithm}
\end{definition}

Each execution of lines~6--7 of Algorithm~\ref{alg:exploration-reza} is called a \emph{test}: it selects the edge $e=yv\in\partial^{\rm rem}U$ with $y\in U$ and $v\notin U$, and records whether $e\in A$. We call the test \emph{successful} if it records $(e,1)$, in which case~$v$ is added to~$U$.
Note that only tested edge indicators $\mathbf 1\{e\in A\}$ are recorded into $\mathcal{R}$ by the process, and not all edges of $A$ are eventually exposed. 
Each $\mathcal{C}_u$ is a connected component of $(V,A)$. 
The complete record $\mathcal{R}$ therefore determines $\mathcal{C}(A)$, and hence both component statistics $W_L$ and $W_{\rm long}$ of Lemmas~\ref{lem:entropy-mass} and~\ref{lem:fk-source}.

\begin{definition}\label{def:good-test}
Consider a test of $e=yv$, so that $y\in U$ and $v\notin U$. 
For an integer~$R\ge1$, the \emph{forward region} of the test is
\[
 F=\bigl\{w\in V(G^{\rm rem}) : \operatorname{dist}_{G^{\rm rem}\setminus \{yv\}}(v,w)\le R\bigr\},
\]
the set of vertices joined to~$v$ by a path of length at most~$R$ in the remaining graph that does not use~$yv$. 
We call a test of $e=yv$ \emph{good} if
\begin{enumerate}
\item $F\cap U=\emptyset$, and
\item the ball of radius~$R+1$ about~$v$ in~$G^{\rm rem}$ is a tree.
\end{enumerate}
Both conditions depend only on $G^{\rm rem}$ and~$U$ at the time of the test. 
In particular, whether a test is good is known before $\mathbf 1\{e\in A\}$ is revealed.
\end{definition}

Consider a good test of $e=yv$ with forward region~$F$. Let $F^+$ be the subgraph of~$G^{\rm rem}$ consisting of all edges with at least one endpoint in~$F$, together with their endpoints. We call~$F^+$ the \emph{forward tree} of the test. Then
\[
 V(F^+)=F\cup\partial F^+,\qquad \partial F^+=\{y\}\cup\{w : \operatorname{dist}_{G^{\rm rem}-yv}(v,w)=R+1\},
\]
and $e\in E(F^+)$. Since the test is good, $F^+$ is a tree in which~$v$ and every other vertex of~$F$ have at most~$d$ children, so $|F|\le b_R:=\sum_{j=0}^R d^j$.

Let~$\mathcal R_j$ be the record after the first~$j$ tests, i.e., the value of $\mathcal{R}$ when the counter in Algorithm~\ref{alg:exploration-reza} equals $j$. 
We call any possible value~$h$ of~$\mathcal R_j$ a \emph{history}. 
For a good test, we make the following observations on its forward region. 

\begin{observation}\label{obs:exploration-local-laws}
Fix a history~$h$ of $\mathcal{R}_j$ at which 
the $(j+1)$-th test on $e=yv$ is good, and let~$F^+$ be its forward tree.
\begin{enumerate}
\item For every $\tau\in[q]^{V\setminus F}$ with $\mu(\sigma_{V\setminus F}=\tau \mid \mathcal{R}_j = h)>0$,
$$
 \mu(\sigma_F\in\cdot\mid\sigma_{V\setminus F}=\tau, \mathcal{R}_j = h)=\mu_{F^+}^\tau(\sigma_F\in\cdot)
.$$
\item  For every $\eta\in\{0,1\}^{E\setminus E(F^+)}$ with $\phi(\omega_{E\setminus E(F^+)}=\eta \mid \mathcal{R}_j = h)>0$,
$$
 \phi(\omega_{E(F^+)}\in\cdot\mid\omega_{E\setminus E(F^+)}=\eta, \mathcal{R}_j=h)=\phi_{F^+}^\zeta(\cdot)
,$$ 
where~$\zeta$ is the partition of~$\partial F^+$ induced by the open edges of~$\eta$.
\end{enumerate}
\end{observation}

\subsection{Subcriticality of the exploration under the Potts distribution}

We begin by showing that under the Potts distribution~$\mu$, the probability of including the edge at the root of a sufficiently large  tree of degree at most $\Delta$ is strictly less than~$1/d$, uniformly over the boundary condition on the leaves of the tree. 
We recall $\hat p=p/[p+q(1-p)]$, put~$s_\infty=r+(1-r)\hat p<1/d$ for $r$ as in~\eqref{eq:parameter-defs}, and fix~$s\in(s_\infty,1/d)$.

\begin{lemma}\label{lem:H-local-opening}
Suppose $\beta<\beta_u(q,\mathbb T_\Delta)$.
There exists~$R\ge1$, depending only on $\Delta,q,\beta,s$, such that the following holds. Let $F^+$ be the forward tree of a good test of~$yv$. Let~$\tau$ be an arbitrary spin assignment to~$\partial F^+$, sample the spins $\sigma$ on~$F$ from~$\mu_{F^+}^\tau$, and then sample~$H$ through~$\mathbb P_{\sigma,M}$. Then
\begin{equation}\label{eq:H-local-opening-bound-Potts}
 \PP_{\mu_F^\tau,M}(yv\in H)\le s\,.
\end{equation}
\end{lemma}

As $s < 1/d$, Lemma~\ref{lem:H-local-opening} makes the exploration process subcritical at good tests. 
\begin{proof}[Proof of Lemma~\ref{lem:H-local-opening}]
The bound reduces to the decay, under tree uniqueness, of the probability that the root of a $d$-ary tree is connected to its leaves under wired boundary condition. 
Let~$e=yv$. By the definition of~$M$, given~$\sigma$ we have $\mathbf 1\{e\in H\}=\max\{\omega_e,\eta_e\}$, where $\omega_e\sim\mathrm{Ber}(p)$ if $\sigma_y=\sigma_v$ and $\omega_e=0$ otherwise, and $\eta_e\sim\mathrm{Ber}(r)$ is independent of $(\sigma,\omega)$. Equivalently, $e\in H$ with probability 
\begin{equation}\label{eq:sprinkling-identity}
 p+(1-p)r=\frac{B-1}{B}+\frac{r}{B}=\frac{r(\tilde B-1)+r}{B}=\frac{r\tilde B}{B}\,,
\end{equation}
if $\sigma_y=\sigma_v$, and with probability~$r$ otherwise. Hence

\begin{equation}\label{eq:H-ES-reduction}
 \PP_{\mu_F^\tau,M}(e\in H)
 =r+(1-r)\PP_{\mu_F^\tau,\ES}(\omega_e=1)\,,
\end{equation}
where $\PP_{\mu_F^\tau,\ES}$ is the Edwards--Sokal law  in~\eqref{eq:bottleneck-PhatSW-to-P}, applied on~$F^+$ to the spins $\sigma_F\sim\mu_F^\tau$. We show below that, uniformly in~$\tau$, the probability $\PP_{\mu_F^\tau,\ES}(\omega_e=1)$ is at most a quantity tending to $\hat p$ as $R\to\infty$. Since
\[
 s_\infty:=r+(1-r)\hat p
 =\hat p+r(1-\hat p)
 =\hat p+\frac{(B-1)(1-\tilde\theta)}{q\tilde\theta}\cdot\frac{q}{B+q-1}
 =\hat p+\hat p\,\frac{1-\tilde\theta}{\tilde\theta}
 =\frac{\hat p}{\tilde\theta}<\frac1d\,,
\]
by~\eqref{eq:parameter-defs} and $\tilde\theta>d\hat p$, this gives the lemma.

We next compare the random-cluster marginal of~$\PP_{\mu_F^\tau,\ES}$ with the wired random-cluster measure~$\phi^1_{F^+}$ on~$F^+$, in which all vertices of~$\partial F^+$, including~$y$, are identified. Summing the Edwards--Sokal weight over~$\sigma_F$ shows that the random-cluster marginal of~$\PP_{\mu_F^\tau,\ES}$ is the random-cluster measure on~$F^+$ with boundary condition~$\xi_\tau$, in which boundary vertices of the same $\tau$-colour are identified, conditioned on the event~$D_\tau$ that no two boundary vertices of different $\tau$-colours are connected. The event~$D_\tau$ is decreasing, so by positive association conditioning on it can only decrease the measure stochastically. Since $\xi_\tau$ is dominated by the wired boundary condition, boundary monotonicity then gives
\[
 \PP_{\mu_F^\tau,\ES}(\omega_e\in\cdot)\preceq\phi^1_{F^+}(\omega_e\in\cdot)
 \qquad\text{for every }\tau\,;
\]
see~\cite[Section~2.2.1]{BG21}.

Let~$x_h$ be the wired root-to-boundary connection probability on a complete~$d$-ary tree of depth~$h$. By uniqueness,~$x_h\to0$ as~$h\to\infty$; see~\cite[Section~2.2.2, including (2.2)]{BG21}, in which a tree has~$\Delta$ branches at the root, and by monotonicity deleting one branch off the root gives our version of $x_h$. In particular,~$x_{R+1}$ upper bounds the probability that~$v$ connects to the wired leaves in~$F^+\setminus\{e\}$.

To relate $\phi_{F^+}^1(\omega_e=1)$ to $x_{R+1}$, note that if~$v$ already reaches the wired leaves in~$F^+\setminus\{e\}$, including~$e$ contributes~$p$ to $\phi_{F^+}^1(\omega_e=1)$; otherwise it merges two components and contributes~$p/q$. Excluding~$e$ contributes~$1-p$ in either case. The resulting normalized probability is
\begin{equation}\label{eq:H-auxiliary-opening}
 \PP_{\mu_F^\tau,\ES}(\omega_e=1)
 \le r_R:=\frac{p x_{R+1}+(p/q)(1-x_{R+1})}
 {x_{R+1}+(1-p+p/q)(1-x_{R+1})}
 =\frac{\hat p[1+(q-1)x_{R+1}]}
 {1+(q-1)\hat p x_{R+1}}\,.
\end{equation}
Since~$x_h\to0$, choose~$R\ge1$ large enough that
\begin{equation}\label{eq:choose-R}
 r+(1-r)\frac{\hat p[1+(q-1)x_{R+1}]}
 {1+(q-1)\hat p x_{R+1}}
 =s_\infty+
 \frac{(q-1)\hat p(1-s_\infty)x_{R+1}}
      {1+(q-1)\hat p x_{R+1}}
 \le s\,.
\end{equation}
Together with~\eqref{eq:H-ES-reduction}, this proves~\eqref{eq:H-local-opening-bound-Potts}. 
\end{proof}

We also obtain the following Lemma as a consequence of this proof, which already follows by the random-cluster calculations above, and domination of any random-cluster boundary condition by the wired one. 

\begin{lemma}\label{lem:H-local-opening-FK}
Suppose $q\ge 1$ and $\beta<\beta_u(q,\mathbb T_\Delta)$. There exists~$R\ge1$, depending only on $\Delta,q,\beta,s$, such that the following holds. Let $F^+$ be the forward tree of a good test of~$e=yv$, and let~$\zeta$ be an arbitrary boundary partition of~$\partial F^+$. Then
\begin{equation}\label{eq:H-local-opening-bound-FK}
 \phi_{F^+}^\zeta(\omega_e=1)\le s<1/d\,.
\end{equation}
\end{lemma}

\subsection{An intermediate subcritical exploration process}\label{sec:information}
We consider applying Algorithm~\ref{alg:exploration-reza} to an edge set $A$ drawn from an arbitrary distribution.

\begin{definition}\label{def:P-revealing-distribution}
To treat the spin and edge cases together, we define $P$, $X$ and $\rho$ as follows.
\begin{enumerate}
\item Spin case: given an arbitrary spin distribution~$\nu$, sample $\sigma\sim\nu$, then $H\sim\PP_{\sigma,M}$, and explore~$H$ by Algorithm~\ref{alg:exploration-reza}. The distribution of the resulting record~$\mathcal R$ is denoted by~$P$. Let $X=\sigma$ be the Potts configuration and $\rho=\mu$ the reference distribution.
\item Edge case: given an arbitrary edge distribution~$\nu$, sample $\omega\sim\nu$ and explore the open edges of~$\omega$ by Algorithm~\ref{alg:exploration-reza}. The distribution of~$\mathcal R$ is again denoted by~$P$. Let $X=\omega$ and $\rho=\phi$.
\end{enumerate}
\end{definition}

In both cases, $\nu$ is the law being compared with~$\rho$.
Given a history $h$ with $P(\mathcal R_j=h)>0$,
write~$\nu_h$ and~$\rho_h$ for the conditional laws of~$X$
given that history, starting from~$\nu$ and~$\rho$, respectively.
Note that in the spin case, $\rho_h(\cdot) = \mu(\cdot \mid \mathcal{R}_j = h)$ is the conditional spin distribution, given the revealed edge records. In the edge case, $\rho_h=\phi(\cdot\mid \mathcal{R}_j=h)$
is just random-cluster distribution conditioned on the tested edge values. 
By~\eqref{eq:H-local-opening-bound-Potts} and~\eqref{eq:H-local-opening-bound-FK}, if an edge $e$ is generated from the distributions~$ \PP_{\mu_F^\tau,M}$ (resp. $\phi_{F^+}^\zeta$), a good test finds it present with probability at most~$s<1/d$, uniformly in the boundary condition $\tau$ on $\partial F^+$ (resp. $\zeta$). 
On the other hand, since $\nu$ is arbitrary, $\nu_h$ has some complicated distribution, and this uniform $1/d$ upper bound need not hold under~$\nu_h$.

We construct an intermediate distribution $Q$ on edge test outcomes that, conditional on $h$, 
samples the configuration $\tau$ outside $F$ from $\nu_h$, 
and then samples the configuration $X$ in $F$ from $\rho^\tau_F$. 
On the one hand, such a construction guarantees that good tests remain subcritical under $Q$; on the other hand, 
it enables the difference between revealing edges under $Q$ and under $\nu$ to be controlled. Next we formally define $Q$.

\begin{definition}\label{def:Q-revealing-distribution}
Let~$Q$ be the distribution of the record produced by running Algorithm~\ref{alg:exploration-reza} in which, at each test (line~7),  the outcome of the edge $e$ given the current record~$h$, is not read off from a fixed set~$A$ but drawn from a distribution $Q_h$ as follows. 
\begin{enumerate}
\item If $P(\mathcal R_j=h)>0$ and the test is good, the indicator of~$e$ is drawn as follows.
 \begin{itemize}
 \item Spin case: sample $\tau=\sigma_{V\setminus F}$ from~$\nu_h$, then $\sigma_F$ from~$\mu_{F^+}^\tau$, and include~$e$ with probability $\PP_{\sigma,M}(e\in H)$.
 \item Edge case: sample $\eta=\omega_{E\setminus E(F^+)}$ from~$\nu_h$, then $\omega_{E(F^+)}$ from~$\phi_{F^+}^{\zeta}$, where~$\zeta$ is the partition of~$\partial F^+$ induced by~$\eta$, and record~$\omega_e$.
 \end{itemize}
\item If $P(\mathcal R_j=h)>0$ and the test is not good, then the outcome of $e$ has the same distribution as under~$P$.
\item If $P(\mathcal R_j=h)=0$,  the test records $(e,0)$.
\end{enumerate}
\end{definition}

By Observation~\ref{obs:exploration-local-laws}, the distributions $\mu_{F^+}^\tau$ and $\phi_{F^+}^\zeta$ used in item~1 of Definition~\ref{def:Q-revealing-distribution} are the conditional distributions of~$\rho_h$ inside~$F$ (resp.~$E(F^+)$) given the exterior. Lemmas~\ref{lem:H-local-opening}--\ref{lem:H-local-opening-FK} imply that 
\begin{align}\label{eq:Q-is-subcritical}
    Q(e \text{ is present} \mid e \text{ is a good test},\mathcal R_j = h) \le s <1/d\,.
\end{align}
Therefore, so long as bad tests are rare (which is true given the  treelike geometry of the random graph), the exploration process under $Q$ will be subcritical. 
The next lemma bounds the relative entropy of $P$ with respect to~$Q$ in terms of~$I_\rho(\nu)$.

\begin{lemma}\label{lem:exploration-comparison}
The distribution~$P$ is absolutely continuous with respect to~$Q$. Moreover, for every $R\ge1$, there is a constant $C_{\rm obs}=O(b_R)\cdot e^{O(b_R)}$, depending only on $\Delta$, $q$, $\beta$ and $R$, such that
\begin{equation}\label{eq:exploration-comparison}
 \KL(P\Vert Q)\le C_{\rm obs}\,I_\rho(\nu)\,.
\end{equation}
\end{lemma}

For the proof of Lemma~\ref{lem:exploration-comparison} we first provide a crude bound for approximate tensorization that will be used on $F\subseteq V(G)$, that is, for any spin assignment $\tau\in [q]^V$, and any function $f:[q]^V\rightarrow \mathbb R_{\ge 0}$, 
\begin{equation}
    \label{eq: crude AT for small graph}
\Ent_{\mu_F^\tau} (f) \le C_{\rm AT} \sum_{x\in F} \E_{\mu_F^\tau}[\Ent_x(f)].
\end{equation}
By taking $f=d\nu/d_{\mu_F^\tau}$, in Lemma~\ref{lem:local-kl},  we  present an equivalent form of \eqref{eq: crude AT for small graph} in terms of KL divergence and aggregate local entropy.
In particular, this bound will be useful to compare the exploration under the Potts distribution~$\mu$ to the exploration under another spin distribution~$\nu$, for example~$\nu=\mu(\cdot\mid A)$ for a potential bottleneck set~$A$. Then we show this bound also holds for the random-cluster distribution.

\begin{lemma}[{\cite[Lemma~2.2]{CMT15}}]\label{lem:local-kl}
Let $G$ be a finite graph of maximum degree at most $\Delta$.
\begin{enumerate}
\item  Let $F\subseteq V(G)$ be nonempty, fix spins $\tau$ on $V(G)\setminus F$, and let $\mu_F^\tau$ be the resulting conditional Potts distribution on $[q]^F$. For every probability distribution $\nu$ on $[q]^F$,
\begin{equation}\label{eq:local-kl-bound}
 \KL(\nu\Vert\mu_F^\tau)\le B^{3\Delta|F|}\, I_{\mu_F^\tau}(\nu)\,.
\end{equation}
\item Let $J\subseteq E(G)$ be nonempty, fix an arbitrary boundary condition $\zeta$, and let $\phi_J^\zeta$ be the resulting random-cluster distribution on $\{0,1\}^J$. For every probability distribution $\nu$ on $\{0,1\}^J$,
\begin{equation}\label{eq:local-kl-bound-FK}
 \KL(\nu\Vert\phi_J^\zeta)\le q^{3|J|}\, I_{\phi_J^\zeta}(\nu)\,.
\end{equation}
\end{enumerate}
\end{lemma}

\begin{proof}[Proof of Lemma~\ref{lem:exploration-comparison}]
Given a history~$\mathcal{R}_j=h$, let~$P_h$ and~$Q_h$ be the respective distributions of the outcome of the $(j+1)$-th test under $P$ and under $Q$. Whenever~$P(\mathcal{R}_j=h)>0$, $Q_h$ gives positive probability to every configuration allowed by~$\nu_h$. Thus~$P_h$ is absolutely continuous with respect to $Q_h$ and in turn,~$P$ is absolutely continuous with respect to $Q$. The law of total relative entropy gives
\begin{equation}\label{eq:record-chain-rule}
 \KL(P\Vert Q)
 =\sum_{j=0}^{|E|-1}
   \E_P\!\left[
     \KL(P_{\mathcal R_j}\Vert Q_{\mathcal R_j})
   \right]\,.
\end{equation}
Tests that are not good contribute zero, since 
$Q_h=P_h$ in this case. 
We therefore wish to bound $\KL(P_h \Vert Q_h)$ at good tests.

At a good test with history~$h$, let~$A(h)=F$ in the spin case and~$A(h)=E(F^+)$ in the edge case. Under both~$P_h$ and~$Q_h$, the configuration outside~$A(h)$ is sampled from its marginal under~$\nu_h$. Given these exterior configurations, the configurations in~$A(h)$ are sampled from their conditional law under~$\nu_h$ for~$P_h$, and under~$\rho_h$ for~$Q_h$. 
By the chain rule for relative entropy, the data processing inequality, and the local KL comparison bounds of Lemma~\ref{lem:local-kl}, we get
\begin{equation}\label{eq:local-record-cost-a}
 \begin{aligned}
 \KL(P_h\Vert Q_h)
 &\le
 \E_{X_{A(h)^c}\sim\nu_h}\Big[
 \KL\big(
   \nu_h(X_{A(h)}\in\cdot\mid X_{A(h)^c})
   \,\big\Vert\,
   \rho_h(X_{A(h)}\in\cdot\mid X_{A(h)^c})
 \big)\Big]\\
 &\le e^{O(\Delta b_R)}\sum_{i\in A(h)}
 \E_{X_{-i}\sim\nu_h}\Big[
 \KL\big(\nu_h(X_i\in\cdot\mid X_{-i})\,\big\Vert\,
         \rho_h(X_i\in\cdot\mid X_{-i})\big)\Big]\,,
 \end{aligned}
\end{equation}
where a coordinate $i$ is a vertex in the spin case and an edge in the edge case, and we used $|A(h)|\le\Delta b_R$.
For a coordinate~$i$, define
\begin{equation}\label{eq:D}
 D_i(h)
 =\mathbb E_{X_{-i} \sim \nu_h}[
   \KL\!\left(
      \nu_h(X_i\in\cdot\mid X_{-i})
      \,\Vert\,
      \rho_h(X_i\in\cdot\mid X_{-i})
   \right)]\,.
\end{equation}
Combining~\eqref{eq:record-chain-rule}, \eqref{eq:local-record-cost-a} and~\eqref{eq:D}, we obtain
\begin{align}
    \label{eq:record-cost-reduction}
 \KL(P\Vert Q)
 &\le e^{O(\Delta b_R)}
 \E_P\!\Big[
   \sum_{\substack{0\le j<|E|:\\\text{test }j+1\text{ is good}}}
   \ \sum_{i\in A(\mathcal R_j)}D_i(\mathcal R_j)
 \Big]\,.
\end{align}

Note that 
$D_i(\mathcal R_0)=\E_{\nu_{-i}}
 [\KL\!\left(\nu(X_i\in \cdot\mid X_{-i})\,\Vert\,
                    \rho(X_i\in \cdot\mid X_{-i})\right)]$.
We next aim to show that for $i\in \mathcal{R}_j$, it holds that 
$\E_P[D_i(\mathcal{R}_j)]\le D_i(\mathcal R_0)$.
We use the following consequence of the chain rule. If~$\alpha$ and~$\beta$ are two distributions of a pair~$(U,Y)$ with the same conditional distribution of~$Y$ given~$U$, then
\begin{equation}\label{eq:KL-two-orders}
 \E_{\alpha}\big[\KL\big(\alpha(U\in\cdot\mid Y)\,\Vert\,\beta(U\in\cdot\mid Y)\big)\big]
 \le \KL\big(\alpha(U\in\cdot)\,\Vert\,\beta(U\in\cdot)\big)\,.
\end{equation}
Fix a coordinate~$i$. If $i\in A(\mathcal R_j)$, then the edge tested at step~$j+1$ has an endpoint within distance~$R$ of a vertex of~$i$, so it is one of at most $2\Delta(1+\Delta b_R)$ edges, and each edge is tested at most once. For each such edge~$e$, let~$\mathcal R^{(e)}$ be the record just before~$e$ is tested, or the full record if~$e$ is never tested. 
Observe that given~$X$, $\mathcal R^{(e)}$ is generated by the same rule whether~$X$ is drawn from~$\nu$ or~$\rho$. 
 Applying~\eqref{eq:KL-two-orders} conditionally on~$X_{-i}$, with $U=X_i$ and $Y=\mathcal R^{(e)}$, and averaging over~$X_{-i}\sim\nu$, gives
\[
 \E_P\big[D_i(\mathcal R^{(e)})\big]\le D_i(\mathcal R_0)
 =\E_{\nu}\big[\KL\big(\nu(X_i\in\cdot\mid X_{-i})\,\Vert\,\rho(X_i\in\cdot\mid X_{-i})\big)\big]\,.
\]
Since $D_i\ge0$, summing over these edges and over~$i$ yields
\[
 \E_P\Big[\sum_{j:\,\mathrm{good}}\ \sum_{i\in A(\mathcal R_j)}D_i(\mathcal R_j)\Big]
 \le 2\Delta(1+\Delta b_R)\sum_iD_i(\mathcal R_0)=2\Delta(1+\Delta b_R)\,I_\rho(\nu)\,,
\]
and together with~\eqref{eq:record-cost-reduction} this gives $\KL(P\Vert Q)\le O(b_R)\cdot e^{O(b_R)}I_\rho(\nu)$, as claimed.
\end{proof}

To conclude this subsection, we convert the relative entropy bound of Lemma~\ref{lem:exploration-comparison} into a bound on expectations under~$P$. The quantities we need to control, such as $W_L$ and $W_{\rm long}$, are functions of the form $W(\mathcal R)$. Under~$P$ they are hard to analyze directly, but under~$Q$ the exploration is subcritical, and in Section~\ref{sec:tails} we will show that their exponential moments $\E_Q[e^{\lambda W(\mathcal R)}]$, for a suitable fixed $\lambda>0$, are close to~$1$. The following corollary bounds $\E_P[W(\mathcal R)]$ by $\frac{1}{\lambda}\log\E_Q[e^{\lambda W(\mathcal R)}]$, at an additional cost of $\frac{1}{\lambda}C_{\rm obs}I_\rho(\nu)$.

\begin{corollary}\label{cor:entropy-transfer}
For every nonnegative function $W=W(\mathcal R)$ of the  record~$\mathcal R$ and every~$\lambda>0$,
\begin{equation}\label{eq:entropy-transfer}
 \E_P[W(\mathcal R)]\le\frac1\lambda\Big(C_{\rm obs}\,I_\rho(\nu)+\log\E_Q\big[e^{\lambda W(\mathcal R)}\big]\Big)\,.
\end{equation}
\end{corollary}

\begin{proof}
Let $S$ be the set of records~$h$ with $P(\mathcal R=h)>0$. By Lemma~\ref{lem:exploration-comparison}, $Q(\mathcal R=h)>0$ for every $h\in S$. By concavity of the logarithm and Jensen's inequality,
\[
\begin{aligned}
 \lambda\,\E_P[W(\mathcal R)]-\KL(P\Vert Q)
 &=\sum_{h\in S}P(\mathcal R=h)\log\frac{e^{\lambda W(h)}\,Q(\mathcal R=h)}{P(\mathcal R=h)}\\
 &\le\log\sum_{h\in S}e^{\lambda W(h)}\,Q(\mathcal R=h)
 \le\log\E_Q\big[e^{\lambda W(\mathcal R)}\big]\,.
\end{aligned}
\]
Rearranging and bounding $\KL(P\Vert Q)\le C_{\rm obs}I_\rho(\nu)$ by Lemma~\ref{lem:exploration-comparison} gives~\eqref{eq:entropy-transfer}.
\end{proof}

\subsection{Tail bounds on component sizes under $Q$}\label{sec:tails}

We now show that under~$Q$ every component has an exponential size tail, uniformly in the past of the exploration. By~\eqref{eq:Q-is-subcritical}, each good test finds its edge present with probability at most $s<1/d$. A test that is not good can only occur within distance $R+1$ of an edge lying outside the tree discovered so far. Hence, if we  guarantee that there are few such edges compared to the size of the component then the exploration of each component is dominated by a subcritical branching process.

 Using that $s<1/d$, fix~$\delta>0$ sufficiently small, and~$\gamma\in(0,1)$ sufficiently close to $1$ such that
\begin{equation}\label{eq:sparse-parameters}
 \delta< \frac{1-ds}{16(1+\Delta b_R)^2},\qquad  \text{and} \qquad ds^\gamma<1\,.
\end{equation}

\begin{lemma}\label{lem:exploration-tails}
There exist constants $c_0, C_1, C_{\rm tail}>0$, depending only on $\Delta,R,s,\delta,\eta$, such that the following holds. Let~$G$ be a graph on $n$ vertices of maximum degree~$\Delta$ satisfying $|E(G[U])|\le(1+\delta)|U|$  for every~$U\subseteq V$ with~$|U|\le\eta n$. For any given distribution~$\nu$, let~$Q$ be its distribution of records as given in Definition~\ref{def:Q-revealing-distribution}. Let~$h$ be a history with $Q(\mathcal R_j=h)>0$ after which Algorithm~\ref{alg:exploration-reza} is at line~3, and let~$\mathcal C$ be the component that is explored next. Then for every $m\ge1$,
\begin{equation}\label{eq:component-tail-statement}
 Q\big(|\mathcal C|\ge m\mid\mathcal R_j=h\big)\le C_{1}\,e^{-2c_0m}\,,
\end{equation}
and for every integer $a\ge0$,
\begin{equation}\label{eq:component-moment-statement}
 \E_Q\big[e^{c_0|\mathcal C|}\one\{|\mathcal C|>a\}\mid\mathcal R_j=h\big]\le C_{\rm tail}\,e^{-c_0a}\,.
\end{equation}
\end{lemma}

\begin{proof}[Proof of Lemma~\ref{lem:exploration-tails}]
First suppose $m\le\eta n/(2\Delta b_R)$. Stop the exploration of~$\mathcal C$ when it is completed or when $|U|=m$, whichever comes first. Every edge tested before stopping has an endpoint in~$U$, and~$U$ is connected with at most~$m$ vertices, so at most $\Delta m-(m-1)=1+dm$ tests are performed before stopping.

Let~$T_0$ be the set of edges~$e$ for which their tests are successful after~$h$ and before stopping. Each such test adds its endpoint~$v\notin U$ to~$U$ at line~9 of Algorithm~\ref{alg:exploration-reza}, so~$T_0$ is a spanning tree of~$G[U]$ at the stopping time. Let~$U^+$ be the radius-$(R+1)$ neighbourhood of~$U$ in~$G$. Then~$G[U^+]$ is connected and $|U^+|\le(1+\Delta b_R)|U|\le 2\Delta b_R m\le\eta n$, so~$T_0$ extends to a spanning tree~$T$ of~$G[U^+]$.
Let $E_{\rm exc}:=E(G[U^+])\setminus E(T)$ be the set of excess edges, so that by assumption,
\begin{equation}\label{eq:excess}
 |E_{\rm exc}|=e_G(U^+)-|U^+|+1\le 1+2\delta\Delta b_R m\,.
\end{equation}

Consider a successful test of $e=yv$ before stopping that is not good, so that $v\in U$ and $e\in T_0$ at the stopping time. 
We claim that $\operatorname{dist}_G(v,E_{\rm exc})\le R+1$.
Since distances in~$G^{\rm rem}$ are at least those in~$G$, every vertex within $G^{\rm rem}$-distance~$R+1$ of~$v$ lies in~$U^+$. By Definition~\ref{def:good-test}, there are two cases.
\begin{enumerate}
\item[(a)] $G^{\rm rem}[B_{R+1}(v)]$ is not a tree. Then it contains a cycle, and thus an edge in $E_{\rm exc}$.
\item[(b)] $F\cap U\neq\emptyset$, with~$U$ taken at the time of the test. Pick $w\in F\cap U$. By definition of~$F$, there is a path~$P$ from~$v$ to~$w$ in $G^{\rm rem}\setminus\{yv\}$ with $V(P)\subseteq F$. Let~$P'$ be the path from~$w$ to~$y$ in~$T_0$. Then $P\cup P'\cup\{yv\}$ contains a cycle through~$yv$. Since $P'\cup\{yv\}\subseteq T_0\subseteq T$ and~$T$ is a tree, some edge of~$P$ lies in~$E_{\rm exc}$, and its endpoints lie in~$F$.
\end{enumerate}
A successful test of an edge $e=yv$ adds a new vertex~$v$ to~$U$, and by the claim, if the test is not good then $v\in\bigcup_{f\in E_{\rm exc}}B_{R+1}(f)$. Since $|B_{R+1}(f)|\le 2(1+\Delta b_R)\le 4\Delta b_R$ for every edge~$f$, by~\eqref{eq:excess} the number of successful tests before stopping that are not good is at most
\begin{equation}\label{eq:bad-count}
 4\Delta b_R\,|E_{\rm exc}|\le 4\Delta b_R\,(1+2\delta\Delta b_R m)\,.
\end{equation}

Reaching~$|U|=m$ requires $m-1$ successful tests, so by~\eqref{eq:bad-count} it requires at least $m-1-4\Delta b_R(1+2\delta\Delta b_R m)$ successful good tests. For $i\ge1$, let~$Z_i$ be the indicator that the $i$th test after~$h$ is good and successful, with $Z_i=0$ if there is no such test. Whether this test is good is determined by~$\mathcal R_{j+i-1}$. If it is good, then it is successful with conditional probability at most~$s$ by~\eqref{eq:Q-is-subcritical}; otherwise $Z_i=0$, by Definition~\ref{def:Q-revealing-distribution}. Hence, for every $\lambda>0$,
\begin{equation}\label{eq:test-mgf}
 \E_Q\big[e^{\lambda Z_i}\mid\mathcal R_{j+i-1}\big]\le 1-s+se^\lambda\,.
\end{equation}
Since~$1-8\delta \Delta^2 b_R^2 > (1+ds)/2 >ds$ by~\eqref{eq:sparse-parameters}, choose a fixed sufficiently small~$\lambda>0$ such that
\begin{equation}\label{eq:tail-constants}
 (1-8 \delta \Delta^2 b_R^2)\lambda
     -d\log(1-s+se^\lambda)>0\,.
\end{equation}
On the event $\{|\mathcal C|\ge m\}$, at most $1+dm$ tests are performed before stopping, and at least $(1-8\delta\Delta^2b_R^2)m-1-4\Delta b_R$ of them are good and successful. Hence, iterating~\eqref{eq:test-mgf} by applying Markov's inequality to $e^{\lambda\sum_iZ_i}$ and the tower property, for all $1\le m\le\eta n/(2\Delta b_R)$,
\begin{align}
 Q\big(|\mathcal C|\ge m\mid\mathcal R_j=h\big)
 &\le Q\Big(\sum_{i=1}^{1+dm}Z_i\ge(1-8\delta\Delta^2b_R^2)m-1-4\Delta b_R\ \Big|\ \mathcal R_j=h\Big)\nonumber\\
 &\le e^{-\lambda[(1-8\delta\Delta^2b_R^2)m-1-4\Delta b_R]}\,
     \E_Q\Big[e^{\lambda\sum_{i=1}^{1+dm}Z_i}\ \Big|\ \mathcal R_j=h\Big]\nonumber\\
 &\le e^{-\lambda[(1-8\delta\Delta^2b_R^2)m-1-4\Delta b_R]}\,(1-s+se^\lambda)^{1+dm}\nonumber\\
 &= e^{\lambda(1+4\Delta b_R)}(1-s+se^\lambda)\cdot
    \exp\Big(-m\big[(1-8\delta\Delta^2b_R^2)\lambda-d\log(1-s+se^\lambda)\big]\Big)\nonumber\\
 &= C_1e^{-a_0m},\label{eq:component-tail}
\end{align}
where $a_0:=(1-8\delta\Delta^2b_R^2)\lambda-d\log(1-s+se^\lambda)>0$ by~\eqref{eq:tail-constants}, and $C_1:=e^{\lambda(1+4\Delta b_R)}(1-s+se^\lambda)$.

The case of~$\eta n/(2\Delta b_R)<m\le n$ follows by \eqref{eq:component-tail} with a change of the constant $a_0 \mapsto a_0 \eta / (2\Delta b_R)$,  proving~\eqref{eq:component-tail-statement} for all~$m\le n$.
Finally, by~\eqref{eq:component-tail-statement},
\begin{align}
 \E_Q\big[e^{c_0|\mathcal C|}\one\{|\mathcal C|>a\}\mid\mathcal R_j=h\big]
 \le\sum_{m>a}e^{c_0m}\,Q\big(|\mathcal C|\ge m\mid\mathcal R_j=h\big)
 \le C_{1}\sum_{m>a}e^{-c_0m}
 \le C_{\rm tail}\,e^{-c_0a}\,,\label{eq:truncated-moment}
\end{align}
after choosing~$C_{\rm tail}$ to be a large constant if necessary.
\end{proof}

\subsection{Proofs of the two main lemmas}\label{sec:conclusions}
We conclude Section~\ref{sec:key-stationary-estimate} with the proofs of Lemmas~\ref{lem:entropy-mass} and~\ref{lem:fk-source},
by combining the exponential tail bounds under the $Q$ revealing (Lemma~\ref{lem:exploration-tails}) with the KL-divergence comparison between $P$ and $Q$ (Lemma~\ref{lem:exploration-comparison}, in the form of Corollary~\ref{cor:entropy-transfer}).

We define a graph event that guarantees most tests along the revealing are good tests, and show it holds with high probability for random $\Delta$-regular graphs.  
 
\begin{lemma}\label{lem:poly-geometry}
There exist constants~$\eta\in(0,1/2]$ and~$\alpha,\zeta>0$ such that, for every fixed~$K_w>0$, there is a fixed integer~$t_w$ for which a random~$\Delta$-regular graph $G$ satisfies the following properties with probability~$1-o(1)$:
\begin{enumerate}
\item $|E(G[U])|\le(1+\delta)|U|$ for every~$U\subseteq V$ with~$|U|\le\eta n$;
\item $|E(G[U])|-|U|+1\le t_w$ for every connected~$U\subseteq V$ with~$|U|\le K_w\log n$;
\item all simple cycles of length at most~$\ell=\lfloor\alpha\log n\rfloor$ are pairwise vertex-disjoint;
\item writing~$N_j(G)$ for the number of cycles of length~$j$,
\begin{equation}\label{eq:T}
 T(G):=\sum_{j=\ell+1}^{\lceil K_w\log n\rceil}
           j^2N_j(G)s^{\gamma(j-1)}\le n^{-\zeta}\,.
\end{equation}
\end{enumerate}
The constants~$\eta,\alpha,\zeta$ depend only on~$\Delta,\delta,s,\gamma$; the bound~$t_w$ may additionally depend on~$K_w$.
\end{lemma}
\begin{proof}[Proof of Lemma~\ref{lem:poly-geometry}]
Property~1 follows e.g., from~\cite[Theorem~4.16]{HLW06}, applied with parameter~$2\delta$, for some sufficiently small~$\eta>0$.
For property~2, apply the configuration-model counting estimate
in~\cite[proof of Theorem~4.16, pp.~479--480]{HLW06} to sets
of~$s$ vertices spanning at least~$s+t_w$ edges. This gives
\begin{equation}\label{eq:log-excess-error}
 \begin{aligned}
 \PP\bigl(\exists U:\ 1\le |U|\le K_w\log n,\,
                      |E(G[U])|\ge |U|+t_w\bigr)
 &= O\left(
 \sum_{1\le s\le K_w\log n}
 (e^2\Delta)^s
 \left(\frac{e\Delta s}{n}\right)^{t_w} \right)\\
 &=O\left((\log n)^{t_w+1}
 n^{K_w\log(e^2\Delta)-t_w}\right)=o(1),
 \end{aligned}
\end{equation}
where we choose a fixed integer~$t_w>K_w\log(e^2\Delta)$.
For connected~$U$, failure of property~2 implies
$|E(G[U])|\ge |U|+t_w$, proving the claim.

For Property 3, choose~$\alpha\le1/(5\log d)$. It then follows from~\cite[Lemma~4.1]{BGGSVY20}. For Property 4, the cycle-count estimate in the proof of~\cite[Lemma~4.2]{BGGSVY20} gives~$\E[N_j(G)]\le C_\Delta d^j/j$ uniformly for~$3\le j\le\lceil K_w\log n\rceil$. Hence
\begin{equation}\label{eq:T-expectation}
 \E[T(G)]
 \le C_\Delta s^{-\gamma}\sum_{j>\ell}j(ds^\gamma)^j
 \le \frac{C_\Delta s^{-\gamma}(\ell+1)}{(1-ds^\gamma)^2}\,(ds^\gamma)^{\ell+1}
 \le \frac{C_\Delta(\alpha\log n+1)}{s^\gamma(1-ds^\gamma)^2}\,n^{-\alpha\log(1/(ds^\gamma))}\,.
\end{equation}
Choose any fixed~$0<\zeta<\alpha\log(1/(ds^\gamma))$. Markov's inequality proves property~4. The choices of~$\eta,\alpha,\zeta$ are independent of~$K_w$, and intersecting the four events proves the lemma.
\end{proof}

For  the proofs of Lemmas~\ref{lem:entropy-mass} and~\ref{lem:fk-source},
given constants $c_0$ from Lemma~\ref{lem:exploration-tails} and $\alpha$ from Lemma~\ref{lem:poly-geometry}, 
we choose constants~$C_0>\max\{7/c_0,\alpha\}$ and~$K_w>(2\Delta b_R) C_0$, and put the cutoff $ L= C_0\log n$ for component sizes in the quantity $W_L$. 
Then choose~$t_w$ using Lemma~\ref{lem:poly-geometry} for this value of~$K_w$. By these choices, we have~$\ell<L$ and~$(2\Delta b_R) L\le K_w\log n$. 
In the remaining proofs, fix any~$G$ satisfying the properties of Lemma~\ref{lem:poly-geometry}, which holds with probability $1-o(1)$.

\begin{proof}[\textbf{\emph{Proof of Lemma~\ref{lem:entropy-mass}}}]
For an arbitrary spin distribution $\nu$, define the exploration processes $P$ and $Q$ as in Definitions~\ref{def:P-revealing-distribution}--\ref{def:Q-revealing-distribution}. 
By Corollary~\ref{cor:entropy-transfer} with $\lambda =c_0$, 
\begin{align}\label{eq:W_L-comparison-bound}
	\mathbb E_{\nu,M} [W_L(H)]  = \mathbb E_{P}[W_L] \le \frac{1}{c_0} C_{\rm obs} I_\mu(\nu) + \frac{1}{c_0} \log \mathbb E_{Q} [e^{c_0 W_L}]\,.
\end{align}
For the exploration under $Q$,
let $u_1,\ldots,u_K$ be the vertices chosen at line~3 of Algorithm~\ref{alg:exploration-reza}, in the order they are chosen, so that $\mathcal C_{u_1},\ldots,\mathcal C_{u_K}$ are the components in the order they are explored; since they are disjoint and nonempty, $K\le n$. For $1\le i\le K$, let $\mathcal R^{(i)}$ be the record when line~3 is executed for the $i$-th time, i.e., just before the exploration of $\mathcal C_{u_i}$ starts. For $K<i\le n$, we use the convention $\mathcal C_{u_i}=\emptyset$ and $\mathcal R^{(i)}=\mathcal R$.
For $i\le n$, Lemma~\ref{lem:exploration-tails} gives
\begin{equation}\label{eq:large-one-moment}
\E_Q\big[e^{c_0|\mathcal C_{u_i}|\one\{|\mathcal C_{u_i}|>L\}}\,\big|\,\mathcal R^{(i)}\big]
 \le 1+ C_{\rm tail} e^{-c_0L}\,.
\end{equation}
Since~$W_L=\sum_{i=1}^n|\mathcal C_{u_i}|\one\{|\mathcal C_{u_i}|>L\}$, conditioning successively on $\mathcal R^{(n)},\mathcal R^{(n-1)},\ldots,\mathcal R^{(1)}$ and applying~\eqref{eq:large-one-moment} at each step gives
\begin{equation}\label{eq:large-moment}
 \log\E_Qe^{c_0W_L} =\log\E_Q\Big[\prod_{i=1}^{n} e^{c_0 |\mathcal C_{u_i}| \one\{|\mathcal C_{u_i}|>L\}}\Big]
 \le n\log(1+ C_{\rm tail} e^{-c_0L})\le C_{\rm tail} ne^{-c_0L}\,.
\end{equation}
By our choice of $C_0$ and $L$,  the right-hand side of \eqref{eq:large-moment} is at most $c_0 n^{-5}$ for large $n$, and thus the proof is complete by \eqref{eq:W_L-comparison-bound}. 
\end{proof}

\begin{remark}\label{rmk: weak prob bound} Lemma~\ref{lem:info:bound} follows from the coupling~$\omega\subseteq H$. Indeed, if~$\mathrm{Bad}_L$ denotes the event that a component has more than~$L$ vertices, then
\begin{equation}\label{eq:ES-large-consequence}
 \PP_{\nu,\ES}(\mathrm{Bad}_L)
 \le\PP_{\nu,\rm{M}}(\mathrm{Bad}_L)
 \le\E_{\nu,\rm{M}}[W_L(H)]\,.
\end{equation}
\end{remark}

The proof for the edge revealing which bounds weighted long cycle counts is more involved and is where the last three properties of Lemma~\ref{lem:poly-geometry} are additionally needed.  

\begin{proof}[\textbf{\emph{Proof of Lemma~\ref{lem:fk-source}}}]
For an arbitrary edge distribution~$\nu$, define the exploration processes $P$ and $Q$ as in Definitions~\ref{def:P-revealing-distribution}--\ref{def:Q-revealing-distribution}. Let $W = W_{\rm long}$ from~\eqref{eq:fk-long-mass}. We will bound~$\log\E_Q[ e^{(1-\gamma)c_0W} ]$ by a quantity tending to zero, as then we can apply Corollary~\ref{cor:entropy-transfer} with $\lambda=(1-\gamma)c_0$. The idea is that the contribution to $W$ from components containing $O(\log n)$ sized cycles is bounded by properties 3--4 and in particular, the bound of~\eqref{eq:T} in Lemma~\ref{lem:poly-geometry}; the contribution from components larger than $L$ are polynomially small by the tail bound used already in~\eqref{eq:large-one-moment}--\eqref{eq:large-moment} in the spin case. 

We begin by bounding the probability that a certain prescribed set of edges $J$ are all discovered to be present under $Q$. 
Conditioned on a history $\mathcal R$ that concludes the revealing of a component, stop the next component's exploration when it completes or first reaches~$L$ vertices. Its radius-$(R+1)$ neighborhood in $G$ has at most~$(2\Delta b_R) L\le K_w\log n$ vertices. Property~2 of Lemma~\ref{lem:poly-geometry} and the same argument as for~\eqref{eq:bad-count}, imply that the number of successful tests that are not good in the revealing of the next component is at most $4\Delta b_R t_w$.

Fix a deterministic set~$J\subset E$ of size $m$. To bound the probability that all of $J$ are discovered to be present, start a nonnegative process~$(Z_i)$ at $Z_0=1$. Whenever a good test examines an edge of~$J$, set~$Z_{i+1} =Z_i/s$ if the edge is present, and $Z_{i+1}=0$ if it is not; in all other cases (the test is not good, its edge is not in~$J$, or the exploration has stopped), let $Z_{i+1}=Z_i$. Thus, if $\mathcal R_i$ consists of the history $\mathcal R$ and the history generated by the first $i$ tests of the new component being revealed,
\begin{equation}\label{eq:required-supermartingale}
 \E_Q[Z_{i+1}\mid \mathcal R_i]
 =Z_i \,
   \frac{Q(\text{ test edge $i+1$ is present} \mid \text{test $i+1$ is good}\, ,\, \mathcal R_i)}{s}
 \le Z_{i}
\end{equation}
on the event that $i<\tau$, test~$i+1$ is good and its edge lies in~$J$; otherwise $Z_{i+1}=Z_i$.

If every edge of~$J$ is discovered to be present in the component, at least~$m-4 \Delta b_R t_w$ of them were tested at good tests. Let~$\tau\le|E|$ be the number of tests until the exploration stops. Then $Z_\tau\ge s^{-(m-4\Delta b_R t_w)}$. Since~$\tau$ is a bounded stopping time, optional stopping and Markov's inequality give 
\begin{equation}\label{eq:required-probability}
 Q(J\text{ are discovered to be present by the stopping time}
      \mid \mathcal R)
 \le s^{m-4 \Delta b_R t_w}\,.
\end{equation}

We now wish to sum over possible candidate $J$'s that witness long cycles in the induced graphs of components. 
Suppose a completed component~$U$ has~$|U|\le L$ and the induced graph~$G[U]$ contains one simple cycle longer than~$\ell = \alpha \log n$. Its discovered edges along the exploration form a spanning tree~$\mathcal T$ of~$G[U]$. By Property~3 of Lemma~\ref{lem:poly-geometry}, at least one cycle containing one edge on top of edges of $\mathcal T$ is longer than $\ell$. 

Choose such a long fundamental cycle of length~$j>\ell$ and omit its unique edge outside~$\mathcal T$. Let~$r$ be the length of the path in~$\mathcal T$ from the seed~$v$ of the component to its first vertex on this cycle. Together, these specify $j+r-1$ edges that are discovered to be present in the connected component, with 
\begin{equation}\label{eq:witness-sizes}
 j>\ell,\qquad j+r\le |U|\le L\,.
\end{equation}
If these $j+r-1$ edges are all discovered to be present in the component, they witness a component $U$ contributing to the sum in the definition of $W$. 

Let~$\mathcal W(v)$ be the set of all choices of a simple length-$j$ cycle, an omitted cycle edge, and a simple length-$r$ path from~$v$ first meeting the cycle at its last vertex, with~$j>\ell$ and~$j+r\le L$. Define
\begin{equation}\label{eq:witness-weight}
 w(v)=s^{-4 \Delta b_R t_w}\sum_{\mathcal W(v)}s^{j+r-1}\,.
\end{equation}
By~\eqref{eq:required-probability},~$w(v)$ bounds the conditional probability that the component seeded at~$v$ has size at most~$L$ and whose induced graph contains a simple cycle longer than~$\ell$, uniformly over the preceding components.

Starting from any length-$j$ cycle in $G$, there are at most $j^2 d^r$ many summands in $\sum_{v} \sum_{\mathcal W(v)}$; $j$ for the missing edge from the cycle, $j$ for the vertex $u$ of the cycle to which the path of length $r$ connects, and $d^r$ for the $v$ to $u$ path. 
Using~$(\sum_i z_i)^\gamma\le\sum_i z_i^\gamma$ for~$z_i\ge0$, then dropping the restriction~$j+r\le L$, gives
\begin{equation}\label{eq:witness-sum}
 \begin{aligned}
 \sum_v w(v)^\gamma
 &\le C\sum_{j=\ell+1}^{L}j^2N_j(G)
                  \sum_{r\ge0}d^r s^{\gamma(j+r-1)} \le\frac{C}{1-ds^\gamma}
       \sum_{j=\ell+1}^{L}j^2N_j(G)s^{\gamma(j-1)}
 \le C T(G)\,,
 \end{aligned}
\end{equation}
where $C$ only depends on the fixed quantities.

Let~$\mathcal R$ be the history immediately before the current component starts, and let~$\mathcal E$ be the event that this component, with seed~$v$ and size~$S$, has a simple cycle longer than $\ell$ in its induced graph. H\"older's inequality and~\eqref{eq:component-moment-statement} with $a=0$ gives
\begin{equation}\label{eq:holder}
 \begin{aligned}
 &\E_Q[e^{(1-\gamma)c_0S}\one\{\mathcal E,S\le L\}
              \mid\mathcal R]\le
 \bigl(\E_Q[e^{c_0S}\mid\mathcal R]\bigr)^{1-\gamma}
 Q(\mathcal E,S\le L\mid\mathcal R)^\gamma
 \le C_{\rm tail} w(v)^\gamma\,.
 \end{aligned}
\end{equation}
For~$S>L$,~\eqref{eq:component-moment-statement} with $a=L$ bounds the corresponding expectation by~$C_{\rm tail}e^{-c_0L}$. Consequently,
\begin{equation}\label{eq:cycle-one-moment}
 \E_Q[e^{(1-\gamma)c_0S\one_{\mathcal E}}
          \mid\mathcal R]
 \le1+C_{\rm tail}\bigl(w(v)^\gamma+e^{-c_0L}\bigr)\,.
\end{equation}

Let~$S_i$ and~$v_i$ be the size and the seed of the $i$-th revealed component, let~$\mathcal E_i$ be its long-cycle event, and put~$Y_i=S_i\one_{\mathcal E_i}$.
Put
\begin{equation}\label{eq:predictable-bound}
 A_i=C_{\rm tail}\bigl(w(v_i)^\gamma+e^{-c_0L}\bigr)\,.
\end{equation}
This quantity is $\mathcal R$-measurable, i.e., it is known before the $i$-th component is explored. Iterating~\eqref{eq:cycle-one-moment}, using~$(1+x)e^{-x}\le1$, gives
\begin{equation}\label{eq:compensated-moment}
 \E_Q\Big[\exp\Big((1-\gamma)c_0\sum_iY_i-\sum_iA_i\Big)\Big]
 \le1\,.
\end{equation}
The bound of~\eqref{eq:witness-sum} gives the deterministic bound~$\sum_iA_i\le C T(G)+Cne^{-c_0L}$. Since~$\sum_iY_i=W$, we conclude that
\begin{equation}\label{eq:long-moment}
 \log\E_Q[e^{(1-\gamma)c_0W}]
 \le C T(G)+Cne^{-c_0L}
 \le C(n^{-\zeta}+n^{-6})\,.
\end{equation}
Here, the last inequality used property 4 of Lemma~\ref{lem:poly-geometry}. 

Finally, applying~\eqref{eq:entropy-transfer} and Lemma~\ref{lem:exploration-comparison} gives
\begin{equation}\label{eq:final-constants}
 \begin{gathered}
 \E_\nu [W]\le \frac{C_{\rm obs}}{(1-\gamma)c_0} I_\phi(\nu)+C(n^{-\zeta}+n^{-6})\,,
 \end{gathered}
\end{equation}
as desired.
\end{proof}

\section{Analysis of the Markov chain $M$}\label{sec:M-FK-comparison}
We will study spectral gap and entropy contraction for the Markov chain $M^{(r)}$ from Definition~\ref{def:M-SW-chain} through comparison with the FK dynamics, whose transition matrix is denoted by $P_{\rm FK}$.

In this section, operators act on functions: a Markov kernel $P$ from $\Omega_1$ to
$\Omega_2$ acts by $(Pg)(x)=\sum_y P(x,y)g(y)$, so composing kernels
in the order they are applied corresponds to composing operators in the
same order. We write $\langle\cdot,\cdot\rangle_\mu$,
$\langle\cdot,\cdot\rangle_\pi$, $\langle\cdot,\cdot\rangle_\phi$,
$\|\cdot\|_\mu$, etc.\ for the inner products and norms of $L^2(\mu)$,
$L^2(\pi)$, $L^2(\phi)$. 
On the triple space of $(\sigma, A, N) \in [q]^V \times \{0,1\}^E \times \{0,1\}^E$ we use the random variables
\[
 H:=A\cup N,\qquad Y:=\bigl(N,\,A_{E\setminus N}\bigr).
\]
Note that $H=N\cup\{e\in E\setminus N:\ A_e=1\}$ is a function of $Y$.
We consider $\pi$ to be the stationary joint law of $(\sigma,A,N)$, where
$(\sigma,A)$ is distributed according to the joint law $\mathbb P_{\mu,\rm SW}$ and
$N\sim\operatorname{Ber}(r)^{\otimes E}$ is an independent edge percolation. 
In this section, $r\in (0,1)$ can be arbitrary, $M = M^{(r)}$ and $\tilde B= \tilde B(r)=1+(B-1)/r$.
   In Sections~\ref{sec:M-FK-comparison} and~\ref{sec:fk-entropy}, write $m=|E|$.  
\subsection{Spectral Gap  for \texorpdfstring{$M$}{M}}
\label{sec:ums}

\begin{lemma}\label{lem:ums-M}
Suppose there exists a constant $C_{\rm gap}\ge1$ such that on any graph $G=(V,E)$, $\gap(P_{\rm FK})\ge1/(C_{\rm gap}m)$. Then for any $r\in (0,1)$, 
\[
 \gap(M)\ge\frac{r}{2(C_{\rm gap}+1)}.
\]
\end{lemma}

We first define two auxiliary chains. For $R\subseteq E$ and a random-cluster configuration $A$, we use the
notation $A_R\in \{0,1\}^{R}$ for the random-cluster configuration that agrees with $A$ but is restricted to $R$.
Also recall that we use $\phi$ for the random-cluster distribution.

For convenience, let $s =1-r$. Let $K_s$ be a Markov chain on random-cluster configurations that 
transitions from a configuration $A$ to $A'$ as follows:
\begin{enumerate}
\item Include each edge $e \in E$ in $R$ independently with probability $s$;
\item Draw $A'\sim\phi(\,\cdot\mid A'_R=A_R)$.
\end{enumerate}
Let $\mathcal B_s$ be a Markov chain on Potts configurations that 
from a spin configuration $\sigma$:
\begin{enumerate}
\item Open each monochromatic edge independently with probability $p$ to obtain $A$;
\item Apply one $K_s$ update to $A$, obtaining $A'$;
\item Give each connected component in $A'$ an independent uniform color in $[q]$, obtaining the new spin configuration $\sigma'$.
\end{enumerate}

We give the following facts for these chains and comparisons between them. 
\begin{fact}
\label{fact:1}
For $r\in (0,1)$, $\mathcal B_{1-r}$ is ergodic, positive semidefinite and reversible w.r.t.~$\mu$.
Moreover, $\gap(M)\ge\gap(\mathcal B_{1-r})$.
\end{fact}
\begin{fact} 
\label{fact:2}
 For $s \in (0,1)$, $K_s$ is ergodic, positive semidefinite, and reversible w.r.t.~$\phi$.
 Moreover, 
$\gap(\mathcal B_s)\ge\gap(K_s)$ for every $0\le s\le1$.
\end{fact}
\begin{fact}
\label{fact:3}
Suppose $\gap(P_{\rm FK})\ge1/(C_{\rm gap}m)$,
where $C_{\rm gap}\ge1$. Then, for every $0\le s\le1$
$$\gap(K_s)\ge \frac{(1-s)|E|}{(C_{\rm gap}+1)(|E|+1)}.$$
\end{fact}

Lemmas~\ref{fact:1}, \ref{fact:2} and \ref{fact:3} together give precisely
\[
 \gap(M)\ge\gap(\mathcal B_{1-r})
 \ge\gap(K_{1-r})
 \ge\frac{r|E|}{(C_{\rm gap}+1)(|E|+1)} \ge \frac{r}{2(C_{\rm gap}+1)},
\]
as claimed in Lemma~\ref{lem:ums-M}. 
Next we will prove these facts.  
To facilitate the comparisons, we define the following Markov operators:

\begin{enumerate}
    \item 
$\mathsf P$:
Starting from a spin configuration $\sigma$, open each monochromatic edge independently with
probability $p$, close the disagreeing edges, and independently draw $N\sim\operatorname{Ber}(r)^E$. Return the triple $(\sigma,A,N)$.
The adjoint $\mathsf P^*$ maps the triple $(\sigma,A,N)$
to $\sigma$.

\item $L$:
Starting from $(\sigma,A,N)$, keep only the union $H=A\cup N$.
Draw a fresh triple $(\sigma',A',N')$ from $\pi$ conditional on
$A'\cup N'=H$. Its spin coordinate has law $\mu_H$.

\item $\mathcal L$:
Starting from $(\sigma,A,N)$, keep $N$ and the open and closed states of every edge in $R=E\setminus N$. Draw
$A'\sim\phi(\,\cdot\mid A'_R=A_R)$, color the $A'$-components
independently and uniformly to obtain $\sigma'$, and set $N'=N$.

\item $\mathcal C$: Starting from a spin configuration $\sigma$, open each edge whose endpoints have the same color independently with probability $p$, and close every other edge.
Let
$\mathcal C^*$ denote its adjoint that given an edge configuration $A$, independently assign each connected component a uniform color.
\end{enumerate}

Observe that
under $\mathsf P$, given $\sigma$, each edge lies in $H=A\cup N$
independently, with probability $p+(1-p)r=r\tilde B/B$ if $\sigma_u=\sigma_v$,
and with probability $r$ otherwise, which is the first
step of $M$. By~\eqref{eq: reversibility}, the spin coordinate of
$\pi(\,\cdot\mid H)$ has law $\mu_H$. Hence $M=\mathsf PL\mathsf P^*$.
Under $\mathsf P$, the set $R=E\setminus N\sim\operatorname{Ber}(1-r)^E$
is independent of $(\sigma,A)$. Then $\mathcal L$ performs
$A'\sim\phi(\,\cdot\mid A'_R=A_R)$ and recolors, so
$\mathcal B_{1-r}=\mathsf P\mathcal L\mathsf P^*$. Finally
$\mathcal B_s=\mathcal CK_s\mathcal C^*$ by definition, and one SW move
is $\mathcal C\mathcal C^*$.

Furthermore, we also have the following fact about the adjoint of $L$.
 
\begin{fact}\label{fact:4}
$L=L^*$, $L=L^2$, and $L\mathcal L=\mathcal LL=L$.
In particular, $L=L\mathcal LL^*$.
\end{fact}
\begin{proof}
By definition, $L$ redraws the triple from $\pi(\,\cdot\mid A'\cup N'=H)$,
so $Lg=\E_\pi[g\mid H]$. We claim that also $\mathcal Lg=\E_\pi[g\mid Y]$.
Under $\pi$, $N$ is independent of $(\sigma,A)$, and $A_R$ is a function
of $A$. Hence
\[
 \pi\bigl(\sigma,A,N\mid N, A_R\bigr)
 =\phi(A\mid A_R)\,\pi_{\ES}(\sigma\mid A),
\]
which is exactly
the output law of $\mathcal L$ started from $Y$,
and it proves the claim.
 
Thus $L$ and $\mathcal L$ are conditional expectations on $L^2(\pi)$, with
respect to $\sigma(H)$ and $\sigma(Y)$. A conditional expectation
$\Pi=\E_\pi[\,\cdot\mid\mathcal G]$ is an orthogonal projection. It is
self-adjoint, because
$\langle g,\Pi h\rangle_\pi=\E_\pi\bigl[\E_\pi[g\mid\mathcal G]\,
\E_\pi[h\mid\mathcal G]\bigr]=\langle\Pi g,h\rangle_\pi$. It is
idempotent, because $\E_\pi[\E_\pi[g\mid\mathcal G]\mid\mathcal G]
=\E_\pi[g\mid\mathcal G]$. This gives $L=L^*$ and $L=L^2$, and also
$\mathcal L=\mathcal L^*=\mathcal L^2$.
 
Since $H=N\cup\{e\in E\setminus N:\ A_e=1\}$ is a function of $Y$, we
have $\sigma(H)\subseteq\sigma(Y)$. The tower property gives
$L\mathcal Lg=\E_\pi[\E_\pi[g\mid Y]\mid H]=\E_\pi[g\mid H]=Lg$.
Also $Lg$ is $\sigma(H)$-measurable, hence $\sigma(Y)$-measurable, so
$\mathcal LLg=Lg$. Finally $L\mathcal LL^*=L\mathcal LL=LL=L$.
\end{proof}

Now we provide the proofs of Lemmas~\ref{fact:1}, \ref{fact:2} and \ref{fact:3}. 
\begin{proof}[Proof of Lemma~\ref{fact:1}]
Let $f\in L^2(\mu)$ and $g=\mathsf P^*f$, i.e.\ $g(\sigma,A,N)=f(\sigma)$.
Since $\mu(\sigma)\mathsf P(\sigma;A,N)=\pi(\sigma,A,N)$, for any function
$h\in L^2(\pi)$,
\[
 \langle f,\mathsf Ph\rangle_\mu
 =\sum_{\sigma,A,N}\pi(\sigma,A,N)\,f(\sigma)\,h(\sigma,A,N)
 =\langle g,h\rangle_\pi .
\]
Taking $h=\mathcal L g$ and using $\mathcal B_{1-r}=\mathsf P\mathcal L\mathsf P^*$
together with $\mathcal L=\mathcal L^2=\mathcal L^*$ (Lemma~\ref{fact:4}),
\[
 \langle f,\mathcal B_{1-r}f\rangle_\mu
 =\langle f,\mathsf P\mathcal Lg\rangle_\mu
 =\langle g,\mathcal Lg\rangle_\pi
 =\langle g,\mathcal L^*\mathcal Lg\rangle_\pi
 =\langle \mathcal Lg,\mathcal Lg\rangle_\pi
 =\|\mathcal Lg\|_\pi^2 .
\]
Thus, $\mathcal B_{1-r}=\mathsf P\mathcal L\mathsf P^*$ is self-adjoint,
with $\langle f,\mathcal B_{1-r}f\rangle_\mu \ge0$.
It is also ergodic, since $\mathcal B_{1-r}(\sigma,\cdot)\ge r^m\mu(\cdot)$.
In the same way, taking $h=Lg$ and using $M=\mathsf PL\mathsf P^*$ and $L=L^2=L^*$,
\[
 \langle f,Mf\rangle_\mu
 =\langle g,Lg\rangle_\pi
 =\|Lg\|_\pi^2 .
\]
Since $L=L\mathcal L$ and the projection $L$ is a contraction,
\[
 \langle f,Mf\rangle_\mu=\|L\mathcal Lg\|_\pi^2
 \le\|\mathcal Lg\|_\pi^2=\langle f,\mathcal B_{1-r}f\rangle_\mu .
\]
Equivalently, $\cE_M(f)\ge\cE_{\mathcal B_{1-r}}(f)$ for every $f$.

Both chains are reversible w.r.t.~$\mu$, so
$\gap(M)\ge\gap(\mathcal B_{1-r})$.
\end{proof}
 
\begin{proof}[Proof of Lemma~\ref{fact:2}]
Throughout this proof, $h$ denotes a function in $L^2(\phi)$ of the random-cluster
configuration $A\in\{0,1\}^E$, and $f$ denotes a function in $L^2(\mu)$ of
the spin configuration $\sigma\in[q]^V$.

For $R\subseteq E$ define $\Pi_R:L^2(\phi)\to L^2(\phi)$ by
$\Pi_R h=\E_\phi[h\mid A_R]$, the operator that redraws $A$ from $\phi$
given $A_R$. By definition of $K_s$,
\[
 K_s=\sum_{R\subseteq E}s^{|R|}(1-s)^{m-|R|}\,\Pi_R .
\]
Each $\Pi_R$ is an orthogonal projection on $L^2(\phi)$ so $K_s$ is a convex combination of self-adjoint
positive semidefinite operators, hence itself self-adjoint and positive
semidefinite.
 It is also ergodic for $s<1$, since 
$K_s(A,\cdot)\ge(1-s)^m\phi(\cdot)$.
Consequently, for every $h\in L^2(\phi)$ with $\E_\phi h=0$,
\[
 \cE_{K_s}(h)=\|h\|_\phi^2-\langle h,K_sh\rangle_\phi\le\|h\|_\phi^2,
\]
so $0\le\gap(K_s)\le1$.

Recall that $\mathcal C^*:L^2(\mu)\to L^2(\phi)$ is given by
$(\mathcal C^*f)(A)=\E_{\pi_{\ES}}[f(\sigma)\mid A]$, the average of $f$
over independent uniform colourings of the components of $A$, and that
$\mathcal C:L^2(\phi)\to L^2(\mu)$ is its adjoint.

Let $f\in L^2(\mu)$ with $\E_\mu f=0$, and set $h=\mathcal C^*f\in L^2(\phi)$.
Then $\E_\phi h=\E_\mu f=0$, and by Jensen's inequality
$\|h\|_\phi^2=\sum_A\phi(A)\bigl(\E_{\pi_{\ES}}[f\mid A]\bigr)^2
\le\|f\|_\mu^2$. Since $\mathcal B_s=\mathcal CK_s\mathcal C^*$,
\[
 \begin{aligned}
 \cE_{\mathcal B_s}(f)
 &=\|f\|_\mu^2-\langle f,\mathcal CK_s\mathcal C^*f\rangle_\mu
 =\|f\|_\mu^2-\langle h,K_sh\rangle_\phi
 =\|f\|_\mu^2-\|h\|_\phi^2+\cE_{K_s}(h)\\
 &\ge\|f\|_\mu^2-\bigl(1-\gap(K_s)\bigr)\|h\|_\phi^2
 \ \ge\ \gap(K_s)\,\|f\|_\mu^2,
 \end{aligned}
\]
where the last step uses $1\ge \gap(K_s)$ and
$\|h\|_\phi\le\|f\|_\mu$. 
Therefore, we prove that $\gap(\mathcal B_s)\ge\gap(K_s)$. 
\end{proof}

\begin{proof}[Proof of Lemma~\ref{fact:3}]
Let $G_k$ choose $k$ edges from $E$ uniformly and jointly
resample their states conditional on the other bonds.
Let $G_0=I$.
Lee's discrete-time block comparison~\cite[Corollary~3.6]{Lee24},
applied to the random-cluster law on $\{0,1\}^{E}$, gives, for $1\le k\le |E|$,
\[
 \gap(G_k)
 \ge1-\left(1-\frac{k}{|E|+1}\right)^{1/(C_{\rm gap}+1)}
 \ge\frac{k}{(C_{\rm gap}+1)(|E|+1)}.
\]
The last inequality follows from the fact that $(1-x)^a\le 1-ax$ for $0<a,x<1$.

Define
$J:=|E\setminus R|$, the number of resampled edges in one step of $K_s$.
Each edge is resampled independently with probability $1-s$, so
$J\sim\operatorname{Bin}(|E|,1-s)$. Conditional on $J=k$, the set
$E\setminus R$ is uniform among the $k$-element subsets of $E$.
Hence
\[
 K_s= \sum_{k=0}^m\PP(J=k)G_k = \E[G_J],
\]
and for every function $h:\{0,1\}^{E}\rightarrow\mathbb R$ such that $\E_\phi h=0$,
\[
 \begin{aligned}
 \cE_{K_s}(h)
 &=\langle h,(I-K_s)h\rangle_\phi
 =\sum_{k=0}^m\PP(J=k)\langle h,(I-G_k)h\rangle_\phi
 =\sum_{k=0}^m\PP(J=k)\cE_{G_k}(h)
 =\E[\cE_{G_J}(h)]\\
 &\ge \E[\gap(G_J)] \langle h,h\rangle_\phi
 \ge \E\left[\frac{J}{(C_{\rm gap}+1)(|E|+1)}\right]\langle h,h\rangle_\phi
 \ge\frac{(1-s)|E|}{(C_{\rm gap}+1)(|E|+1)}\langle h,h\rangle_\phi.
 \end{aligned}
\]
Thus, we establish the lower bound for $\gap(K_s)$.
\end{proof}

\subsection{Entropy contraction for \texorpdfstring{$M$}{M}}
\label{sec:ums-entropy}
In this section we prove the entropy contraction inequality as in Lemma~\ref{lem:entropy-M}. 
For this, we assume that approximate tensorization of entropy holds for the random-cluster measure $\phi$, which is a stronger assumption than even the optimal mixing time of the FK dynamics, and we will prove this assumption in Section~\ref{sec:fk-entropy}.
Our main result of this section is the following theorem.

\begin{theorem}
\label{thm:ums-M-entropy}
Suppose on a $\Delta$-regular graph $G=(V,E)$, there is a constant $C_{\rm ent}>1$ such that for every function $h:\{0,1\}^E\rightarrow \mathbb R_{\ge 0}$, we have 
\begin{equation}\label{eq:ums-fk-entropy}
 \Ent_\phi(h)\le C_{\rm ent}\sum_{e\in E}
       \E_\phi\bigl[\Ent(h\mid A_{-e})\bigr],
\end{equation}
where $\Ent(h\mid A_{-e})$ denotes the entropy of $h$ under $\phi(\,\cdot\mid A_{-e})$.
Then, for every $f:[q]^V \rightarrow \mathbb R_{\ge 0}$,
we get that
\begin{equation}\label{eq:ums-M-entropy-margin}
 \E_{H\sim\PP_{\mu, \rm{M}}}\Ent_{\mu_H}(f)
 \ge\frac{rm}{(C_{\rm ent}+1)(m+1)}\Ent_\mu(f).
\end{equation}
\end{theorem}

Similar to the proof of Lemma~\ref{lem: gap of M}, the proof of Theorem~\ref{thm:ums-M-entropy} follows from three comparisons of Markov operators. 
However, it is easier to perform the comparisons in terms of  entropy of operators corresponding to the ``half steps'' of the Markov chains $M, K_s $ and $\mathcal{B}_s$.  

We review and set up some notations. Throughout this section, we fix a graph $G$ that satisfies \eqref{eq:ums-fk-entropy}, we let $h:\{0,1\}^E\rightarrow \mathbb R_{\ge 0}$ be a function with $\E_{\phi} h =1$, and $f:[q]^V \rightarrow \mathbb R_{\ge 0}$ be a function with $\E_\mu f =1$.
By scaling, it is sufficient to work with these normalized functions.
Recall the operators $\mathsf{P}, L, \mathcal{L}$, $\mathcal{C}$ and their adjoints from  Section~\ref{sec:ums}; they are the half-step operators used to compose the Markov updates.
Moreover, we write
\[
 F_s(h)=\E_{R\sim\operatorname{Ber}(s)^E}
       \Ent_\phi\bigl(\E_\phi[h\mid A_R]\bigr).
\]
This is the relative entropy after the first half step of a $K_s$ update under the observation $(R, A_R)$.
Lastly, recall that $\mathcal{C}^*f$ is seen as a function on $\{0,1\}^E$ given by $(\mathcal C^*f)(A)=\E_{\pi_{\ES}}[f(\sigma)\mid A]$.

Then we have the following facts for comparisons. 

\begin{fact}\label{fact:ums-ent1}
$\Ent_\pi(L\mathsf P^*f)\le
  \Ent_\pi(\mathcal L\mathsf P^*f)$.
\end{fact}
\begin{fact}\label{fact:ums-ent2}
$\Ent_\pi(\mathcal L\mathsf P^*f)=F_{1-r}(\mathcal C^*f)$ and
$\Ent_\phi(\mathcal C^*f)\le\Ent_\mu(f)$.
\end{fact}
\begin{fact}\label{fact:ums-ent3}
Suppose~\eqref{eq:ums-fk-entropy} holds for every function $h$. Then for every 
$0\le s\le1$,
\[
 F_s(h)\le
 \left(1-\frac{(1-s)|E|}{(C_{\rm ent}+1)(|E|+1)}\right)\Ent_\phi(h).
\]
\end{fact}
With these facts, we conclude the proof of Theorem~\ref{thm:ums-M-entropy}.
\begin{proof}[Proof of Theorem~\ref{thm:ums-M-entropy}]
    Lemmas~\ref{fact:ums-ent1}, \ref{fact:ums-ent2} and \ref{fact:ums-ent3} together give
     \begin{align*}
 \Ent_\pi(L\mathsf P^*f)
 &\le\Ent_\pi(\mathcal L\mathsf P^*f)
 =F_{1-r}(\mathcal C^*f)
 \le\left(1-\frac{r m}{(C_{\rm ent}+1)(m+1)}\right)\Ent_\phi(\mathcal{C}^* f)\\
 &\le \left(1-\frac{rm}{(C_{\rm ent}+1)(m+1)}\right)\Ent_\mu( f) \end{align*}
 Since $L\mathsf P^*f=\E_\pi[f(\sigma)\mid H]$, the law of total entropy gives
\[
 \Ent_\mu(f)=\Ent_\pi(L\mathsf P^*f)
             +\E_{H\sim\PP_{\mu, \rm{M}}}\Ent_{\mu_H}(f).
\]
Therefore we establish the inequality \eqref{eq:ums-M-entropy-margin}.
\end{proof}

We prove the three lemmas next.

\begin{proof}[Proof of Lemma~\ref{fact:ums-ent1}]
By Lemma~\ref{fact:4}, $L\mathcal L=L$, so it yields that
\[
 \Ent_\pi(L\mathsf P^*f)
 =\Ent_\pi(L\mathcal L\mathsf P^*f)
 \le\Ent_\pi(\mathcal L\mathsf P^*f),
\]
where the last step is due to data processing inequality.
\end{proof}

\begin{proof}[Proof of Lemma~\ref{fact:ums-ent2}]
Put $h:=\mathcal C^*f$, so that $h(A)=\E_\pi[f(\sigma)\mid A]$. Under~$\pi$ the
marginal of $\sigma$ is $\mu$ and the marginal of $A$ is $\phi$, so
$\E_\phi h=\E_\mu f=1$.
By definition $(\mathsf P^*f)(\sigma,A,N)=f(\sigma)$, and  $\mathcal L(\cdot)=\E_\pi[\,\cdot\mid Y]$ with $Y=(N,A_R)$ and $R=E\setminus N$.
Since $Y$ is a function of $(A,N)$, the tower property gives
\[
 \mathcal L\mathsf P^*f
 =\E_\pi\bigl[\,\E_\pi[f(\sigma)\mid A,N]\bigm| Y\bigr].
\]
Under~$\pi$, $N$ is independent of $(\sigma,A)$, so
$\E_\pi[f(\sigma)\mid A,N]=\E_\pi[f(\sigma)\mid A]=h(A)$. By the same
independence,
$\pi(A=\cdot \mid N)=\phi(A=\cdot)$.
Hence,
\begin{equation}\label{eq:LPf-identification}
 (\mathcal L\mathsf P^*f)(\sigma,A,N)=\E_\phi[h\mid A_R]
 =:h_R(A_R).
\end{equation}
Let $g:=\mathcal L\mathsf P^*f$. By the chain rule for entropy, 
\[
 \Ent_\pi(g)=\Ent_\pi\bigl(\E_\pi[g\mid N]\bigr)
 +\E_\pi\bigl[\Ent_\pi(g\mid N)\bigr],
\]
where $\Ent_\pi(g\mid N)$ is the entropy of $g$ under $\pi(\,\cdot\mid N)$.
By~\eqref{eq:LPf-identification}, we have
$\E_\pi[g\mid N]=\E_\phi[h_R]=\E_\phi h=1$, so the first term vanishes.
Again by~\eqref{eq:LPf-identification}, conditionally on $N$ we have
$A\sim\phi$ and $g=h_R(A_R)$, so
$\Ent_\pi(g\mid N=n)=\Ent_\phi\bigl(\E_\phi[h\mid A_R]\bigr)$.
Since $R\sim\operatorname{Ber}(1-r)^E$,
\[
 \Ent_\pi(\mathcal L\mathsf P^*f)
 =\E_{R\sim\operatorname{Ber}(1-r)^E}
   \Ent_\phi\bigl(\E_\phi[h\mid A_R]\bigr)
 =F_{1-r}(h).
\]

It remains to show
$\Ent_\phi(h)\le\Ent_\mu(f)$.
Let $\Phi(x)=x\log x$ with the convention $0\log0 = 0$. Since $\E_\phi h=\E_\mu f=1$, Jensen's inequality
applied to $h(A)=\sum_\sigma\pi(\sigma\mid A)f(\sigma)$ gives
\[
 \Ent_\phi(h)=\sum_{A}\phi(A)\,\Phi(h(A))
 \le\sum_{\sigma,A}\pi(\sigma,A)\,\Phi(f(\sigma))
 =\sum_{\sigma}\mu(\sigma)\,\Phi(f(\sigma))
 =\Ent_\mu(f),
\]
and this concludes the proof.
\end{proof}

\begin{proof}[Proof of Lemma~\ref{fact:ums-ent3}]
For $0\le j\le m$, let $\E_{R:\,|R|=j}$ denote the expectation over a set
$R\subseteq E$ chosen uniformly at random among the $\binom{m}{j}$ subsets of
$E$ of size~$j$.
By~\eqref{eq:ums-fk-entropy} and~\cite[Theorem~3.5(2)]{Lee24}, we obtain, for $0\le k\le m$,
\[
 \begin{aligned}
 \E_{R:\,|R|=m-k}\Ent_\phi\bigl(\E_\phi[h\mid A_R]\bigr)
 &\le\prod_{j=m-k+1}^{m}
       \left(1-\frac1{j(C_{\rm ent}+1)}\right)\Ent_\phi(h)
 \le\left(1-\frac{k}{m+1}\right)^{1/(C_{\rm ent}+1)}\Ent_\phi(h)\\
 &\le\left(1-\frac{k}{(C_{\rm ent}+1)(m+1)}\right)\Ent_\phi(h).
 \end{aligned}
\]
The second inequality uses $1-x\le e^{-x}$ and
$\sum_{j=m-k+1}^m j^{-1}\ge\log\frac{m+1}{m-k+1}$; the third uses
$(1-x)^{a}\le1-ax$ for $x\in[0,1]$ and $a=1/(C_{\rm ent}+1)\in(0,1)$.

Finally, let $R\sim\operatorname{Ber}(s)^E$ and
$J:=|E\setminus R|\sim\operatorname{Bin}(m,1-s)$. Conditional on $J=k$, $R$ is
uniformly random among subsets of size $m-k$. Therefore
\[
 \begin{aligned}
 F_s(h)
 &=\sum_{k=0}^m\PP(J=k)\,
   \E_{R:\,|R|=m-k}\Ent_\phi\bigl(\E_\phi[h\mid A_R]\bigr)
 \le\sum_{k=0}^m\PP(J=k)\left(1-\frac{k}{(C_{\rm ent}+1)(m+1)}\right)\Ent_\phi(h)\\
 &
 = \left(1-\frac{\E[J]}{(C_{\rm ent}+1)(m+1)}\right)\Ent_\phi(h)
 =\left(1-\frac{(1-s)m}{(C_{\rm ent}+1)(m+1)}\right)\Ent_\phi(h).
 \end{aligned}
\]
and the proof is complete.
\end{proof}

\section{Log-Sobolev inequality for the FK dynamics}
\label{sec:fk-entropy}
 
Our goal in this section is to establish approximate tensorization of entropy for the
random-cluster distribution with high probability over the random
graph in the uniqueness regime, as formalized in the following theorem.
\begin{theorem}
    \label{thm:fk-at}
    For integers $\Delta \geq 3$ and $q \geq 2$, let $p < p_u$.
  There is a constant $C_{\rm ent}=C_{\rm ent}(\Delta, q, p)>1$ such that with probability $1-o(1)$, a random $\Delta$-regular graph $G=(V,E)$ satisfies  for every function $h:\{0,1\}^E\rightarrow \mathbb R_{\ge 0}$,
\begin{equation}\label{eq:ums-fk-entropy2}
 \Ent_\phi(h)\le C_{\rm ent}\sum_{e\in E}
       \E_\phi\bigl[\Ent(h\mid A_{-e})\bigr],
\end{equation}
where $\Ent(h\mid A_{-e})$ denotes the entropy of $h$ under $\phi(\,\cdot\mid A_{-e})$.
\end{theorem}

The high-level idea of the proof is to consider a distribution
with a product structure $\pi_\ell$, for which approximate
tensorization of entropy is straightforward, and then
compare $\pi_\ell$ with the random-cluster distribution $\phi$.
The product structure  of $\pi_\ell$ will stem from the fact that
with high probability over the random graph,
cycles of length at most $\alpha \log n$ are pairwise vertex-disjoint for a small constant $\alpha > 0$.
Note that longer cycles are plentiful in $G$, but the expected
total size of random-cluster components whose induced subgraphs contain such cycles is $o(1)$. The influence of those longer cycles, and the fact that they do not affect the order of the modified log-Sobolev constant of $\phi$ compared to $\pi_\ell$ comes from the FK shattering-escape dichotomy which we recall from Lemma~\ref{lem:fk-source}.

Let $\mathcal G_n$ be the probability $1-o(1)$ graph event of Lemma~\ref{lem:fk-source}. Let $\mathcal C_\ell$ be the family
of edge sets of simple cycles in $G$ of length at most $\ell = \alpha \log n$ for a small constant $\alpha$.
These cycles are pairwise vertex-disjoint. Recall that
$\theta=\frac{B-1}{B+q-1}$ and consider the following product distribution over cycles in $\mathcal C_\ell$ and their complement: 
\begin{equation}\label{eq:fk-reference}
  \pi_\ell
 =\bigotimes_{C\in\mathcal C_\ell}\eta_C
   \;\otimes\!\bigotimes_{\substack{e\in E\\
       e\notin\bigcup_{C\in\mathcal C_\ell}C}}
       \operatorname{Ber}(\theta),
  \end{equation}
  where $\eta_C$ is the random-cluster distribution over the cycle $C$: i.e., for $X\subset E(C)$, 
  \begin{equation}
 \eta_C(X)
 =\frac{1}{1+(q-1)\theta^{|C|}}
       \theta^{|X|}(1-\theta)^{|C|-|X|}q^{\one\{X=C\}}\,.
  \end{equation}
 
The product structure of $\pi_\ell$ gives tensorization of entropy across cycles in $\mathcal C_\ell$ and other edges directly;
see, e.g.,~\cite[eq.~(1.2)]{CMT15}.
By applying a standard change of measure comparison, we get approximate tensorization across all edges for $\pi_\ell$.

\begin{fact}
\label{fact:fk-reference}
On every $G\in\mathcal G_n$, for every function $h\ge 0$,
\begin{equation}\label{eq:fk-product}
 \Ent_{\pi_\ell}(h)
 \le q^2\sum_{e\in E}\E_{\pi_\ell}[\Ent(h\mid A_{-e})].
\end{equation}
\end{fact}
 
We next formalize the comparison between the random-cluster distribution $\phi$ and the product distribution $\pi_\ell$.
On a graph $G\in\mathcal G_n$, let
$N_\ell(A):=|\{C\in\mathcal C_\ell: C\subseteq A\}|$ be the number of short cycles, of length at most $\alpha \log n$ for a small constant $\alpha > 0$,
that are entirely open in $A$, and define
\[
 \Lambda(A):=|A|-|V|+|\cC(A)|-N_\ell(A)
 \qquad\text{and}\qquad
 Z:=\E_{\pi_\ell}\bigl[q^{\Lambda}\bigr]=\sum_{A\subseteq E}q^{\Lambda(A)}\pi_\ell(A).
\]
Here $|A|-|V|+|\cC(A)|$ is the cycle rank of $(V,A)$, i.e., the number of
independent cycles, so $\Lambda(A)$ counts the open cycles that are not
accounted for by the short open cycles.
The following identity follows directly from the
definitions of $\phi$ and $\pi_\ell$.
\begin{fact}\label{fact:fk-change-of-measure}
On every $G\in\mathcal G_n$, we have
$
 \phi(A)=Z^{-1}q^{\Lambda(A)}\pi_\ell(A).
$
\end{fact}
For a positive distribution $\nu$ over edge configurations, write
\[
 \mathcal S_\nu(h)=\sum_{e\in E}\E_\nu[\Ent_\nu(h\mid A_{-e})]
\]
for the right-hand side of approximate tensorization, so that
Theorem~\ref{thm:fk-at} reads $\Ent_\phi(h)\le C_{\rm ent}\,\mathcal S_\phi(h)$.
To transfer Lemma~\ref{fact:fk-reference} from $\pi_\ell$ to $\phi$, we compare
both sides of the inequality under the two measures. The point is that $\pi_\ell$
already accounts for every short open cycle, so $\phi$ and $\pi_\ell$ differ only
through long open cycles. For $\Ent(h)$, this is because the factor $q^{\Lambda}$ in
Lemma~\ref{fact:fk-change-of-measure} equals $1$ unless $A$ has an open cycle
longer than $\ell$. For $\mathcal S (h)$, it is because the conditional laws of $\phi$
and $\pi_\ell$ at an edge $e$ coincide unless opening $e$ closes such a cycle. In
both cases the discrepancy is confined to the components counted by $\wLong$. The
next two lemmas make this precise. Lemma~\ref{lemma:fk-global-entropy} compares
$\Ent_\phi$ with $\Ent_{\pi_\ell}$, with an error $o(1)\E_\phi[h]$, and
Lemma~\ref{lemma:fk-conditional-entropy} compares $\mathcal S_{\pi_\ell}$ with
$\mathcal S_\phi$, with an error $\E_\phi[h\wLong]$. The latter error is then
controlled by~\eqref{eq:fk-source} in the proof of Theorem~\ref{thm:fk-at}.
 
\begin{lemma}
\label{lemma:fk-global-entropy}
On every $G\in\mathcal G_n$, for every $h:\{0,1\}^E\rightarrow  \mathbb R_{\ge 0}$,
\begin{equation}\label{eq:fk-global-entropy-comparison}
 \Ent_\phi(h)\le\frac1Z\Ent_{\pi_\ell}(q^{\Lambda}h)+o(1)\E_\phi[h].
\end{equation}
\end{lemma}
\begin{lemma}
\label{lemma:fk-conditional-entropy}
On every $G\in\mathcal G_n$, for every $h:\{0,1\}^E\rightarrow \mathbb R_{\ge 0}$,
\begin{equation}\label{eq:fk-functional-comparison}
 \frac1Z\mathcal S_{\pi_\ell}(q^{\Lambda}h)
 \le\mathcal S_\phi(h)+\frac{\Delta\log q}{2}\E_\phi[h\wLong].
\end{equation}
\end{lemma}
 
We now prove Theorem~\ref{thm:fk-at} assuming these lemmas. Lemma~\ref{lem:fk-source} was proved in Section~\ref{sec:conclusions}, and we postpone the proofs of Lemmas~\ref{fact:fk-reference},~\ref{fact:fk-change-of-measure},~\ref{lemma:fk-global-entropy} and~\ref{lemma:fk-conditional-entropy} to Section~\ref{subsec:fk-reference-proof}.
\label{subsec:fk-main-proof}
 
\begin{proof}[Proof of Theorem~\ref{thm:fk-at}]
 
It suffices to prove a logarithmic Sobolev inequality for the FK dynamics, which is known to imply approximate tensorization (see e.g. \cite[Proposition 1.1]{CMT15}). In particular, if for every real function $g:\{0,1\}^E\rightarrow \mathbb R$, we have
\begin{equation}\label{eq:lsi}
\Ent_\phi(g^2)\le C_{\rm ent}m\cE_\phi(g,g),
\end{equation}
then for $h:\{0,1\}^E\rightarrow \mathbb R_{\ge 0}$, set $g=\sqrt h$, and the standard inequality
$\Var_\rho(\sqrt h)\le\Ent_\rho(h)$ in each conditional edge law gives
\[
 m\cE_\phi(\sqrt h,\sqrt h)
 =\sum_e\E_\phi[\Var(\sqrt h\mid A_{-e})]
 \le\mathcal S_\phi(h).
\]
Combining this with \eqref{eq:lsi} proves the theorem.
 
Fix a real function $g$ with $\E_\phi[g^2]>0$; the case
$\E_\phi[g^2]=0$ is immediate since both sides of~\eqref{eq:lsi} become $0$.
Lemma~\ref{lemma:fk-global-entropy}
gives
\[
 \Ent_\phi(g^2)
 \le\frac1Z\Ent_{\pi_\ell}(q^{\Lambda}g^2)+o(1)\E_\phi[g^2].
\]
Combined with the inequalities in
Lemma~\ref{fact:fk-reference} (applied to $q^{\Lambda}g^2$)
and Lemma~\ref{lemma:fk-conditional-entropy} we obtain
\begin{equation}\label{eq:fk-long-cycle-error}
 \Ent_\phi(g^2)\le q^2\mathcal S_\phi(g^2)
       +\frac{q^2\Delta\log q}{2}\E_\phi[g^2\wLong]+o(1)\E_\phi[g^2].
\end{equation}
 
We focus next on bounding  $\E_\phi[g^2\wLong]$. Observe that
  Lemma~\ref{lem:fk-source} bounds $\E_\nu[\wLong]$ under any probability
  distribution $\nu$. The idea is to incorporate the factor $g^2$
  into the distribution.
  For this, define the distribution $\nu_g$ that is given by 
  $
   \nu_g(A)=\frac{g^2(A)\phi(A)}{\E_\phi[g^2]}
  $.
  Applying
  Lemma~\ref{lem:fk-source} to $\nu_g$ gives
  $
   \E_{\nu_g}[\wLong]\le C_{\rm long} I_\phi(\nu_g)+o(1),
  $
and since $\E_{\nu_g}[\wLong]= \frac{\E_\phi[g^2\wLong]}{\E_\phi[g^2]}$, we obtain
\begin{equation}
\label{eq:s2}
   \E_\phi[g^2\wLong]\le C_{\rm long} I_\phi(\nu_g)\E_\phi[g^2] +o(1)\E_\phi[g^2].
\end{equation}
  Next, we bound $I_\phi(\nu_g)\E_\phi[g^2]$ in terms of
  $\mathcal S_\phi(g^2)$.
 Let $\nu_g(A_{-e})$ and $\phi(A_{-e})$
  denote the marginal probabilities of the configuration outside of the edge $e$.
  Then,
  \begin{align*}
   \nu_g(A_{-e})
   &=\nu_g(A_{-e},e=1)+\nu_g(A_{-e},e=0)\\
   &=\frac{1}{\E_\phi[g^2]}\Big[g^2(A_{-e},e=1)\phi(A_{-e},e=1)
                +g^2(A_{-e},e=0)\phi(A_{-e},e=0)\Big],
  \end{align*}
  and thus
  \[
  \begin{aligned}
   \frac{\nu_g(A_{-e})}{\phi(A_{-e})}
   &=\frac{1}{\E_\phi[g^2]}\Bigl[g^2(A_{-e},e=1)\phi(e=1\mid A_{-e})
  +g^2(A_{-e},e=0)\phi(e=0\mid A_{-e})\Bigr]
   =\frac{\E_\phi[g^2\mid A_{-e}]}{\E_\phi[g^2]}.
  \end{aligned}
  \]
  From this, we get
  $
   \frac{\nu_g(A_e\mid A_{-e})}{\phi(A_e\mid A_{-e})} =\frac{g^2(A)}{\E_\phi[g^2\mid A_{-e}]},
  $
  which combined with the definitions of $\nu_g$ and $I_\phi(\nu_g)$ gives
  \begin{equation}\label{eq:fk-density}
  \begin{aligned}
   \E_\phi[g^2] I_\phi(\nu_g)
   &=\E_\phi[g^2] \sum_{e\in E}\E_{\nu_g}\!\left[
       \log\frac{\nu_g(A_e\mid A_{-e})}{\phi(A_e\mid A_{-e})}\right]
   =\sum_{e\in E}\E_\phi\!\left[
       g^2\log\frac{g^2}{\E_\phi[g^2\mid A_{-e}]}\right]\\
   &=\sum_{e\in E}\E_\phi[\Ent_\phi(g^2\mid A_{-e})] = \mathcal S_\phi(g^2),
  \end{aligned}
  \end{equation}
  where the second to last equality follows from the definition of the conditional
  entropy.
  Plugging this identity into~\eqref{eq:s2} gives
  \begin{equation}\label{eq:fk-source-functional}
   \E_\phi[g^2\wLong]\le C_{\rm long}\mathcal S_\phi(g^2)+o(1)\E_\phi[g^2].
  \end{equation}
  Consequently, the inequality in~\eqref{eq:fk-long-cycle-error}
  becomes
  \begin{equation}\label{eq:fk-defective-at}
   \Ent_\phi(g^2)\le K\mathcal S_\phi(g^2)+o(1)\E_\phi[g^2],
  \end{equation}
    for a constant $K=K(\Delta,q,B)>0$.
 
 For every edge $e$, the conditional edge law $\phi(\cdot \mid A_{-e})$ assigns positive probability at least $\min\{\hat p, 1-p\}$ to each of its two states, so its log-Sobolev constant is bounded below uniformly. Hence, uniformly for every edge $e$, we have $\Ent_\phi(g^2 \mid A_{-e})\le J\Var_\phi(g\mid A_{-e})$, for a suitable constant $J = J(p,q) > 0$. Thus, by the definition of $\mathcal S_\phi$,
  \begin{equation}\label{eq:fk-source-energy}
   \mathcal S_\phi(g^2)
   \le J\sum_{e\in E}\E_\phi[\Var_\phi(g\mid A_{-e})]
   =Jm\cE_\phi(g,g).
  \end{equation}
  Combining this with \eqref{eq:fk-defective-at} yields
  \begin{equation}\label{eq:fk-defective-lsi}
   \Ent_\phi(g^2)\le KJm\cE_\phi(g,g)+o(1)\E_\phi[g^2].
  \end{equation}
Applying \eqref{eq:fk-defective-lsi} to $g-\E_\phi[g]$, we get
$$
 \Ent_\phi((g-\E_\phi[g])^2)\le KJm\cE_\phi(g,g)+o(1)\Var_\phi(g),
$$
since subtracting a constant leaves the Dirichlet form unchanged.
By Rothaus' centering inequality, 
\[
 \Ent_\phi(g^2)
 \le\Ent_\phi\bigl((g-\E_\phi[g])^2\bigr)+2\Var_\phi(g)
 \le KJm\cE_\phi(g,g)+(2+o(1))\Var_\phi(g).
\]
Finally, to bound the variance, note that
the results in~\cite{BG21}
give a graph event $\mathcal H_n$ of probability
$1-o(1)$ such that for every $G\in\mathcal H_n$, we have
$t_{\rm mix}(n^{-a})\le O(m\log n)$ for some $a>0$; see, e.g.,~\cite[the proof of Theorem~1.1]{BG21}. 
This immediately implies that
$\gap(P_{\rm FK})\ge\frac1{C_{\rm BG}m}$ for a suitable constant $C_{\rm BG} > 0$; see, e.g.,~\cite[Theorem~12.5]{LP17}.
So for a $G\in\mathcal G_n\cap\mathcal H_n$, we have
$
 \Var_\phi(g)\le C_{\rm BG}m\cE_\phi(g,g).
$
Substituting this variance bound into the preceding entropy estimate gives
\[
 \Ent_\phi(g^2)
 \le\bigl(KJ+(2+o(1))C_{\rm BG}\bigr)m\cE_\phi(g,g),
\]
as desired.
\end{proof}
 
\subsection{Deferred proofs of lemmas}
\label{subsec:fk-reference-proof}
 
It remains for us to provide the proofs of Lemmas~\ref{fact:fk-reference},~\ref{fact:fk-change-of-measure},~\ref{lemma:fk-global-entropy} and~\ref{lemma:fk-conditional-entropy}.
 
\begin{proof}[Proof of Lemma~\ref{fact:fk-reference}]
On $G\in\mathcal G_n$ the cycles in $\mathcal C_\ell$ are vertex-disjoint, so
$\pi_\ell$ is a product measure over the blocks $C\in\mathcal C_\ell$ and the
single edges outside them. By~\cite[eq.~(1.2)]{CMT15} applied to these blocks,
\[
 \Ent_{\pi_\ell}(h)\le\sum_{C\in\mathcal C_\ell}\E_{\pi_\ell}[\Ent(h\mid A_{-C})]
   +\sum_{e\notin\cup\mathcal C_\ell}\E_{\pi_\ell}[\Ent(h\mid A_{-e})].
\]
Since $\pi_\ell(\cdot\mid A_{-C})=\eta_C$ and
$\pi_\ell(\cdot\mid A_{-e})=\eta_C(\cdot\mid A_{C-e})$ for $e\in C$, and since
$q\ge1$, it suffices to show that for each $C\in\mathcal C_\ell$ and every $f\ge0$,
\begin{equation}\label{eq:cycle-at}
 \Ent_{\eta_C}(f)\le q^2\sum_{e\in C}\E_{\eta_C}[\Ent_{\eta_C}(f\mid X_{C-e})].
\end{equation}
Let $\mathbb B_C=\operatorname{Ber}(\theta)^{\otimes C}$ and
$Z_C=1+(q-1)\theta^{|C|}$, so that
$\eta_C/\mathbb B_C=q^{\one\{X=C\}}/Z_C$, and hence
$Z_C^{-1}\mathbb B_C\le\eta_C\le qZ_C^{-1}\mathbb B_C$.
For $e\in C$, the conditional laws $\eta_C(\cdot\mid X_{C-e})$ and
$\mathbb B_C(\cdot\mid X_{C-e})$ coincide unless $C-e$ is entirely open, in
which case
$\mathbb B_C(\cdot\mid X_{C-e})\le(1+(q-1)\theta)\,\eta_C(\cdot\mid X_{C-e})\le q\,\eta_C(\cdot\mid X_{C-e})$.
We use the Holley--Stroock perturbation bound~\cite{HS87}
(see also~\cite[proof of Lemma~2.2]{CMT15}): if $\mu\le b\,\nu$, then
$\Ent_\mu(f)\le b\Ent_\nu(f)$. Applying it once globally and once to each
conditional law, together with~\cite[eq.~(1.2)]{CMT15} for $\mathbb B_C$ and
$\mathbb B_C\le Z_C\,\eta_C$, gives
\[
\begin{aligned}
 \Ent_{\eta_C}(f)
 &\le\frac{q}{Z_C}\Ent_{\mathbb B_C}(f)
 \le\frac{q}{Z_C}\sum_{e\in C}\E_{\mathbb B_C}[\Ent_{\mathbb B_C}(f\mid X_{C-e})]\\
 &\le\frac{q^2}{Z_C}\sum_{e\in C}\E_{\mathbb B_C}[\Ent_{\eta_C}(f\mid X_{C-e})]
 \le q^2\sum_{e\in C}\E_{\eta_C}[\Ent_{\eta_C}(f\mid X_{C-e})].
\end{aligned}
\]
This proves~\eqref{eq:cycle-at}, and hence~\eqref{eq:fk-product} with constant $q^2$.
\end{proof}
 
\begin{proof}[Proof of Lemma~\ref{fact:fk-change-of-measure}]
Recall that $\phi(A)\propto p^{|A|}(1-p)^{|E|-|A|}q^{|\cC(A)|}$ with $p=1-1/B$.
Since $p/(1-p)=B-1=q\,\theta/(1-\theta)$, we have
\[
 \phi(A)\propto q^{|A|+|\cC(A)|}\,\theta^{|A|}(1-\theta)^{|E|-|A|}
 \propto q^{|A|-|V|+|\cC(A)|}\,\mathbb B(A),
\]
where $\mathbb B=\operatorname{Ber}(\theta)^{\otimes E}$.
On the other hand, since the cycles in $\mathcal C_\ell$ are disjoint,
\eqref{eq:fk-reference} gives
\[
 \pi_\ell(A)=\prod_{C\in\mathcal C_\ell}\frac{q^{\one\{C\subseteq A\}}}{1+(q-1)\theta^{|C|}}\;\mathbb B(A)
 \propto q^{N_\ell(A)}\,\mathbb B(A).
\]
Thus we have $\phi(A)\propto q^{\Lambda(A)}\pi_\ell(A)$, and the normalizing
constant is $\bigl(\sum_{A}q^{\Lambda(A)}\pi_\ell(A)\bigr)^{-1}=Z^{-1}$.
\end{proof}
 
\begin{proof}[Proof of Lemma~\ref{lemma:fk-global-entropy}]
  Lemma~\ref{fact:fk-change-of-measure} gives
  $\E_{\pi_\ell}[q^{\Lambda}h]=Z\E_\phi[h]$. Then, the definition of entropy gives
  \[
  \begin{aligned}
   \frac1Z\Ent_{\pi_\ell}(q^{\Lambda}h)
   &=\frac1Z\E_{\pi_\ell}\!\left[
       q^{\Lambda}h\log\frac{q^{\Lambda}h}{\E_{\pi_\ell}[q^{\Lambda}h]}\right]
   =\E_\phi\!\left[h\log\frac{q^{\Lambda}h}{Z\E_\phi[h]}\right]\\
   &=\E_\phi\!\left[h\log\frac{h}{\E_\phi[h]}\right]
     +(\log q)\E_\phi[h\Lambda]-(\log Z)\E_\phi[h].
  \end{aligned}
  \]
  Rearranging, we have
  \begin{equation}
  \label{eq: rewriting ent h}
   \Ent_\phi(h)
   =\frac1Z\Ent_{\pi_\ell}(q^{\Lambda}h)-(\log q)\E_\phi[h\Lambda]
      +(\log Z)\E_\phi[h].
  \end{equation}
  We now bound the last two terms in \eqref{eq: rewriting ent h}.
  By Lemma~\ref{lem:fk-source}, the cycles counted by $N_\ell(A)$
  are edge-disjoint, so $N_\ell(A)$ is at most the cycle rank
  $|A|-|V|+|\cC(A)|$. Consequently,
  $\Lambda(A)\ge0$, and since $h\ge0$,
  we have $-(\log q)\E_\phi[h\Lambda]\le0$.
 
  To bound the second term, it suffices to show that $\log Z=o(1)$.
  Lemma~\ref{fact:fk-change-of-measure} gives
  \[
   \E_\phi[q^{-\Lambda}]
   =\sum_{A\subseteq E}q^{-\Lambda(A)}Z^{-1}q^{\Lambda(A)}\pi_\ell(A)
   =Z^{-1},
  \]
  and Jensen's inequality  yields
  $
   -\log Z=\log\E_\phi[q^{-\Lambda}]\ge-(\log q)\E_\phi[\Lambda].
  $
  Also $\Lambda\ge0$ implies $Z=\E_{\pi_\ell}[q^{\Lambda}]\ge1$.
  Thus $0\le\log Z\le(\log q)\E_\phi[\Lambda]$, and we only need to bound
  $\E_\phi[\Lambda]$.
 
  By the definition of $\Lambda$, each open
  component $U$ contributes
  \[
   \#\{\text{open edges in }U\}-(|U|-1)
   -\#\{\text{open short cycles in }U\}
  \]
  to $\Lambda(A)$.
 If $U$ is not counted by $\wLong$, all its open cycles are short and
edge-disjoint. Removing one edge from each leaves a tree with
$|U|-1$ edges. Hence the number of open edges in $U$ equals
$|U|-1$ plus the number of open short cycles, so its contribution
to $\Lambda(A)$ is zero.
  Otherwise, if $U$ is counted by $\wLong$, its contribution is at most its number
  of open edges which can be bounded by $\Delta|U|/2$.
  Summing over these components, whose total number of
  vertices is $\wLong(A)$, gives
  \begin{equation}\label{eq:fk-potential}
   0\le\Lambda(A)\le\frac{\Delta}{2}\wLong(A).
  \end{equation}
  Applying Lemma~\ref{lem:fk-source} with $\nu=\phi$ and using
  $I_\phi(\phi)=0$, we obtain $\E_\phi[\wLong]=o(1)$.
  Consequently, $\E_\phi[\Lambda]\le(\Delta/2)\E_\phi[\wLong]=o(1)$.
  As such, we have obtained that
  \[
   \Ent_\phi(h)\le\frac1Z\Ent_{\pi_\ell}(q^{\Lambda}h)+o(1)\E_\phi[h]\,. \qedhere
  \]
\end{proof}
 
\begin{proof}[Proof of Lemma~\ref{lemma:fk-conditional-entropy}]
By Lemma~\ref{fact:fk-change-of-measure}, we have
$\phi=Z^{-1}q^{\Lambda}\pi_\ell$.
Fix an edge $e$ and a configuration $A$. Then
$\Lambda(A\cup\{e\})-\Lambda(A\setminus\{e\})\in\{0,1\}$.
Indeed, adding $e$ closes at most one simple cycle of length at most
$\ell$, and doing so also increases the cycle rank $|A|-|V|+|\cC(A)|$ by one. If $\Lambda(A\cup\{e\})-\Lambda(A\setminus\{e\})=1$, an open
path in $A\setminus\{e\}$ joining the endpoints of $e$, together with $e$, forms a simple cycle
of length greater than $\ell$;
this event can only happen when both endpoints 
of $e$ lie in a component $U\in \mathcal{C}(A\setminus \{e\})$ whose induced subgraph $G[U]$  contains a long cycle. 
Consequently,
\begin{equation}\label{eq:fk-residual}
 \sum_{e\in E}\bigl[\Lambda(A\cup\{e\})-\Lambda(A\setminus\{e\})\bigr]\le\frac\Delta2 \wLong(A).
\end{equation}
The conditional probabilities $\phi(A_e\mid A_{-e})$ and $\pi_\ell(A_e\mid A_{-e})$ differ by at most the factor
$q^{\Lambda(A\cup\{e\})-\Lambda(A\setminus\{e\})}$, 
so with any distribution $\nu$ on edge configurations we obtain
\begin{align}
I_{\pi_\ell}(\nu)
&=\sum_{e\in E}\E_{\nu}\Bigl[\sum_{x\in\{0,1\}}\nu(A_e=x\mid A_{-e})
\log\frac{\nu(A_e=x\mid A_{-e})}{\pi_\ell(A_e=x\mid A_{-e})}\Bigr]\nonumber\\
&=I_\phi(\nu)+\sum_{e\in E}\E_{\nu}\Bigl[\sum_{x\in\{0,1\}}\nu(A_e=x\mid A_{-e})
\log\frac{\phi(A_e=x\mid A_{-e})}{\pi_\ell(A_e=x\mid A_{-e})}\Bigr]\nonumber\\
&\le I_\phi(\nu)+\sum_{e\in E}\E_{\nu}\Bigl[\sum_{x\in\{0,1\}}\nu(A_e=x\mid A_{-e})
\bigl[\Lambda(A\cup\{e\})-\Lambda(A\setminus\{e\})\bigr]\log q\Bigr]\nonumber\\
&=I_\phi(\nu)+\log q\;\E_{\nu}\Bigl[\sum_{e\in E}\bigl[\Lambda(A\cup\{e\})-\Lambda(A\setminus\{e\})\bigr]\Bigr]\nonumber\\
&\le I_\phi(\nu)+\frac{\Delta\log q}{2}\E_\nu[\wLong].\label{eq:fk-information-comparison}
\end{align}

Now let $h:\{0,1\}^E\rightarrow \mathbb R_{\ge 0}$ be a function with $\E_\phi[h]>0$ and let $\nu_h$ be the distribution on edge configurations given by
$\nu_h(A)=h(A)\phi(A)/\E_\phi[h]$ for $A\subseteq E$.
By Lemma~\ref{fact:fk-change-of-measure}, we have $\E_{\pi_\ell}[q^{\Lambda}h]=Z\E_\phi[h]$, and
$\nu_h(A)=q^{\Lambda(A)}h(A)\pi_\ell(A)/\E_{\pi_\ell}[q^{\Lambda}h]$.
Using these two interpretations of $\nu_h$,
the same calculation leading to \eqref{eq:fk-density} gives
\begin{equation}\label{eq:nu-h-identities}
\E_\phi[h]I_\phi(\nu_h)=\mathcal S_\phi(h) \quad \text{and} \quad 
Z\E_\phi[h]I_{\pi_\ell}(\nu_h)
    =
 \E_{\pi_\ell}[q^\Lambda h]I_{\pi_\ell}(\nu_h)
 =\mathcal S_{\pi_\ell}(q^{\Lambda}h).
\end{equation}
The normalization in $\nu_h$ also gives
\begin{equation}\label{eq:nu-h-wlong}
 \E_\phi[h]\,\E_{\nu_h}[\wLong]
 =\E_\phi[h]\sum_{A\subseteq E}
       \frac{h(A)\phi(A)}{\E_\phi[h]}\wLong(A)
 =\E_\phi[h\wLong].
\end{equation}
Therefore,
\begin{align*}
\frac1Z\mathcal S_{\pi_\ell}(q^{\Lambda}h)
&=\E_\phi[h]\,I_{\pi_\ell}(\nu_h)
&&\text{by the second identity in \eqref{eq:nu-h-identities}}\\
&\le\E_\phi[h]\Bigl(I_\phi(\nu_h)+\frac{\Delta\log q}{2}\,\E_{\nu_h}[\wLong]\Bigr)
&&\text{by \eqref{eq:fk-information-comparison} with $\nu=\nu_h$, and $\E_\phi[h]>0$}\\
&=\E_\phi[h]\,I_\phi(\nu_h)+\frac{\Delta\log q}{2}\,\E_\phi[h]\,\E_{\nu_h}[\wLong]\\
&=\mathcal S_\phi(h)+\frac{\Delta\log q}{2}\,\E_\phi[h\wLong]
&&\text{by the first identity in \eqref{eq:nu-h-identities} and \eqref{eq:nu-h-wlong}.}
\end{align*}
This is \eqref{eq:fk-functional-comparison} for $h$ with $\E_\phi[h]>0$.
If $\E_\phi[h]=0$, then $h\equiv0$ since $\phi$ gives positive mass to every configuration,
so both sides of \eqref{eq:fk-functional-comparison} vanish.
Hence \eqref{eq:fk-functional-comparison} holds for every $h:\{0,1\}^E\rightarrow\mathbb R_{\ge0}$,
which proves the lemma.
\end{proof}

\section{Logarithmic mixing for Swendsen--Wang dynamics}
\label{sec:rr-sw}

Our goal in this section is to establish
Theorem~\ref{mainthm:Swendsen-Wang}.
Fix $L=C_0\log n$ as in Lemma~\ref{lem:entropy-mass} and 
let $K_w\ge C_0$ in
Lemma~\ref{lem:poly-geometry}. 
Our proof relies on the following two facts.

\begin{lemma}\label{lemma:rr-sw-local}
There exists a constant
$\alpha=\alpha(\Delta,q,B)>0$, such that with probability $1-o(1)$ a random $\Delta$-regular graph $G$ is such that 
for every function $f:[q]^V\to\mathbb R_{\ge0}$ we have
\[
 \alpha\mathcal H_{\widetilde M}(f)
 \le\Ent_\mu(f)-\Ent_\mu(P_{\rm SW}f).
\]
\end{lemma}

\begin{lemma}\label{lemma:rr-sw-smoothing}
On every graph of maximum degree $\Delta$, there is a constant
$\kappa=\kappa(\Delta,q,B)>0$ such that, for every nonnegative
function $f$ on $[q]^V$ with $\E_\mu[f]=1$, we have
\begin{equation}\label{eq:rr-sw-entropy-ineq3}
 I_\mu((P_{\rm SW}f)\mu)
 \le\kappa\bigl[\Ent_\mu(f)-\Ent_\mu(P_{\rm SW}f)\bigr].
\end{equation}
\end{lemma}

\begin{proof}[Proof of Theorem~\ref{mainthm:Swendsen-Wang}]
Suppose $G$ is on the probability $1-o(1)$ event that is the intersection of the graph events in Lemmas~\ref{lem:entropy-mass},~\ref{lem:entropy-M},~\ref{lem:poly-geometry} and~\ref{lemma:rr-sw-local}. 
Let $f$ be an arbitrary nonnegative function on $[q]^V$ with
$\E_\mu[f]=1$. Since $P_{\rm SW}f\ge0$ and $\E_\mu[P_{\rm SW}f]=1$,
$(P_{\rm SW}f)\mu$ is a probability distribution.
Applying~\eqref{eq: entropy M vs tilda M} to
$\sqrt{P_{\rm SW}f}$ and using Lemma~\ref{lem:entropy-mass}, we get
\[
 \mathcal H_M(P_{\rm SW}f)
 \le\mathcal H_{\widetilde M}(P_{\rm SW}f)
       +C_1 I_\mu((P_{\rm SW}f)\mu)+C_1n^{-5},
\]
for a suitable constant $C_1>0$.
Applying Lemma~\ref{lemma:rr-sw-local} to $P_{\rm SW}f$ and
Lemma~\ref{lemma:rr-sw-smoothing} to $f$, 
for a suitable $C_2>0$, we obtain
\begin{align}
 \mathcal H_M(P_{\rm SW}f)
 &\le\alpha^{-1}
       [\Ent_\mu(P_{\rm SW}f)-\Ent_\mu(P_{\rm SW}^2f)]
      +C_1\kappa
       [\Ent_\mu(f)-\Ent_\mu(P_{\rm SW}f)]
       +C_1n^{-5}\notag\\
 &\le C_2\bigl[\Ent_\mu(f)-\Ent_\mu(P_{\rm SW}^2f)\bigr]
       +C_1n^{-5}.
 \label{eq:rr-sw-entropy-ineq2}
\end{align}
By Lemma~\ref{lem:entropy-M}, we have
$a\Ent_\mu(P_{\rm SW}f)\le\mathcal H_M(P_{\rm SW}f)$
for a constant $a>0$.
Combining this with~\eqref{eq:rr-sw-entropy-ineq2}, and using
$\Ent_\mu(P_{\rm SW}^2f)\le\Ent_\mu(P_{\rm SW}f)$, we get that there exists $\rho \in (0,1)$ such that
\[
 \Ent_\mu(P_{\rm SW}^2f)
 \le \rho \Ent_\mu(f)
       +O(n^{-5}).
\]
Iterating gives for every $k\ge 0$,
\[
 \Ent_\mu(P_{\rm SW}^{2k}f)
 \le\rho^k\Ent_\mu(f)+O(n^{-5}).
\]
We now follow standard arguments (see, e.g., the proof of \cite[Lemma~2.4]{BlancaEtAl2022}), while 
retaining the additive error.
For any configuration $\sigma$, let
$f_\sigma=\one_{\{\sigma\}}/\mu(\sigma)$.
Pinsker's inequality gives
\[
 \|P_{\rm SW}^{2k}(\sigma,\cdot)-\mu\|_{\rm TV}^2
 \le\frac12\Ent_\mu(P_{\rm SW}^{2k}f_\sigma)
 \le\frac12\left[\rho^k \Ent_\mu(f_\sigma)+O(n^{-5})\right],
\]
and the result follows by noting that
$\Ent_\mu(f_\sigma)=\log\frac1{\mu(\sigma)}=O(n)$, and thus
$t_{\rm mix}(P_{\rm SW})=O(\log n)$.
\end{proof}

We provide next the missing proofs of
Lemmas~\ref{lemma:rr-sw-local}
and~\ref{lemma:rr-sw-smoothing}.

\begin{proof}[Proof of Lemma~\ref{lemma:rr-sw-local}]
Suppose $G$ is on the probability $1-o(1)$ event that is the intersection of the graph events in Lemmas~\ref{lem:entropy-mass},~\ref{lem:entropy-M}, and~\ref{lem:poly-geometry}. 
The idea is to compare the SW and $M$ chains after they draw the same
auxiliary graph $H$. For every subgraph $G' \subseteq G$, let $S_{G'}$ denote one SW step acting only on the vertices of $G'$ at interaction $\tilde B$. For fixed $H$, $S_H$ is the product of the kernels over the connected components of $H$. Since $\frac{r\tilde B}{B}(1-(\tilde B)^{-1})=p$, we have
$$
P_{\rm SW}(\sigma,\tau)=\sum_H\mathbb P_{\sigma,\rm M}(H)S_H(\sigma,\tau).$$
Thus, after sampling the subgraph $H$, the SW dynamics performs an $S_H$ update, whereas
$\widetilde M$ resamples every component of size at most $L$,
leaving the spins on larger components unchanged.

Fix $f:[q]^V\to\mathbb R_{\ge0}$. 
By convexity of the map $x \mapsto x\log x$ on $x \ge 0$ with the convention $0 \log 0 = 0$, Jensen's inequality gives for each $\sigma$
\[
 (P_{\rm SW}f)(\sigma)\log(P_{\rm SW}f)(\sigma)
 \le\sum_H\mathbb P_{\sigma,\rm M}(H)
       (S_Hf)(\sigma)\log(S_Hf)(\sigma).
\]
Averaging this inequality over $\sigma\sim\mu$ and using the fact that
$\mu(\sigma)\mathbb P_{\sigma,\rm M}(H)
=\mathbb P_{\mu,\rm M}(H)\mu_H(\sigma)$, see~\eqref{eq: reversibility}, we obtain:
\begin{equation}
\label{eq:sw:ent:1}
 \E_\mu[(P_{\rm SW}f)\log(P_{\rm SW}f)]
 \le\sum_H\mathbb P_{\mu,\rm M}(H)
       \E_{\mu_H}[(S_Hf)\log(S_Hf)].
\end{equation}
Identity~\eqref{eq: reversibility} also gives
$
 \E_\mu[f\log f]
 =\sum_H\mathbb P_{\mu,\rm M}(H)\E_{\mu_H}[f\log f].
$
Since $\E_\mu[P_{\rm SW}f]=\E_\mu[f]$ and $\E_{\mu_H}[S_Hf]=\E_{\mu_H}[f]$, we get from~\eqref{eq:sw:ent:1} and the definition of entropy that
\begin{align}\label{eq:rr-sw-averaged-entropy-loss}
 \Ent_\mu(f)-\Ent_\mu(P_{\rm SW}f)
 &=\E_\mu[f\log f]
   -\E_\mu[(P_{\rm SW}f)\log(P_{\rm SW}f)]\notag\\
   &\ge \sum_H\mathbb P_{\mu,\rm M}(H)\Big(\E_{\mu_H}[f\log f]- 
       \E_{\mu_H}[(S_Hf)\log(S_Hf)]\Big)\notag \\
    &= \sum_H\mathbb P_{\mu,\rm M}(H)
       [\Ent_{\mu_H}(f)-\Ent_{\mu_H}(S_Hf)].
\end{align}
We will prove below that
\begin{equation}\label{eq:rr-sw-component-entropy}
 \Ent_{\mu_F}(S_Fg)\le(1-\alpha)\Ent_{\mu_F}(g)
\end{equation}
for every component $F$ of size at most $L$ and every function
$g:[q]^{V(F)}\to\mathbb R_{\ge0}$, with a common $\alpha>0$.
Assuming this bound for now, for fixed $H$, let $U=U^+(H)$ be the union of the
components with more than $L$ vertices, and let $S_H^-$ be the
transition matrix for the Markov chain on $[q]^V$ that performs SW updates at interaction
$\tilde B$ on the small components leaving the configuration on $U$ fixed.
This auxiliary Markov chain serves as a bridge in the comparison between the SW dynamics and $M$.
The updates on the large components commute with $S_H^-$, so
\[
 \Ent_{\mu_H}(S_Hf)\le\Ent_{\mu_H}(S_H^-f).
\]

Conditional on a configuration $\sigma_U$ on $U$, $\mu_H$ and
$S_H^-$ are products over the small components. 
Therefore,  the entropy chain rule and convexity, together with~\eqref{eq:rr-sw-component-entropy} give
\[
 \Ent_{\mu_H}(S_H^-f\mid\sigma_U)
 \le(1-\alpha)\Ent_{\mu_H}(f\mid\sigma_U).
\]
We also have
$
 \E_{\mu_H}[S_H^-f\mid\sigma_U]
 =\E_{\mu_H}[f\mid\sigma_U]$, so the law of total entropy yields that
\[
 \Ent_{\mu_H}(f)
 =\E_{\mu_H}[\Ent_{\mu_H}(f\mid\sigma_U)]
  +\Ent_{\mu_H}\!\left(\E_{\mu_H}[f\mid\sigma_U]\right),
\]
the first term is exactly the entropy removed by resampling the
small components while keeping $\sigma_U$ fixed. The second term
is unchanged by $S_H^-$. Thus
\begin{align*}
 \Ent_{\mu_H}(f)-\Ent_{\mu_H}(S_Hf)
 &\ge\Ent_{\mu_H}(f)-\Ent_{\mu_H}(S_H^-f)\\
 &=\E_{\mu_H}\!\left[
     \Ent_{\mu_H}(f\mid\sigma_U)
     -\Ent_{\mu_H}(S_H^-f\mid\sigma_U)\right]\\
 &\ge\alpha\E_{\mu_H}[\Ent_{\mu_H}(f\mid\sigma_U)].
\end{align*}
Plugging this inequality into~\eqref{eq:rr-sw-averaged-entropy-loss} gives
\begin{align*}
 \Ent_\mu(f)-\Ent_\mu(P_{\rm SW}f)
 &\ge\alpha\sum_H\mathbb P_{\mu,\rm M}(H)
       \E_{\mu_H}\!\left[
         \Ent_{\mu_H}(f\mid\sigma_{U^+(H)})\right]=\alpha\mathcal H_{\widetilde M}(f).
\end{align*}

It remains to prove~\eqref{eq:rr-sw-component-entropy}.
First suppose that $F=T$ is a tree. Then, there exists $\hat \alpha \in (0,1)$ such that for every $f\ge0$ we have
\begin{equation}\label{eq:rr-sw-tree-entropy}
 \Ent_{\mu_T}(S_Tf)
 \le(1-\hat \alpha)\Ent_{\mu_T}(f).
\end{equation}
This follows from the product structure of both the Potts distribution and the SW dynamics on
the tree with the coordinates given by the root color and the edge differences modulo $q$. The root color is uniform, and the
differences are independent with distribution
$\omega(a)=\frac{(\tilde B)^{\one\{a=0\}}}{\tilde B+q-1}$.
Each coordinate update contracts
entropy at a constant rate, and tensorization gives
\eqref{eq:rr-sw-tree-entropy}.

If $F$ is not a tree, let $T$ be a spanning tree of $F$.
Let $\exc(F)$ be the number of excess edges, whose deletion leaves $T$.
In the opening step of $S_F$, all these excess edges are closed
with probability at least $(\tilde B)^{-\exc(F)}$. Conditional on this event,
the update is exactly $S_T$. Thus,
$S_F\ge(\tilde B)^{-\exc(F)}S_T$ entrywise.
Moreover, adding the excess edges multiplies each spin weight
by a factor between $1$ and $(\tilde B)^{\exc(F)}$. Thus, if $d_-$ and $d_+$
are the minimum and maximum of the ratio $\mu_F/\mu_T$, then
$d_+/d_-\le(\tilde B)^{\exc(F)}$.
Consequently, by~\eqref{eq:rr-sw-tree-entropy}, we get
\begin{align*}
 \Ent_{\mu_F}(g)-\Ent_{\mu_F}(S_Fg)
 &\ge d_-(\tilde B)^{-\exc(F)}
       [\Ent_{\mu_T}(g)-\Ent_{\mu_T}(S_Tg)]\\
 &\ge\frac{\hat\alpha d_-}{\,d_+}(\tilde B)^{-\exc(F)}\Ent_{\mu_F}(g)
 \ge \hat\alpha(\tilde B)^{-2\exc(F)}\Ent_{\mu_F}(g).
\end{align*}
Lemma~\ref{lem:poly-geometry} bounds ${\exc(F)}$ by a constant 
for every component of size at most $L\le K_w\log n$ and
thus~\eqref{eq:rr-sw-component-entropy} follows.
\end{proof}

\begin{proof}[Proof of Lemma~\ref{lemma:rr-sw-smoothing}]
  Throughout this proof there is no sprinkling: $\pi_{\ES}$ is the joint Edwards--Sokal law of $(\sigma,\omega)$, and $\mathcal C$, $\mathcal C^*$ are the Edwards--Sokal operators of Section~\ref{sec:M-FK-comparison}, so that $(\mathcal C h)(\sigma)=\E_{\pi_{\ES}}[h(\omega)\mid\sigma]$.

We first show that for every function $h\ge0$ on random-cluster configurations with
$\E_\phi[h]=1$ we have
\begin{equation}\label{eq:rr-sw-reverse}
 I_\mu((\mathcal C h)\mu)
 \le\kappa\bigl[\Ent_\phi(h)-\Ent_\mu(\mathcal C h)\bigr].
\end{equation}
Let $G_v$ be the heat-bath Glauber update at vertex $v$, so that
$P_{\rm Gl}=n^{-1}\sum_vG_v$. By the law of total entropy and
Definition~\ref{def:local-relent},
\[
 I_\mu((\mathcal C h)\mu)
 =\sum_{v\in V}\bigl[
       \Ent_\mu(\mathcal C h)-\Ent_\mu(G_v\mathcal C h)\bigr],
\]
and thus~\eqref{eq:rr-sw-reverse} is equivalent to
\[
 \sum_{v\in V}
       [\Ent_\mu(\mathcal C h)-\Ent_\mu(G_v\mathcal C h)]
 \le\kappa[\Ent_\phi(h)-\Ent_\mu(\mathcal C h)].
\]
Partition $V$ into at most $\Delta+1$ independent sets, and fix
one such set $J$. Condition on all spins outside $J$ and all
edges with both endpoints outside $J$. Write $E_v$ for the
edges incident to vertex $v$, and let $\pi_v$
be the conditional joint distribution under $\pi_{\ES}$ on the star centered at $v$. Since $J$ is an independent set, these stars are independent under this conditioning.
Let $\Omega_v=\operatorname{supp}(\pi_v)$,
$\Omega_J=\prod_{v\in J}\Omega_v$, and
$\pi_J=\bigotimes_{v\in J}\pi_v$.
All functions and operators below are defined on $\Omega_J$.

For $v\in J$, let $K_v$ resample the edges in $E_v$ given
the spins, and let $L_v$ resample both the spin and edge
configuration in the star centered at $v$ from $\pi_v$.
Order $J=(v_1,\ldots,v_m)$ and set
\[
 T_i=K_{v_i}\cdots K_{v_1};
\]
we let $T_0=I$ and $K=T_m$.

Let $F:\Omega_J\to\mathbb R_{\ge0}$ be a function that depends only
on the edge configuration. We claim that, for a suitable constant
$C_*=C_*(\Delta,q,B) > 0$, we have
\begin{equation}\label{eq:rr-sw-entropy-inner}
 \Ent_{\pi_J}(T_iF)-\Ent_{\pi_J}(L_{v_i}T_iF)
 \le C_*\bigl[
       \Ent_{\pi_J}(T_{i-1}F)-\Ent_{\pi_J}(T_iF)\bigr].
\end{equation}
To prove the inequality, let $\hat K_v$ be the operator on
$\Omega_J$ that resamples only the spin at $v$, conditional on
all other coordinates. Fix a configuration $\zeta$ of the starts in $J\setminus\{v\}$. Regarding $K_v$ and $\hat K_v$ as operators on $\Omega_v$ with all other stars fixed at $\zeta$, let $
 A_v=\frac12(K_v+\hat K_v)$.
Since $A_v$ is an irreducible Markov chain on the star centered at $v$ with finitely many boundary conditions, there exist constants
$\lambda, b>0$, depending only on $\Delta,q,B$, such that
$\gap(A_v)\ge\lambda$ and $\min_{x\in\Omega_v}\pi_v(x)\ge b$
uniformly over $v$ and the exterior conditioning.
Because $K_v$ and $\hat K_v$ are heat-bath
operators, for every function $g:\Omega_v\to\mathbb R$, we have
\begin{align*}
 \lambda\Var_{\pi_v}(g)
 &\leq\frac12\left[
   \E_{\pi_v}\left[\Var_{\pi_v(\cdot\mid \sigma_v)}(g)\right]
   +\E_{\pi_v}\left[\Var_{\pi_v(\cdot\mid \omega_v)}(g)\right]
 \right],
\end{align*}
where $\sigma_v$ and $\omega_v$ denote the spin and edge configuration on $v$'s star.
For any function $\ell:\Omega_J \rightarrow \mathbb R_{\ge0}$, let
$\ell_\zeta=\ell(\cdot,\zeta)$. Applying the preceding inequality to $g=\sqrt{\ell_\zeta}$, 
and using the standard comparison between entropy and variance, gives
\begin{align*}
 \Ent_{\pi_v}(\ell_\zeta)
 &\le \frac4b\Var_{\pi_v}(\sqrt{\ell_\zeta})\\
 &\le \frac{2}{b\lambda}\left[
   \E_{\pi_v} \left[
     \Var_{\pi_v(\cdot\mid \sigma_v)}(\sqrt{\ell_\zeta})\right]
   +\E_{\pi_v}\!\left[
     \Var_{\pi_v(\cdot\mid \omega_v)}(\sqrt{\ell_\zeta})\right]
 \right]\\
 &\le \frac{2}{b\lambda}\left[
   \E_{\pi_v}\!\left[
     \Ent_{\pi_v(\cdot\mid \sigma_v)}(\ell_\zeta)\right]
   +\E_{\pi_v}\!\left[
     \Ent_{\pi_v(\cdot\mid \omega_v)}(\ell_\zeta)\right]
 \right].
\end{align*}
The law of total entropy gives
\begin{align*}
 \E_{\pi_v}\!\left[
   \Ent_{\pi_v(\cdot\mid \sigma_v)}(\ell_\zeta)\right]
 &=\Ent_{\pi_v}(\ell_\zeta)
   -\Ent_{\pi_v}\!\left(\E_{\pi_v}[\ell_\zeta\mid \sigma_v]\right)
 =\Ent_{\pi_v}(\ell_\zeta)-\Ent_{\pi_v}(K_v\ell_\zeta),\\
 \E_{\pi_v}\!\left[
   \Ent_{\pi_v(\cdot\mid \omega_v)}(\ell_\zeta)\right]
 &=\Ent_{\pi_v}(\ell_\zeta)
   -\Ent_{\pi_v}\!\left(\E_{\pi_v}[\ell_\zeta\mid \omega_v]\right)
 =\Ent_{\pi_v}(\ell_\zeta)-\Ent_{\pi_v}(\hat K_v\ell_\zeta).
\end{align*}
Consequently, choosing $c\ge4/(b\lambda)$ with $c>2$, we obtain
\[
 \Ent_{\pi_v}(\ell_\zeta)
 \le\frac c2\left[
   2\Ent_{\pi_v}(\ell_\zeta)
   -\Ent_{\pi_v}(K_v\ell_\zeta)
   -\Ent_{\pi_v}(\hat K_v\ell_\zeta)
 \right].
\]
For either $H=K_v$ or $H=\hat K_v$, the tower property gives
\[
 \E_{\pi_v}[H\ell_\zeta]=\E_{\pi_v}[\ell_\zeta].
\]
Thus, when applying the law of total entropy to $\ell$ and $H\ell$,
the entropy terms involving these conditional means cancel, yielding
\[
 \E_{\zeta\sim\pi_{J\setminus\{v\}}}\!\left[
   \Ent_{\pi_v}(\ell_\zeta)-\Ent_{\pi_v}(H\ell_\zeta)
 \right]
 =\Ent_{\pi_J}(\ell)-\Ent_{\pi_J}(H\ell).
\]
Similarly, since $(L_v\ell)(\,\cdot\,,\zeta)
=\E_{\pi_v}[\ell_\zeta]$, the law of total entropy gives
\[
 \E_{\zeta\sim\pi_{J\setminus\{v\}}}
   [\Ent_{\pi_v}(\ell_\zeta)]
 =\Ent_{\pi_J}(\ell)-\Ent_{\pi_J}(L_v\ell).
\]
Averaging the local inequality over $\zeta$ therefore gives,
for every $\ell:\Omega_J\to\mathbb R_{\ge0}$,
\begin{equation}\label{eq:rr-sw-star-entropy}
 \Ent_{\pi_J}(\ell)-\Ent_{\pi_J}(L_v\ell)
 \le\frac c2\bigl[
       2\Ent_{\pi_J}(\ell)-\Ent_{\pi_J}(K_v\ell)
                         -\Ent_{\pi_J}(\hat K_v\ell)\bigr].
\end{equation}

The function $g=T_{i-1}F:\Omega_J\to\mathbb R_{\ge0}$ is independent
of the spin at $v_i$. Thus $\hat K_{v_i}g=g$.
Applying this inequality to vertex $v_i$ and the function $g$,
we obtain
\[
 \Ent_{\pi_J}(g)-\Ent_{\pi_J}(L_{v_i}g)
 \le\frac c2[\Ent_{\pi_J}(g)-\Ent_{\pi_J}(K_{v_i}g)].
\]
Since $K_{v_i}L_{v_i}=L_{v_i}K_{v_i}=L_{v_i}$, we obtain
\begin{align*}
 \Ent_{\pi_J}(T_iF)-\Ent_{\pi_J}(L_{v_i}T_iF)
 &=\Ent_{\pi_J}(K_{v_i}g)-\Ent_{\pi_J}(L_{v_i}g)\\
 &=[\Ent_{\pi_J}(g)-\Ent_{\pi_J}(L_{v_i}g)]
     -[\Ent_{\pi_J}(g)-\Ent_{\pi_J}(K_{v_i}g)]\\
 &\le\left(\frac c2-1\right)
     [\Ent_{\pi_J}(g)-\Ent_{\pi_J}(K_{v_i}g)];
\end{align*}
this proves~\eqref{eq:rr-sw-entropy-inner} with $C_*=c/2-1$.

Now, fix $i$ and let $R_i=K_{v_m}\cdots K_{v_{i+1}}$ so that
$K=R_iT_i$. The operator $R_i$ commutes with $L_{v_i}$ because
its factors act on stars other than the one centered at $v_i$.
It also preserves the product law on those stars. Using these
facts and convexity of conditional entropy, for every $u\ge0$
we obtain
\begin{align*}
 \Ent_{\pi_J}(R_iu)-\Ent_{\pi_J}(L_{v_i}R_iu)
 &=\E_{\pi_{J\setminus\{v_i\}}}
       [\Ent_{\pi_{v_i}}(R_iu)]\\
 &\le\E_{\pi_{J\setminus\{v_i\}}}
       [R_i\Ent_{\pi_{v_i}}(u)]\\
 &=\Ent_{\pi_J}(u)-\Ent_{\pi_J}(L_{v_i}u).
\end{align*}
Therefore, applying this inequality to $u=T_iF$ and combining
with~\eqref{eq:rr-sw-entropy-inner}, we deduce that
\begin{align}
 \sum_{v\in J}
   [\Ent_{\pi_J}(KF)-\Ent_{\pi_J}(L_vKF)]
 &\le\sum_{i=1}^m
   [\Ent_{\pi_J}(T_iF)-\Ent_{\pi_J}(L_{v_i}T_iF)]\notag\\
 &\le C_*\sum_{i=1}^m
       [\Ent_{\pi_J}(T_{i-1}F)-\Ent_{\pi_J}(T_iF)]\notag\\
 &=C_*[\Ent_{\pi_J}(F)-\Ent_{\pi_J}(KF)].
 \label{eq:rr-sw-entropy-ineq1}
\end{align}

Let $E_J=\bigcup_{v\in J}E_v$.
For each fixed spin configuration $\xi$ on $V\setminus J$,
and each edge configuration $\eta$ on $E\setminus E_J$, let
\[
 \phi^\xi_{E\setminus E_J}(\eta)
 =\pi_{\ES}(\omega_{E\setminus E_J}=\eta
                  \mid\sigma_{V\setminus J}=\xi),
\]
be the conditional distribution of the exterior edges.
Keeping $\xi$ fixed, define
$F_\eta,F:\Omega_J\to\mathbb R_{\ge0}$ by
\[
 F_\eta(\sigma_J,\omega_{E_J})=h(\omega_{E_J},\eta),
 \qquad
 F=\E_{\eta\sim\phi^\xi_{E\setminus E_J}}[F_\eta].
\]
Here $\pi_J$ depends on $\xi$; conditional on $\xi$, the
exterior edges are independent of the stars.
From these definitions we get the pointwise identity
\begin{equation}\label{eq:rr-sw-entropy-pointwise}
 (KF)(\sigma_J,\omega_{E_J})
 =(\mathcal C h)(\sigma_J,\xi).
\end{equation}
If a function $u(\sigma_J,\omega_{E_J})=u^\perp(\sigma_J)$ depends
only on spins, then
$(L_vu)(\sigma_J,\omega_{E_J})=(G_vu^\perp)(\sigma_J)$ with the
exterior spins fixed. Averaging~\eqref{eq:rr-sw-entropy-ineq1}
therefore gives
\begin{align*}
 \sum_{v\in J}
       [\Ent_\mu(\mathcal C h)-\Ent_\mu(G_v\mathcal C h)]
 &=\E_{\xi\sim\mu_{V\setminus J}}\left[
     \sum_{v\in J}
       [\Ent_{\pi_J}(KF)-\Ent_{\pi_J}(L_vKF)]\right]\\
 &\le C_*\E_{\xi\sim\mu_{V\setminus J}}
       [\Ent_{\pi_J}(F)-\Ent_{\pi_J}(KF)].
\end{align*}
For each fixed $\xi$, put
$m(\xi)=\E_{\pi_J}[F]=\E_{\pi_{\ES}}[h\mid\sigma_{V\setminus{ }J}=\xi]$.
The same mean is obtained for $KF$. Jensen's inequality for
$x\log x$ gives
\[
 \E_\xi\Ent_{\pi_J}(F)
 \le\Ent_\phi(h)-\Ent_{\mu_{V\setminus J}}(m).
\]
Moreover, from~\eqref{eq:rr-sw-entropy-pointwise} and the entropy
chain rule, we have
\[
 \E_\xi\Ent_{\pi_J}(KF)
 =\Ent_\mu(\mathcal C h)-\Ent_{\mu_{V\setminus J}}(m).
\]
Hence, we obtain that
\begin{align*}
 \sum_{v\in J}
       [\Ent_\mu(\mathcal C h)-\Ent_\mu(G_v\mathcal C h)]
 &\le C_*[\Ent_\phi(h)-\Ent_\mu(\mathcal C h)].
\end{align*}
Summing over the at most $\Delta+1$ independent sets gives
\eqref{eq:rr-sw-reverse} with $\kappa=(\Delta+1)C_*$.
Finally, apply this bound to $h=\mathcal C^*f$. Since
$P_{\rm SW}=\mathcal C\mathcal C^*$ and conditional Jensen's
inequality gives $\Ent_\phi(\mathcal C^*f)\le\Ent_\mu(f)$,
\[
 I_\mu((P_{\rm SW}f)\mu)
 \le\kappa\bigl[\Ent_\phi(\mathcal C^*f)-\Ent_\mu(P_{\rm SW}f)\bigr]
 \le\kappa\bigl[\Ent_\mu(f)-\Ent_\mu(P_{\rm SW}f)\bigr],
\]
as required.
\end{proof}

Lemma~\ref{lemma:rr-sw-smoothing} yields the following spectral
gap and modified log-Sobolev comparisons with the Glauber dynamics which are of independent interest.

\begin{corollary}
\label{cor:rr-sw-relaxation}
\label{cor:rr-sw-entropy-comparison}
Fix integers~$\Delta\ge2$ and $q\ge2$, and let $B>1$.
There exists a constant $K=K(\Delta,q,B)>0$ such that,
on every $n$-vertex graph of maximum degree at most~$\Delta$,
\begin{enumerate}[(i)]
    \item $\gap(P_{\rm SW})^{-1} \le \frac{K}{n}\gap(P_{\rm Gl})^{-1}$;
  \item $\rho(P_{\rm SW})^{-1}
 \le \frac{K}{n}\rho(P_{\rm Gl})^{-1}$,
where $\rho(P)$ denotes the modified log-Sobolev constant of $P$.
\end{enumerate}
\end{corollary}

\begin{proof}
Let $\kappa$ be the constant in Lemma~\ref{lemma:rr-sw-smoothing}. By the total law of entropy, Lemma~\ref{lemma:rr-sw-smoothing} can be written as
\[
 \sum_{v\in V}
 \bigl[\Ent_\mu(P_{\rm SW}f)-\Ent_\mu(G_vP_{\rm SW}f)\bigr]
 \le\kappa\bigl[\Ent_\mu(f)-\Ent_\mu(P_{\rm SW}f)\bigr]
\]
for every $f\ge0$ with $\E_\mu[f]=1$.
Fix a function $g$ with $\E_\mu[g]=0$ and set $f=1+\varepsilon g$,
which is nonnegative for sufficiently small $|\varepsilon|$.
By the Taylor expansion of $(1+x)\log(1+x) = x + \frac{x^2}{2} + o(x^2)$, we have
\[
 \Ent_\mu(1+\varepsilon g)
 =\varepsilon\E_\mu[g]
   +\frac{\varepsilon^2}{2}\E_\mu[g^2]
   +o(\varepsilon^2)
 =\frac{\varepsilon^2}{2}\Var_\mu(g)+o(\varepsilon^2).
\]
The same expansion applies with $g$ replaced by $P_{\rm SW}g$
or $G_vP_{\rm SW}g$. Dividing the preceding entropy inequality
by $\varepsilon^2/2$ and letting $\varepsilon\to0$ therefore gives
\[
 \sum_{v\in V}
 \bigl[\Var_\mu(P_{\rm SW}g)-\Var_\mu(G_vP_{\rm SW}g)\bigr]
 \le\kappa\bigl[\Var_\mu(g)-\Var_\mu(P_{\rm SW}g)\bigr].
\]

For every vertex $v \in V$, we have $G_vP_{\rm SW}g =\E_\mu[P_{\rm SW}g\mid\sigma_{-v}]$, so the
law of total variance gives
\[
 \Var_\mu(P_{\rm SW}g)-\Var_\mu(G_vP_{\rm SW}g)
 =
 \E_\mu\!\left[
   \Var_{\mu(\cdot\mid\sigma_{-v})}
       (P_{\rm SW}g)
 \right].
\]
Consequently,
\[
 \sum_{v\in V}\E_\mu\!\left[
   \Var_{\mu(\cdot\mid\sigma_{-v})}
       (P_{\rm SW}g)
 \right]
 \le\kappa\bigl[\Var_\mu(g)-\Var_\mu(P_{\rm SW}g)\bigr].
\]

We next apply the Poincar\'e inequality for the Glauber dynamics to $P_{\rm SW}g$. Since $P_{\rm Gl}=n^{-1}\sum_{v\in V}G_v$,
\[
 n\mathcal E_{P_{\rm Gl}}(P_{\rm SW}g,P_{\rm SW}g)
 =
 \sum_{v\in V}
 \bigl[
   \Var_\mu(P_{\rm SW}g)
   -\Var_\mu(G_vP_{\rm SW}g)
 \bigr]=
 \sum_{v\in V}
 \E_\mu\!\left[
   \Var_{\mu(\cdot\mid\sigma_{-v})}
   (P_{\rm SW}g)
 \right].
\]
Therefore, by the Poincar\'e inequality and the estimate above, we have
\[
n\gap(P_{\rm Gl})\Var_\mu(P_{\rm SW}g)
 \le
 \kappa\bigl[
   \Var_\mu(g)-\Var_\mu(P_{\rm SW}g)
 \bigr].
\]
Rearranging this inequality, we have
\[
 \Var_\mu(P_{\rm SW}g)
 \le
 \frac{\kappa}{\kappa+n\gap(P_{\rm Gl})}\Var_\mu(g).
\]
Finally, reversibility and positive semidefiniteness of $P_{\rm SW}$ imply
\[
 \gap(P_{\rm SW})^{-1}
 \le
 \left(1-\sqrt{
   \frac{\kappa}{\kappa+n\gap(P_{\rm Gl})}
 }\right)^{-1}
 \le 2+\frac{2\kappa}{n}\gap(P_{\rm Gl})^{-1}.
\]

For the MLSI comparison, write $\rho_{\rm Gl}=\rho(P_{\rm Gl})$. By homogeneity, it suffices to consider a function $f\ge0$ with $\E_\mu[f]=1$. We now set $h=P_{\rm SW}f$. For every pair of configurations $\sigma$ and $\tau$, $P_{\rm SW}(\sigma, \tau) > 0$. Hence $h$ is strictly positive, so the ratios and the logarithms involving $h$ below are well-defined. We first claim that
\[
 \frac{h(\sigma)}{h(\sigma')}\le R:=(qB)^\Delta
\]
whenever $\sigma$ and $\sigma'$ differ at only one vertex $v$.

Indeed, let $Q_v$ denote the Swendsen--Wang kernel with all edges
incident to~$v$ deleted. Then $Q_vf$ is independent of the initial
spin at~$v$, and
\[
 B^{-\Delta}Q_vf\le P_{\rm SW}f\le q^\Delta Q_vf.
\]
The lower bound follows by requiring all edges incident to~$v$
to be closed. For the upper bound, conditional on the remaining
bonds, adding these edges merges at most~$\Delta$ components and
hence increases any recoloring probability by at most~$q^\Delta$.
The claimed ratio bound follows.

Writing $\mu_v=\mu(\cdot\mid\sigma_{-v})$, this gives
\[
 \operatorname{Cov}_{\mu_v}(h,\log h)
 \le 2R\,\Ent_{\mu_v}(h).
\]
To verify this inequality, let $m = \min_{\tau : \tau_{-v} = \sigma_{-v}}h(\tau)$ and $M = \max_{\tau : \tau_{-v} = \sigma_{-v}}h(\tau)$. Note that $M / m \leq R$ from the argument above.

Let $\tau, \tau'$ be independent configurations from $\mu_v$. Since the map $x \mapsto \log x$ is $(m^{-1})$-Lipschitz on the interval $[m, M]$, the mean value theorem gives $|\log h(\tau) - \log h(\tau')| \leq m^{-1}|h(\tau)-h(\tau')|$, and we have
\begin{align*}
 \operatorname{Cov}_{\mu_v}(h,\log h)
 &=\frac12\E\left[
   (h(\tau)-h(\tau'))
   (\log h(\tau)-\log h(\tau'))
 \right]\\
 &\le\frac1{2m}\E\left[
   (h(\tau)-h(\tau'))^2
 \right]
 =\frac{\Var_{\mu_v}(h)}{m}.
\end{align*}

We now consider the function $\lambda(x) = x\log x$ on the interval $[m, M]$. Then, $\lambda''(x) = 1/x \geq 1/M$ on $[m, M]$. For fixed $x_0 \in [m, M]$, the convexity of $\lambda$ and Taylor's theorem gives $\lambda(x) \geq \lambda(x_0) + \lambda'(x_0)(x-x_0) + \frac{1}{2M}(x-x_0)^2$, equivalently, $\E[\lambda(x)] \geq \lambda(x_0) + \lambda'(x_0)\E[x-x_0] + \frac{1}{2M}\E[(x-x_0)^2]$. Therefore, with $x_0 = \E_{\mu_v}[h]$, 
\[
 \Ent_{\mu_v}(h)
 =\E_{\mu_v}[\lambda(h)]
   -\lambda(\E_{\mu_v}[h])
 \ge\frac{\Var_{\mu_v}(h)}{2M}.
\]

By the definitions of the MLSI constant and the Dirichlet form, we obtain
\[
\rho_{\rm Gl}\Ent_\mu(h)
\le \mathcal E_{P_{\rm Gl}}(h,\log h)=\frac1n\sum_{v\in V}
   \E_{\mu_{-v}}\!\left[
     \operatorname{Cov}_{\mu_v}(h,\log h)
   \right].
\]
Combining the two estimates above and using the law of total entropy, we obtain
\[
    \frac1n\sum_{v\in V}
   \E_{\mu_{-v}}\!\left[\operatorname{Cov}_{\mu_v}(h,\log h)\right] \leq \frac{2R}n\sum_{v\in V}
   \E_{\mu_{-v}}\!\left[\Ent_{\mu_v}(h)\right]
    =\frac{2R}n\sum_{v\in V}
   \E_\mu\!\left[\Ent_\mu(h)-\Ent_\mu(G_vh)\right] 
\]
Finally, Lemma~\ref{lemma:rr-sw-smoothing} gives
\[
\rho_{\rm Gl}\Ent_\mu(h) \leq \frac{2R}n\sum_{v\in V}\E_\mu\!\left[\Ent_\mu(h)-\Ent_\mu(G_vh)\right]  \leq \frac{2R\kappa}{n}
   \bigl[\Ent_\mu(f)-\Ent_\mu(h)\bigr].
\]
Consequently,
\[
 \Ent_\mu(P_{\rm SW}f)
 \le
 \left(
   1-\frac{n\rho_{\rm Gl}}{n\rho_{\rm Gl}+2R\kappa}
 \right)\Ent_\mu(f).
\]

The Poincar\'e quotient with a nonconstant function depending only on the spin at one vertex gives $\gap(P_{\rm Gl})\le 1/n$, thus we have $\rho_{\rm Gl}\le2\gap(P_{\rm Gl})\le2/n$. The first inequality comes from the standard comparison between the MLSI constant and the spectral gap. Taking $K=2+2R\kappa$ yields
\[
 \Ent_\mu(P_{\rm SW}f)
 \le
 \left(1-\frac{n\rho_{\rm Gl}}{K}\right)\Ent_\mu(f).
\]
Together with $\Ent_\mu(f)-\Ent_\mu(P_{\rm SW}f) \le \mathcal E_{P_{\rm SW}}(f,\log f)$, we finally obtain
\[\mathcal E_{P_{\rm SW}}(f,\log f) \geq \frac{n\rho_{\rm Gl}}{K}\Ent_\mu(f).\]
This holds for any function $f\geq 0$ with $\E_\mu[f]=1$, therefore, $\rho(P_{\rm SW})^{-1} \leq \frac{K}{n} \rho (P_{\rm Gl})^{-1}$.

Finally, $n\gap(P_{\rm Gl})\le1$ and $R\ge1$ allow the additive constant in the preceding spectral-gap estimate to be absorbed into the same~$K$.
\end{proof}

\section{Potts Glauber dynamics on $\mathbb Z^d$}
\label{var:sec:torus-extension}

In this section, we prove Theorem~\ref{mainthm:Glauber-lattice}. We now use $d \ge 2$ to be the lattice dimension, and use $\beta_c = \beta_c(q,\mathbb Z^d)$. We will prove Theorem~\ref{mainthm:Glauber-lattice} in the context of the torus $(\mathbb Z/m\mathbb Z)^d$, then in Remark~\ref{rem:free-bc} explain the (minor geometric) changes needed for the exploration process along the boundary if one were instead considering Potts Glauber dynamics on the box $\{1,\ldots,m\}^d \subset \mathbb Z^d$ with free boundary. 

For fixed $\beta<\beta_c$, we use $\mu$ for the Potts distribution on $(\mathbb Z/m\mathbb Z)^d$ and $\mathcal E_{\rm Gl}$ for the Glauber dynamics' Dirichlet form, as before. We set $B=e^\beta$, $B_c=e^{\beta_c}$, and choose fixed
\[
 0<r<\frac{B_c-B}{B_c+q-1},\qquad
 \tilde B=1+\frac{B-1}{r},\qquad
 B_+=\frac{B+(q-1)r}{1-r}<B_c.
\]

The proof will follow the same skeleton as Section~\ref{sec:spectral-gap-rrg}, with the same chains $M$ and
$\widetilde M$ of Definitions~\ref{def:M-SW-chain} and \ref{def:M-tilde-chain} with sprinkling parameter $r$, the same aggregate local entropy $I_\mu(\nu)$ of Definition~\ref{def:local-relent}, and the
same large-component vertex count $W_L$~\eqref{eq:W-L}. The role of $B_+$ will be that the sprinkled configuration $H\sim \mathbb P_{\mu,M}$ is easily seen to be stochastically below a parameter $B_+$ random-cluster model, and therefore satisfy exponential decay on its origin-to-boundary connection probabilities in boxes with all-wired boundary conditions. This ``sharpness" of the FK phase transition was established in $\mathbb Z^d$ in~\cite{duminil2019sharp}. This in particular means that if $F_N$ is a box of side-length $N$ centered at $0$ in $\mathbb Z^d$, and $\mu_{F_N}^\tau$ is the Potts measure at parameter $B$ with boundary condition $\tau$, and $\phi_{F_N,B_+}^1$ is the random-cluster distribution on $F_N$ with wired boundary, then since $B<B_c$, there is a constant $c(q,B,d)$ such that 
\begin{align}\label{eq:exp-decay-high-temp-FK}
   \mathbb P_{\mu_F^\tau,M} (0 \leftrightarrow \partial F_N ) \le  \phi_{B_+}^1(0\leftrightarrow \partial F_N) \le \exp( - c N)\,.
\end{align}

Going through the proof skeleton of Section~\ref{sec:spectral-gap-rrg}, we observe that Lemma~\ref{lem: gap of M} holds on any graph. 

Lemma~\ref{lemma:dirichlet:stronger} relied on the fact that the cut-width of all connected subgraphs of a random $\Delta$-regular graph of size at most $L$ had cut-width at most $\Delta \log L + O(1)$. The following is the lattice replacement for that ingredient of simulating the recoloring steps on small components. 

\begin{lemma}\label{lemma:dirichlet:stronger-lattice}
Fix $d\ge 2$, $q\ge2$, $\tilde B>1$. There exists $C$ such that if the cutoff for $\widetilde M$ is at $L \le m/2$, and the chains $P_{\rm Gl},P_{\widetilde M}$ are for Potts on  either $(\mathbb Z/m\mathbb Z)^d$ or side-length $m$ boxes of $\mathbb Z^d$, then 
\begin{equation*}
 \cE_{\rm Gl}(f)\ge\frac{1}{Cn e^{ C L^{(d-1)/d}}}\cE_{\widetilde M}(f)\,.
\end{equation*}
\end{lemma}

\begin{proof}
    If we establish that for any connected subgraph $F$ of $\mathbb Z^d$, one has cutwidth 
    \begin{align*}
        {\rm cw}(F) \le C_d |F|^{(d-1)/d}
    \end{align*}
    for a constant $C_d$, the remainder of the proof will be identical to that of Lemma~\ref{lemma:dirichlet:stronger}. (The assumption $|F|\le m/2$ ensures that if it is a subgraph of $(\mathbb Z/m\mathbb Z)^d$ it can be embedded in $\mathbb Z^d$.) This cutwidth bound is a standard consequence of isoperimetry in $\mathbb Z^d$: For instance, this follows from recursive application of Theorem 2.2.1 of~\cite{SeparatorsIsoperimetry}. 
\end{proof}

Lemma~\ref{lem:function-information} comparing the aggregate local entropy to the Dirichlet form also held for all graphs and parameter values. That brings us to the lattice version of the shattering-escape dichotomy.

\subsection{The shattering-escape dichotomy on the lattice}
The following lemma is the replacement for Lemma~\ref{lem:entropy-mass}. In the proof of Lemma~\ref{lem:entropy-mass}, the treelike structure of the random regular graph was used significantly; we describe how the proof changes compared to Section~\ref{sec:key-stationary-estimate} on the lattice. 

\begin{lemma}\label{var:lem:torus-mass}
Fix $d\ge 2$, $q\ge 2$, and $\beta<\beta_c$. 
For fixed $C_0$ sufficiently large and $L= C_0\log n $,
every spin distribution $\nu$ on configurations on $(\mathbb Z/m\mathbb Z)^d$ satisfies
\begin{equation}\label{var:eq:torus-mass-general}
 \mathbb E_{\nu,M}[W_L(H)]\le C I_\mu(\nu)+m^{-5 d}.
\end{equation}
\end{lemma}

\begin{proof}
Write $V=(\mathbb Z/m\mathbb Z)^d$ and $n=m^d$. Let $R$ be a sufficiently large constant to be chosen later, and for the cleanest exposition suppose $R$ divides $m$ (if not, the box sizes  in the partition will be in $[R,2R]$ instead of exactly $R$ and the argument will otherwise be mutatis mutandis). Partition $(\mathbb Z/m\mathbb Z)^d$ into $m/R$ many boxes $(B_x)_{x\in\mathcal X}$ of side-length $R$, indexed by $\mathcal X= R (\mathbb Z/(m/R)\mathbb Z)^d \subset (\mathbb Z/m\mathbb Z)^d$. Give $\mathcal X$ its nearest-neighbor graph structure.

For each box $B_x$, let
\begin{equation*}
F_x=\bigcup_{\substack{y\in\mathcal X:  \|x-y\|_\infty \le R}}B_y,
\qquad
\partial_{\rm in}F_x
=\{v\in F_x:\ v\sim w\text{ for some }w\notin F_x\}\,.
\end{equation*}
Thus $F_x$ is the union of the $3^d$ boxes given by $B_x$ and all surrounding $B_y$. Every vertex of $B_x$ is at graph distance at least $R$ from $\partial_{\rm in}F_x$. Note that  $|B_x|\le (2R)^d$, $|F_x|\le (6R)^d$, and each vertex of $V$ belongs to at most $3^d$ of the sets $F_x$.

The analogue of the exploration process Definition~\ref{def:edge-exploration-process} is the following: Fix an arbitrary ordering of the vertices and edges. To explore the component $\mathcal C_v$ of a vertex $v$ in $H$, initially declare $v$ discovered.
\begin{enumerate}
\item At each step, record the indicator $\mathbf 1\{e\in H\}$ for the first untested edge incident to an already discovered vertex. 
\item If $e\in H$, declare its other endpoint discovered as well. Continue until every edge incident to the discovered set has been tested; this set is then exactly $\mathcal C_v$. 
\end{enumerate}
Once the component $\mathcal C_v$ of a vertex $v$ is explored, repeat from the least remaining vertex until all components have been explored.

The record of this process contains only the tested edge indicators and determines $H$. Write $P$ for its law when $\sigma\sim\nu$ and $H\sim \mathbb P_{\sigma,M}$. 
We next define the comparison law (which is the analogue of the revealing $Q$ from Definition~\ref{def:Q-revealing-distribution}). We say $F_x,F_y$ are adjacent if $\|x-y\|_\infty \le 3R$ (equivalently $F_x,F_y$ are at $(\mathbb Z/m\mathbb Z)^d$-distance at most one. The $F$-graph thus has maximum degree at most $7^{d}-1$ and admits a proper coloring with $k=7^d$ colors. We use an arbitrary such proper coloring $\rm col$ and for each color $j\in [7^d]$, let
\begin{equation*}
U_j=\bigcup_{x:\,\operatorname{col}(x)=j}F_x\,.
\end{equation*}
The $F_x\in U_j$ are at $(\mathbb Z/m\mathbb Z)^d$-distance at least two from one another. 

To reveal components under law $Q_j$, first sample
$\sigma_{V\setminus U_j}=\tau$ from its marginal under $\nu$, then sample $\sigma_{U_j} \sim \mu(\cdot \mid \sigma_{V\setminus U_j} = \tau)$, and finally draw $H$ using $\mathbb P_{\sigma, M}$. Conditional on the exterior spins $\sigma_{V\setminus U_j}=\tau$, the spin distributions in the regions $(F_x)_{x: {\rm col}(x) = j}$ are independent Potts distributions on boxes with fixed boundary conditions $\tau$. By Lemma~\ref{lem:local-kl}, the relative entropy between the two revealing procedures is bounded by the aggregate local entropy: namely, by the steps between~\eqref{eq:local-record-cost-a}--\eqref{eq:record-cost-reduction}, the KL-divergence between the edge-revealing procedures are at most the KL-divergence between their spin distributions, and then Lemma~\ref{lem:local-kl} can be applied to give  
\begin{align}
\operatorname{KL}(P\,\Vert\, Q_j)
&\le
\E_{\nu_{V\setminus U_j}}
\Big[\operatorname{KL}\!\left(
 \nu(\sigma_{U_j}\in\cdot\mid\sigma_{V\setminus U_j})
 \,\|\,
 \mu(\sigma_{U_j}\in\cdot\mid\sigma_{V\setminus U_j})
\right)\Big]\nonumber\\
&\le
\sum_{x:\,\operatorname{col}(x)=j}
\E_\nu\Big[\operatorname{KL}\!\left(
 \nu(\sigma_{F_x}\in\cdot\mid\sigma_{V\setminus F_x})
 \,\|\,
 \mu(\sigma_{F_x}\in\cdot\mid\sigma_{V\setminus F_x})
\right)\Big]\nonumber\\
&\le C_R\sum_{v\in U_j}I_v(\nu)\,,
\label{var:eq:reference-kl}
\end{align}
where $I_v(\nu)$ is the summand at $v$ in $I_\mu(\nu)$, and one may take $C_R=4q^{(2R)^d}B^{4d (6R)^d}$. 

Now mark a box index $x$ \emph{open} if $B_x$ is connected to $\partial_{\rm in}F_x$ through the edges of $H$ in $F_x$. By the high-temperature implication~\eqref{eq:exp-decay-high-temp-FK},
\begin{equation}\label{eq:block-is-open-probability}
\mathbb P_{\mu(\,\cdot\mid\sigma_{V\setminus F_x}=\tau),M}
   (x\text{ is open})
\le\rho_R,\qquad
\rho_R\le e^{-cR}\,.
\end{equation}
Consequently, conditional on $\sigma_{V\setminus U_j}$, the color-$j$ marks under $Q_j$ are independent and each has probability at most $\rho_R$ of being open.

Set $\ell=  L/(2R)^d$. If an $H$-component has more than $L$ vertices, then necessarily, the set of boxes $(F_x)_{x\in \mathcal X}$ it visits  form a $\mathcal X$-connected set with at least $\ell$ vertices. Thus every vertex $v\in V$ counted by $W_L$ lies in a box $F_x$ belonging to a $\mathcal X$-connected component of size at least $\ell$ of open boxes.

For each color $j\in [7^d]$, call a $\mathcal X$-connected set $S\subseteq\mathcal X$
$j$-dense if at least $|S|/k$ of its indices are open boxes
of color $j$. Let $Y_j$ be the maximum total size of a vertex-disjoint
family of $j$-dense connected sets, each of size at least $\ell$; i.e., if $\mathcal F$ is a family of vertex disjoint connected subsets $(S_i)_i \subset \mathcal X$ such that for all $i\le |\mathcal F|$,  $S_i$ is $j$-dense and $|S_i|\ge \ell$, then 
\begin{align*}
    Y_j := \max_{\mathcal F} \sum_{i} |S_i|\,.
\end{align*}
Each connected component of open boxes is
$j$-dense for the color $j$ occurring most frequently in it, so by assigning it to that color, we see that 
\begin{equation}\label{var:eq:cover}
W_L(H)=\sum_{v\in V}\one\{|\mathcal C_v|>L\} 
\le (2R)^d\sum_{j=1}^{7^d} Y_j\,.
\end{equation}
By construction, each $Y_j$ depends only on the open-status of blocks in $U_j$ and these are independent.  

Namely, even conditional on $V\setminus U_j$, under $Q_j$, by~\eqref{eq:block-is-open-probability}, a fixed subset $S\subset \{x:{\rm col}(x)=j\}$ is $j$-dense with probability at most $2^{|S|} \rho_R^{|S|/7^d}$. 
We therefore get 
\begin{equation*}
\E_{Q_j}\!\left[e^{Y_j}\mid\sigma_{V\setminus U_j}\right]
\le \sum_{\mathcal F} \prod_{S\in \mathcal F}  \E_{Q_j} [ e^{|S|} \mathbf 1_{S \text{ is $j$-dense}} \mid \sigma_{V\setminus U_j} ]\le 
\prod_{\substack{S\subseteq\mathcal X\text{ connected}\\|S|\ge\ell}}
 \left(1+u^{|S|}\right),
\qquad
u=2e\rho_R^{1/7^d}\,.
\end{equation*}
There are at most $n (2d)^{s}$ possible $\mathcal X$-connected sets of size $s$. By choosing $R$ sufficiently large (depending on $c$ in $\rho_R$, as well as on $d$), we get a bound on the right-hand side as 
\begin{equation*}
\log\E_{Q_j}[e^{Y_j}]   
\le \sum_{s\ge\ell} n (2d)^{s} u^s
\le n 2^{-\ell}\,.
\end{equation*}
Now using~\eqref{eq:entropy-transfer} and~\eqref{var:eq:reference-kl} we get 
\begin{equation*}
\E_P[Y_j]
\le \operatorname{KL}(P\, \Vert\,  Q_j)+\log\E_{Q_j}[e^{Y_j}]
\le C_R\sum_{v\in U_j}I_v(\nu)+ n2^{-\ell}\,.
\end{equation*}
Plugging this into the $\E_P$-expectation of~\eqref{var:eq:cover}, and plugging in for $\ell$,
\begin{align*}
\E_{\nu,M}[W_L(H)]
&\le 3^d(2R)^d C_R I_\mu(\nu)
     +2k (2R)^d n2^{- L/(2R)^d} \le C I_\mu(\nu)+Cn e^{-cL}\,.
\end{align*}
For $L= C_0\log n$, choosing $C_0$ sufficiently large makes the remainder at most $n^{-5}=m^{-5d}$. 
\end{proof}

\begin{remark}\label{rem:free-bc}
Lemma~\ref{var:lem:torus-mass} holds with essentially the same proof when the underlying graph is the box $\Lambda_m=\{1,\ldots,m\}^d$ with free boundary conditions. Near the boundary of $\Lambda_m$, intersect each region $F_x$ with $\Lambda_m$ and define its inner boundary relative to $\Lambda_m$:
\[
\partial_{\rm in}F_x
=\{v\in F_x:\ v\sim w
  \text{ for some }w\in\Lambda_m\setminus F_x\}.
\]
The only modification is then in the crossing probability for these truncated regions. There are no pinned spins beyond the boundary of $\Lambda_m$. By wired domination and monotonicity of the random-cluster model, the probability that $B_x$ connects to $\partial_{\rm in}F_x$ is bounded above by the corresponding crossing probability in the full distance-$R$ enlargement of $B_x$ in $\mathbb Z^d$, with wired boundary conditions on its entire boundary. Applying~\eqref{eq:exp-decay-high-temp-FK} therefore still gives the same upper bound $\rho_R$ on the probability of marking these boxes open. The remainder of the proof is unchanged.
\end{remark}

\begin{proof}[\textbf{\emph{Proof of Theorem~\ref{mainthm:Glauber-lattice}}}]
Following the steps in the proof of Theorem~\ref{mainthm:Glauber:gap} in Section~\ref{sec:spectral-gap-rrg}, the steps between~\eqref{eq: var indicator} and~\eqref{eq: var M vs tilda M} were graph and parameter independent. Thus we obtain 
by Markov's inequality, and Lemma~\ref{var:lem:torus-mass} (with $C_1= C$) in the torus case and Remark~\ref{rem:free-bc} in the case of an $m\times \cdots \times m$ box, 
\begin{align*}
    \cE_M(f)-\cE_{\widetilde M}(f)
  \le  \E_\mu[f^2]( C_2 I_\mu(\nu_f)+n^{-5} )\,.
\end{align*}
By Lemma~\ref{lem:function-information} (with $C_2 = C$), this yields 
\begin{equation}
\label{eq:lattice-M-truncation-loss}
 \cE_M(f)-\cE_{\widetilde M}(f)
 \le C_1 C_2n\cE_{\rm Gl}(f)+n^{-5}\E_\mu[f^2]\,.
\end{equation}
Combining this with the bounds from Lemma~\ref{lemma:dirichlet:stronger-lattice}, the graph independent comparison of Lemma~\ref{lem: gap of M}, and the spectral gap bound of~\cite{gheissari2024spatial} for the FK dynamics (Corollary 1.1(i) for the torus case and Corollary 1.2(i) in the box case), 
we obtain for some $a(q,d,\beta)>0$ that 
\begin{equation}
     a\Var_\mu(f) \le \mathcal E_{M}(f) 
 \le (Ce^{CL^{(d-1)/d}} + C_1 C_2) n\cE_{\rm Gl}(f)+n^{-5}\E_\mu[f^2]\,. \label{eq: var bound almost done lattice}
\end{equation}
To remove the last term, apply \eqref{eq: var bound almost done lattice} with $g=f-\E_{\mu}[f]$ and subtract the last term from both sides: 
\[
 \bigl(a-n^{-5}\bigr)\Var_\mu(f)\le (C e^{CL^{(d-1)/d}} + C_1 C_2) n\cE_{\rm Gl}(f).
\]
For $n$ large enough that $n^{-5}\le a/2$, and with $L= C_0\log n$,
this gives the claimed Poincar\'e inequality, i.e.\ $\gap(P_{\rm Gl})^{-1}= n e^{O((\log n)^{(d-1)/d})}$, and the mixing time bound with the extra factor of $n$.
\end{proof}

\bibliographystyle{alpha}
\bibliography{references}

\end{document}